\documentclass[VANCOUVER,STIX1COL]{WileyNJD-v2}

\articletype{Article Type}%

\received{<day> <Month>, <year>}
\revised{<day> <Month>, <year>}
\accepted{<day> <Month>, <year>}

\usepackage{amsopn}
\usepackage{xcolor}
\usepackage{bm} 
\usepackage{comment} 
\usepackage{mathtools} 
\mathtoolsset{centercolon} 

\usepackage{makecell} 
\usepackage{multirow} 
\usepackage{makecell} 
\usepackage{pifont}      

\usepackage{boldline}  

\usepackage{grffile} 

\usepackage{cleveref}
\crefname{section}{Section}{Sections}
\crefname{appendix}{Appendix}{Appendices}
\crefname{theorem}{Theorem}{Theorems}
\crefname{lemma}{Lemma}{Lemmas}
\crefname{corollary}{Corollary}{Corollaries}
\crefname{definition}{Definition}{Definitions}
\crefname{figure}{Figure}{Figures}

\Crefname{subsection}{Section}{Sections}
\crefname{subsection}{Section}{Sections}

\crefname{assumption}{Assumption}{Assumptions}
\Crefname{assumption}{Assumption}{Assumptions}

\newcommand{\wh}[1]{\widehat{#1}}
\newcommand{\wt}[1]{\widetilde{#1}}
\newcommand{\im}{\mathrm{i}}
\DeclareMathOperator{\diag}{diag} 
\DeclareMathOperator*{\argmax}{arg\,max}
\DeclareMathOperator*{\argmin}{arg\,min}
\renewcommand{\d}[0]{\ensuremath{\operatorname{d}\!}} 

\begin{document}

\title{Multigrid reduction-in-time convergence for advection problems: A Fourier analysis perspective\protect\thanks{A published version of this  preprint is available at \url{https://doi.org/10.1002/nla.2593}.}\protect\thanks{This work was supported in part by NSERC of Canada. The work of the third author was partially supported by an Australian Government Research Training Program Scholarship.}}

\author[1]{H. De Sterck}

\author[2]{S. Friedhoff}

\author[3]{O. A. Krzysik*}

\author[4]{S. P. MacLachlan}

\authormark{De Sterck \textsc{et al}}

\address[1,3]{\orgdiv{Department of Applied Mathematics}, \orgname{University of Waterloo}, \orgaddress{\state{ON}, \country{Canda}}}

\address[2]{\orgdiv{Department of Mathematics}, \orgname{Bergische Universit\"at Wuppertal}, \orgaddress{\country{Germany}}}

\address[4]{\orgdiv{Department of Mathematics and Statistics}, \orgname{Memorial University of Newfoundland}, \orgaddress{\state{NL}, \country{Canada}}}

\corres{\email{okrzysik@uwaterloo.ca}}


\abstract[Abstract]{%
	A long-standing issue in the parallel-in-time community is the poor convergence of standard iterative parallel-in-time methods for hyperbolic partial differential equations (PDEs), and for advection-dominated PDEs more broadly.
	Here, a local Fourier analysis (LFA) convergence theory is derived for the two-level variant of the iterative parallel-in-time method of multigrid reduction-in-time (MGRIT).
	This closed-form theory allows for new insights into the poor convergence of MGRIT for advection-dominated PDEs when using the standard approach of rediscretizing the fine-grid problem on the coarse grid.
	Specifically, we show that this poor convergence arises, at least in part, from inadequate coarse-grid correction of certain smooth Fourier modes known as characteristic components, which was previously identified as causing poor convergence of classical spatial multigrid on steady-state advection-dominated PDEs.	
	We apply this convergence theory to show that, for certain semi-Lagrangian discretizations of advection problems,  MGRIT convergence using rediscretized coarse-grid operators cannot be robust with respect to CFL number or coarsening factor.
	A consequence of this analysis is that techniques developed for improving convergence in the spatial multigrid context can be re-purposed in the MGRIT context to develop more robust parallel-in-time solvers. 
	This strategy has been used in recent work to great effect; here, we provide further theoretical evidence supporting the effectiveness of this approach.
	}

\keywords{Parallel-in-time methods, MGRIT, Parareal, advection equations, local Fourier analysis}

\maketitle

\section{Introduction}

Parallel-in-time methods for the solution of initial-value, time-dependent partial differential equations (PDEs) have seen an increase in both popularity and relevance over the past two decades, with the advancement of massively parallel computers.
Thorough reviews of the parallel-in-time field are given by Gander\cite{Gander_2015} and, Ong \& Schroder.\cite{Ong_Schroder_2020}
In this paper, we study the convergence of the multigrid reduction-in-time (MGRIT) algorithm,\cite{Falgout_etal_2014} which, as its name implies, achieves temporal parallelism by deploying multigrid techniques in the temporal direction.
Our work also applies to the well-known Parareal method,\cite{Lions_etal_2001} since the two-grid variant of MGRIT is the same as Parareal for certain choices of algorithmic parameters.

A long-standing issue in the parallel-in-time field is the development of efficient solvers for hyperbolic PDEs, and for advection-dominated PDEs more broadly.
It is well-documented that the convergence of MGRIT and Parareal is typically poor for advection-dominated problems.\cite{Chen_etal_2014, 
Dai_Maday_2013,
DeSterck_etal_2021, 
Dobrev_etal_2017, 
Gander_Vandewalle_2007, 
Gander_2008, 
Gander_Lunet_2020,
Hessenthaler_etal_2018, 
Howse_etal_2019, 
Howse2017_thesis,
KrzysikThesis2021, 
Nielsen_etal_2018, 
Ruprecht_Krause_2012, 
Ruprecht_2018, 
Schmitt_etal_2018, 
Schroder_2018,
Steiner_etal_2015}
In fact, slow and non-robust convergence of these methods for advection-dominated problems has, for the most part, precluded any significant parallel speed-ups over sequential time-stepping. 
This is in stark contrast to diffusion-dominated PDEs, for which convergence is typically fast, and is typically robust with respect to PDE, discretization, and algorithmic parameters, resulting in substantial speed-ups over time-stepping.

While our focus in this paper is on multigrid-in-time methods, we note that other parallel-in-time strategies have been developed for advective and wave-propagation problems that do not appear to suffer from the same poor convergence as Parareal and MGRIT.\cite{McDonald_etal_2018,Gander_etal_2019,Gander_Wu_2020,Liu_Wu_2020,Liu_etal_2022} 
In particular, the referenced algorithms rely (in one form or another) on diagonalizing the discretized problem in the time direction and then solving in parallel a set of decoupled linear systems, a technique introduced by Maday and R{\o}nquist.\cite{Maday_Ronquist_2008}

A large number of convergence analyses for MGRIT and Parareal have been developed,\cite{DeSterck_etal_2019, 
Dobrev_etal_2017, 
Friedhoff_MacLachlan2015, 
Gander_2008, 
Gander_etal_2018, 
Gander_Vandewalle_2007, 
Hessenthaler_etal_2020, 
Ruprecht_2018, 
Southworth_2019},
with several paying particular attention to advection-related problems.\cite{
DeSterck_etal_2021,
DeSterck_etal_2019,
Gander_2008,
Gander_Vandewalle_2007, 
Ruprecht_2018}
Yet, there is no widely accepted explanation for what fundamentally makes the parallel-in-time solution of these problems so difficult.
In De~Sterck~et~al.,\cite{DeSterck_etal_2021} informed by the convergence theory of Dobrev~et~al.,\cite{Dobrev_etal_2017} we developed a heuristic optimization strategy for building coarse-grid operators. 
While yielding previously unattained convergence rates, our coarse-grid operators were limited to constant-wave-speed advection problems, relying on prohibitively expensive computation to determine these operators.
So, while there are many MGRIT and Parareal convergence analyses, none so far appear to have inspired the development of efficient and practical solvers for advection-dominated PDEs.

Here, we analyze the convergence of two-level MGRIT through the lens of local Fourier analysis (LFA). 
Originally proposed by Brandt,\cite{Brandt_1977} LFA is a predictive tool for studying the convergence behaviour of multigrid methods, and is the most widely used tool for doing so.\cite{Brandt_1981,
Stueben_Trottenberg_1982,
Vandewalle_Horton_1995,
Trottenberg_etal_2001,
Wienands_Joppich_2005,
Yavneh_1998}
LFA has been used previously to investigate MGRIT convergence,\cite{DeSterck_etal_2019,Friedhoff_MacLachlan2015} 
as well as other multigrid-in-time methods.\cite{Friedhoff_MacLachlan_2013,
Gander_etal_2013,
Gander_Neumuller_2016,
Vandewalle_Horton_1995}
A particular novelty of our LFA theory is that it is presented in closed form, which offers significantly more insight than previous efforts, such as in References~\citenum{DeSterck_etal_2019,Friedhoff_MacLachlan2015}. 
That is, we provide analytical expressions for the LFA predictions of the norm and spectral radius of the error propagation operator, rather than arriving at them through semi-analytical means (i.e., with some numerical assistance).
Such closed-form LFA results are well-known for some simple cases (e.g., optimizing the relaxation parameter in weighted Jacobi), but are not always attainable, particularly when applying the method to complicated PDEs or algorithms.\cite{Brown_etal_2021} 
As another example, even closed-form convergence factors for MGRIT applied to discretizations of linear advection are not currently known.

Our LFA theory yields results consistent with the prior MGRIT analyses of References~\citenum{Dobrev_etal_2017, 
Hessenthaler_etal_2020, 
Southworth_2019,
Southworth_etal_2021}, which were derived using different analysis tools.
Thus, in this sense our LFA does not yield ``new'' final convergence bounds.
However, the fact that we have closed-form convergence results in temporal Fourier space allows us to draw important new connections between the convergence issues of MGRIT and those of classical spatial multigrid methods in a way that these existing theories cannot. 
Specifically, it is widely known for spatial multigrid methods that poor coarse-grid correction of certain smooth Fourier modes known as \textit{characteristic components} leads to convergence issues for steady-state advection-dominated problems.\cite{Brandt_1981,
Brandt_Yavneh_1993,
Yavneh_1998,
Oosterlee_Washio_2000,
Trottenberg_etal_2001,
Wan_Chan_2003,
Bank_etal_2006,
Yavneh_Weinzierl_2012} 
Here, we establish that this is the case for MGRIT also, with poor convergence due, at least in part, to an inadequate coarse-grid correction of smooth space-time characteristic components, which was something initially suggested to us by the author of Reference~\citenum{Yavneh_1998}.

This new connection suggests that MGRIT solvers with increased robustness could be developed by generalizing the fixes proposed in the steady-state setting. 
In fact, in a sequence of recent papers,\cite{DeSterck_etal_2023_SL,DeSterck_etal_2023_MOL,DeSterck_etal_2023_nonlin} we have leveraged this connection to develop fast MGRIT-based solvers for hyperbolic PDEs that far outperform those based on standard rediscretization or direct discretization techniques, even including for nonlinear hyperbolic PDEs with shocks.\cite{DeSterck_etal_2023_nonlin}
Underlying each of these solvers is the use of an idea developed by Yavneh\cite{Yavneh_1998} in the steady-state case of increasing the accuracy of the coarse-grid operator relative to the fine-grid operator. 
By formally establishing this link with spatial multigrid methods, this paper provides the theoretical foundations for the advancements made in References~\citenum{DeSterck_etal_2023_SL,DeSterck_etal_2023_MOL,DeSterck_etal_2023_nonlin}.

The remainder of this paper is organised as follows.
An algorithmic description of MGRIT and key assumptions are presented in \cref{sec:prelim}.
\Cref{sec:error_prop} describes the error propagation operator of the MGRIT algorithm.
The LFA convergence theory is presented in \Cref{sec:LFA}.
A short comparison to related literature is given in \cref{sec:literature_comparison}.
In \cref{sec:char_comp}, the LFA theory is used to describe the link between convergence of MGRIT and spatial multigrid methods, and this is applied to analyze MGRIT convergence for a class of semi-Lagrangian discretizations of linear advection problems.
Concluding remarks are given in \Cref{sec:conclusions}.

\section{Preliminaries}
\label{sec:prelim}

\Cref{sec:MGRIT} provides a brief overview of the MGRIT algorithm as it applies to linear problems, and  \cref{sec:assumptions-notation} describes some key assumptions and notation used in our analysis.
%

\subsection{MGRIT}
\label{sec:MGRIT}

Consider the linear, initial-value problem (IVP) $\frac{\partial u}{\partial t} = \mathcal{L}(u) + f$ with $u(x,0) = u_0(x)$. Suppose this problem is fully discretized in space and time using a one-step method to yield the discrete equations
\begin{align} \label{eq:one-step}
\bm{u}_{n+1} = \Phi \bm{u}_n + \bm{g}_{n+1}, \quad n = 0, 1, \ldots, n_t - 2,
\end{align}
supplemented with the initial data $\bm{u}_0 \in \mathbb{R}^{n_x}$, where $\bm{u}_n$ is the spatially discrete approximation to $u(x,t_n)$.
In \eqref{eq:one-step}, $\Phi \in \mathbb{R}^{n_x \times n_x}$ is the \textit{time-stepping operator} that, in this work, we assume is linear, and is time independent.
Discretized systems as in~\eqref{eq:one-step} arise in several contexts, including the method-of-lines discretization process, for arbitrary spatial discretization of $\mathcal{L}(u)$ and one-step (but, possibly, multi-stage) discretization of the time derivative.
The vector $\bm{g}_{n+1} \in \mathbb{R}^{n_x}$ contains solution-independent information pertaining to spatial boundary conditions and source terms.
The system \eqref{eq:one-step} is naturally solved by sequential time-stepping: Computing $\bm{u}_1$ from $\bm{u}_0$, then $\bm{u}_2$ from $\bm{u}_1$, and so on.
By contrast, a parallel-in-time method seeks the solution of \eqref{eq:one-step} over all values of $n$ in parallel.
With this in mind, it is convenient to assemble the equations in \eqref{eq:one-step} into the global linear system
\begin{align} \label{eq:A0_system}
A_0 
\bm{u}
:=
\begin{bmatrix}
I & \\
-\Phi & I \\
& \ddots & \ddots \\
& & -\Phi & I
\end{bmatrix}
\begin{bmatrix}
\bm{u}_0 \\
\bm{u}_1 \\
\vdots \\
\bm{u}_{n_t-1}
\end{bmatrix}
=
\begin{bmatrix}
\bm{u}_0 \\
\bm{g}_1 \\
\vdots \\
\bm{g}_{n_t-1}
\end{bmatrix}
=:
\bm{b}.
\end{align}

MGRIT\cite{Falgout_etal_2014} iteratively computes the solution of \eqref{eq:A0_system} by employing multigrid reduction techniques in the time direction.
To this end, let us associate with \eqref{eq:one-step} a grid of $n_t$ equispaced points in time: $0 = t_0 < t_1 < \ldots < t_{n_t-1} = T$, with $t_n = n \delta t$ and time step $\delta t$.
Let these time points constitute a ``fine grid,'' and let a coarsening factor $m \in \mathbb{N}$ induce a ``coarse grid,'' consisting of every $m$th fine-grid point. 
The set of points appearing exclusively on the fine grid are called F-points, while those shared by both fine and coarse grids are C-points. 
For simplicity, we assume that $t_0$ is a C-point, and that $n_t$ is divisible by $m$.

A single MGRIT iteration on \eqref{eq:A0_system} consists of pre-relaxation, a coarse-grid correction, and then post-relaxation, which we now detail.
Suppose we have an approximation $\bm{w} \approx \bm{u}$, with $\bm{r} = \bm{b} - A_0  \bm{w}$ denoting its algebraic residual. 
Pre- and post-relaxation on $A_0 \bm{w} \approx \bm{b}$ are built by composing F- and C-relaxation sweeps, which update $\bm{w}$ to set the resulting $\bm{r}$ to zero at F- and C-points, respectively.
%
F-relaxation can be thought of as time-stepping $\bm{w}$ from each C-point across the interval of $m-1$ F-points that follow it, called the CF-interval in the following, and C-relaxation can be thought of as time-stepping $\bm{w}$ from the last F-point in each CF-interval to its neighbouring C-point in the next CF-interval. 
F- and C-relaxations are local operations and are, therefore, highly parallelizable.
In this work, post-relaxation is fixed as a single F-relaxation, while we denote pre-relaxation by $\textrm{F(CF)}^{\nu}$, $\nu \in \mathbb{N}_0$, meaning a single F-relaxation followed by $\nu$ sweeps of CF-relaxation (each composed of a sweep of C-relaxation followed by one of F-relaxation).  Common choices are $\nu=0$ or $1$, written as F- and FCF-relaxation, respectively, but larger values of $\nu$ can also be considered.\cite{Gander_etal_2018,
Southworth_2019,
Hessenthaler_etal_2020}

After pre-relaxation, the coarse-grid correction step solves the reduced system $A_1 \bm{e}_{\Delta} = \bm{r}_{\Delta}$, which has a factor of $m$ fewer time points.
Here, $\bm{r}_{\Delta}$ is the fine-grid residual vector injected to the C-points,
\begin{align} \label{eq:A1_system}
A_1 
:=
\begin{bmatrix}
I & \\
-\Psi & I \\
& \ddots & \ddots \\
& & -\Psi & I
\end{bmatrix}
\approx
\begin{bmatrix}
I & \\
-\Phi^m & I \\
& \ddots & \ddots \\
& & -\Phi^m & I
\end{bmatrix}
\end{align}
is the coarse-grid operator with $\Psi$ denoting the \textit{coarse-grid time-stepping operator} that approximates $\Phi^m$, and $\bm{e}_{\Delta}$ is an approximate error at C-points.
In the two-grid setting analyzed in this paper, this reduced system is solved exactly via sequential time-stepping on the coarse grid, but in the multilevel setting it is solved approximately by recursively applying MGRIT, noting that $A_1$ has the same block bidiagonal structure as $A_0$ in \eqref{eq:A0_system}.
Upon computing the approximate error $\bm{e}_{\Delta}$, it is added to C-point values of $\bm{w}$, and then the post-relaxation is performed. 

When $\Psi = \Phi^m$, $A_1 \bm{e}_{\Delta} = \bm{r}_{\Delta}$ is exactly the C-point Schur complement of the residual equation $A_0 \bm{e} = \bm{r}$, in which $\bm{e} = \bm{u} - \bm{w}$ is the algebraic error, and MGRIT converges exactly in a single iteration; thus, we dub $\Phi^m$ the \textit{ideal coarse-grid time-stepping operator}.
No parallel speed-up can be obtained by using $\Psi = \Phi^m$, however, since the sequential coarse-grid solve is as expensive as solving the original fine-grid problem. 
The convergence of MGRIT is naturally determined by the accuracy of the approximation $\Psi \approx \Phi^m$, where one needs to strike a balance between the complexity or the cost of applying $\Psi$, and the accuracy of the approximation of $\Phi^m$.
In the PDE context, the most common approach for computing $\Psi$ is to rediscretize the fine-grid problem using the enlarged coarse-grid time step $m \delta t$, which tends to work well for diffusion-dominated problems. For advection-dominated problems, however, this typically leads to extremely poor convergence, for which we provide theoretical insight in \cref{sec:char_comp}, based on results derived in \cref{sec:error_prop,sec:LFA}.

\subsection{Assumptions and notation}
\label{sec:assumptions-notation}

The analysis in this paper relies on the following assumptions on the time-stepping operators.
\begin{assumption}[Simultaneous diagonalizability] \label{ass:Phi_Psi}
The fine- and coarse-grid time-stepping operators $\Phi \in \mathbb{R}^{n_x \times n_x}$ and $\Psi \in \mathbb{R}^{n_x \times n_x}$ are simultaneously diagonalizable by a unitary matrix ${\cal U}$,
\begin{align}
\label{eq:Phi_U}
\Phi = {\cal U} \diag \big( \lambda_1, \ldots, \lambda_{n_x} \big) {\cal U}^*,
\quad
\Psi = {\cal U} \diag \big( \mu_1, \ldots, \mu_{n_x} \big) {\cal U}^*,
\end{align}
with $\lambda_i$ and $\mu_i$ denoting the $i$th eigenvalue of $\Phi$ and $\Psi$, respectively.
\end{assumption}
Simultaneous diagonalizability allows the convergence analysis of MGRIT to be simplified immensely, and it has appeared in previous analyses for this reason. \cite{Friedhoff_MacLachlan2015,
Dobrev_etal_2017,
Southworth_2019,
Hessenthaler_etal_2020,
DeSterck_etal_2019}
We note that this assumption still permits a large class of model problems.
For example, suppose that after spatial semi-discretization the linear PDE $\frac{\partial u}{\partial t} = \mathcal{L}(u) + f$ becomes the ODE system $\frac{\d \bm{u}}{\d t} = L \bm{u} + \bm{f}$ with $L \in \mathbb{R}^{n_x \times n_x}$ a matrix with time-independent entries, and $\bm{f}$ a solution-independent source term.
Then if (potentially different) Runge-Kutta time integrators are used on the fine and coarse grids, $\Phi$ and $\Psi$ will be rational functions of $L$. 
If the matrix $L$ is normal (i.e., unitarily diagonalizable), then $\Phi$ and $\Psi$ will be simultaneously diagonalized by the unitary eigenvectors of $L$.
Examples of normal $L$ are those resulting from, for example, standard finite-difference spatial discretizations of constant-coefficient heat equations subject to periodic or Dirichlet boundary conditions.
Similarly, $L$ corresponding to typical spectral or finite-difference spatial discretizations of constant-coefficient linear advection problems would be normal so long as spatial boundaries are periodic, since in these cases $L$ would be circulant; if the advection problem is subject to inflow/outflow spatial boundaries, the unitary diagonalizability assumption is less likely to be satisfied, since such $L$ would typically not be diagonalizable, let alone unitarily diagonalizable (e.g., first- and second-order upwind finite differences result in triangular $L$).

In addition to the above assumption, we also place an assumption on the stability of $\Phi$ and $\Psi$, as follows. 
\begin{assumption}[Stability] \label{ass:stability}
The fine- and coarse-grid time-stepping operators are $\ell^2$-stable, with $\Vert \Phi \Vert_2, \Vert  \Psi \Vert_2 < 1$.
Or equivalently, $|\lambda_i| ,|\mu_i| < 1, \, \forall i \in \{1, \ldots, n_x\}$.
\end{assumption}
\begin{remark}[Unit eigenvalues] \label{rem:constant_mode}
By the above assumption, the LFA theory we present does not apply to non-dissipative time-stepping operators, i.e., those satisfying $|\lambda_{i_*}|, |\mu_{i_*}| \geq 1$ for at least one ${i_*} \in \{1, \ldots, n_x\}$. 
One technical caveat here is that the LFA does apply to pairs of non-dissipative time-stepping operators provided that they treat non-dissipated eigenmodes $i_*$ exactly the same, i.e., $( \lambda_{i_*} )^m = \mu_{i_*} \geq 1$.
The reason is that errors in the direction of these modes is exactly eliminated in one MGRIT iteration (see Section \ref{sec:error_prop_time-only} and then \cref{app:sm:lem:const_mode}, \cref{sm:lem:const_mode}), and, so, their error propagation need not be analyzed by our LFA theory, nor does their presence influence the correctness of the LFA on the remaining eigenmodes.
For example, later we consider PDEs with periodic spatial boundary conditions resulting in $( \lambda_{i_*} )^m = \mu_{i_*} =1$ with $i_*$ corresponding to the constant eigenvector.
\end{remark}

For our analysis, it is also necessary to introduce the lower shift matrix $L_n$, and the Vandermonde-style function $v \colon \mathbb{C}^{n \times n} \to \mathbb{C}^{mn \times n}$ defined by
\begin{align} \label{eq:Ln_v_defs}
L_n := 
\begin{bmatrix}
0 \\
1 & 0 \\
& \ddots & \ddots \\
&  & 1 & 0
\end{bmatrix}
\in \mathbb{R}^{n \times n},
\quad
v(X) 
:=
\begin{bmatrix}
I \\
X \\
\vdots \\
X^{m-1}
\end{bmatrix}.
\end{align}
%

\section{Error propagation}
\label{sec:error_prop}

Suppose $\bm{u}^{(p)}$ is the approximation to the solution of $A_0\bm{u}=\bm{b}$ after $p$ MGRIT iterations, and that $\bm{e}^{(p)} = \bm{u} - \bm{u}^{(p)}$ is the algebraic error in that approximation. Then, the error propagator ${\cal E}$ transforms the initial error $\bm{e}^{(0)}$ as
$
\bm{e}^{(p)} = {\cal E}^p \bm{e}^{(0)}.
$
Thus, the spectral properties of ${\cal E}^p$ characterize the convergence of the iteration. 
For example, the norm and spectral radius of ${\cal E}$ describe the short-term and asymptotic convergence behaviors of the method, respectively.
%
%
The error propagator for a two-grid multigrid method takes the form (see Trottenberg~et~al.,\cite{Trottenberg_etal_2001} Section 2.2.3)
\begin{align} \label{eq:E_standard}
{\cal E} = S_{\rm post} {\cal K} S_{\rm pre}, 
\quad \textrm{where }
{\cal K}
=
I - P A^{-1}_1 R A_0,
\end{align}
where $A_0, A_1$ represent the fine- and coarse-grid matrices, and $S_{\rm pre}$ and $S_{\rm post}$ denote error propagators for pre- and post relaxation, respectively.
The error propagator ${\cal K}$ is that of the coarse-grid correction, in which $R$ and $P$ denote the restriction and interpolation operators, respectively. 
We note that there are several equivalent formulations of the MGRIT error propagator.
We follow References~\citenum{DeSterck_etal_2019}, Section 4.1.3, \& \citenum{Friedhoff_MacLachlan2015}, by considering the propagator on the fine grid, although similar results can be obtained by considering only coarse-grid degrees-of-freedom (DOFs), as in References~\citenum{Dobrev_etal_2017,
Southworth_2019,
DeSterck_etal_2019,
Hessenthaler_etal_2020}.
Similarly, we take $P$ in the above to be injection interpolation, while other works use so-called ideal interpolation, in which the interpolation $P$ also incorporates the post-relaxation from the MGRIT iteration.\cite{Falgout_etal_2014,
Friedhoff_MacLachlan2015,
Gander_etal_2018,
Hessenthaler_etal_2020,
Howse_etal_2019,
Southworth_2019}

As noted above, and as described in~\cref{sec:MGRIT}, in this work we fix the interpolation and restriction operators in \eqref{eq:E_standard} as injection, writing $P$ for injection from the coarse-grid points into the fine grid, and $P^\top$ for its adjoint. We also fix the post-relaxation to be $S^{\rm F}$, with pre-relaxation as $\big( S^{\rm CF} \big)^{\nu} S^{\rm F}$ for some $\nu \in \mathbb{N}_0$, yielding
\begin{align} \label{eq:E_mgrit}
{\cal E} = S^{\rm F} {\cal K} \big( S^{\rm CF} \big)^{\nu} S^{\rm F}, 
\quad \textrm{where }
{\cal K}
=
I - P A^{-1}_1 P^\top A_0.
\end{align}
Expressions for $P$, $S^{\rm F}$, and $S^{\rm CF}$, are derived, for example, in References~\citenum{DeSterck_etal_2019,
Dobrev_etal_2017,
Friedhoff_MacLachlan2015,
Hessenthaler_etal_2020,
Southworth_2019}, so we omit them here, but include their derivations in \Cref{sm:sec:interp_and_relax}.

\subsection{Temporal MGRIT error propagation}
\label{sec:error_prop_time-only}

Since $\Phi$ and $\Psi$ are simultaneously unitarily diagonalizable (see \cref{ass:Phi_Psi}), their eigenvectors ${\cal U}$ can be used to transform \eqref{eq:E_mgrit} into a block diagonal matrix.
More specifically, one has
\begin{align} \label{eq:E_similarity_trans}
{\cal E}
\;
\underset{\rm unitary \, similarity\, transform} \mapsto
\;
\widecheck{{\cal E}}
=
\underset{1 \leq i \leq n_x}\diag \left( {\cal E}_i \right) \in \mathbb{C}^{n_x n_t \times n_x n_t}.
\end{align}
The unitary matrix involved in this transform can be written as the composition $(I_{n_t} \otimes {\cal U} ) {\cal P}$, in which ${\cal P}$ is a permutation matrix reordering space-time vectors from the original ordering where all spatial DOFs at a single time point are blocked together, to one in which all temporal DOFs belonging at a single spatial point are blocked together.\cite{Dobrev_etal_2017,
Friedhoff_MacLachlan2015,
DeSterck_etal_2019,
Hessenthaler_etal_2020,
Southworth_2019}
Since a unitary similarity transform of a matrix preserves its $\ell^2$-norm and eigenvalues, \eqref{eq:E_similarity_trans} simplifies the task of computing the norm and spectral radius of ${\cal E} \in \mathbb{C}^{n_x n_t \times n_x n_t}$ into one of computing norms and spectral radii of $n_x$ error propagators ${\cal E}_i \in \mathbb{C}^{n_t\times n_t}$.
The $i$th diagonal block in \eqref{eq:E_similarity_trans} takes the form
\begin{align} \label{eq:Ei_def}
\textrm{MGRIT error propagation matrix for mode $i$:} \quad
{\cal E}_i 
=
S^{\rm F}_i 
\left( 
I_{n_t}
-
P_i
A_{1,i}^{-1} 
P_i^\top
A_{0,i}
\right)
\big(
S^{\rm CF}_i 
\big)^{\nu}
S^{\rm F}_i 
\in \mathbb{C}^{n_t \times n_t},
\end{align}
and it characterizes the evolution of an error vector in the direction of the $i$th eigenvector of $\Phi$ and $\Psi$.
In other words, the effect of the similarity transform  \eqref{eq:E_similarity_trans} has been to decouple the space-time equation $\bm{e}^{(1)} = {\cal E} \bm{e}^{(0)}$ into $n_x$ time-only problems, with the $i$th such problem representing error propagation of the $i$th spatial mode or eigenvector.

The constituent terms in \eqref{eq:Ei_def} are found by applying the similarity transform from \eqref{eq:E_similarity_trans} to the components of ${\cal E}$ in \eqref{eq:E_mgrit}.
These spatial components can be written as
\begin{subequations}
\begin{alignat}{2}
\label{eq:A0i_def}
&\textrm{Fine-grid matrix for mode $i$:} \quad\quad\quad
A_{0,i}
&&= 
I_{n_t}
-
\lambda_i L_{n_t} \in \mathbb{C}^{n_t \times n_t},
\\
\label{eq:A1i_def}
&\textrm{Coarse-grid matrix for mode $i$:} \quad\quad
A_{1,i} 
&&= 
I_{n_t/m} 
-
\mu_i L_{n_t/m} \in \mathbb{C}^{n_t/m \times n_t/m},
\\
\label{eq:Pi_def}
&\textrm{Interpolation matrix for mode $i$:} \quad\quad
P_i 
&&=
I_{n_t/m} \otimes \bm{e}_1 \in \mathbb{C}^{n_t \times n_t/m}, 
\\
\label{eq:SFi_def}
&\textrm{F-relaxation matrix for mode $i$:} \quad\quad
S^{\rm F}_i 
&&= 
I_{n_t/m} \otimes \left[  v(\lambda_i) \bm{e}_1^\top \right] \in \mathbb{C}^{n_t \times n_t},
\\
\label{eq:SCFi_def}
&\textrm{CF-relaxation matrix for mode $i$:} \quad
S^{\rm CF}_i 
&&= 
L_{n_t/m} \otimes \left[ \lambda_i v(\lambda_i) \bm{e}_m^\top \right] \in \mathbb{C}^{n_t \times n_t}.
\end{alignat}
\end{subequations}
Here, $\bm{e}_1, \bm{e}_m \in \mathbb{R}^m$ are the canonical (column-oriented) basis vectors in the first and $m$th directions, respectively.
Recall the matrix $L_n$ and function $v$ are defined in \eqref{eq:Ln_v_defs}.

\section{Local Fourier analysis}
\label{sec:LFA}

In this section, LFA is used to analyze the temporal MGRIT error propagators ${\cal E}_i$ in \eqref{eq:Ei_def}.
The framework for the analysis is described in \Cref{sec:LFAintro}, key intermediate computations for the theory are given in \Cref{sec:LFAeigenmatrices}, and the main theoretical results are given in \Cref{sec:LFA-estimates}.
Note that the convergence theory developed in this section applies to general $\Phi$ and $\Psi$ satisfying the assumptions described earlier, and specific implications of this theory for advection-dominated PDEs will be described in \cref{sec:char_comp}.

\subsection{Introduction and preliminaries}
\label{sec:LFAintro}

To analyze the two-grid problem with LFA, we consider it posed on a pair of semi-infinite temporal grids, and we ignore the influence and effects of boundary conditions by replacing the initial condition at $t = 0$ and the outflow condition at the final time with a periodic boundary condition. 
To this end, with $\ell$ denoting the level in the multigrid hierarchy, we associate the semi-infinite temporal grids
\begin{align} \label{eq:G_ell}
\bm{G}_{\ell} 
\coloneqq
\left\{ t_k = k m^{\ell} \delta t \colon k \in \mathbb{N}_0 \right\}, \quad \ell \in \{0, 1\}.
\end{align}
On the grids $\bm{G}_{\ell}$, we consider the infinite-dimensional extension of matrix ${\cal E}_i$ in \eqref{eq:Ei_def} and of the matrices in \eqref{eq:A0i_def}--\eqref{eq:SCFi_def} that compose it. 
In addition, we consider the following Fourier modes on grids \eqref{eq:G_ell} with continuously varying frequency $\theta$:
\begin{align} \label{eq:Fourier_modes_def}
\varphi_{\ell}(\theta, t) \coloneqq \exp\left( \frac{\im \theta t}{m^{\ell} \delta t} \right), 
\quad t \in \bm{G}_{\ell}, 
\quad \theta \in \Theta_{\ell}, 
\end{align}
with
\begin{align} \label{eq:Theta_ell_def}
\Theta_{\ell} = 
\begin{cases}
\displaystyle{\left[-\frac{\pi}{m}, 2 \pi -\frac{\pi}{m} \right)}, \quad &\ell = 0, \\
\left[-\pi, \pi \right) \quad &\ell = 1.
\end{cases}
\end{align}
Any intervals of length $2 \pi$ could be used for $\Theta_{\ell}$, but these choices provide some notational simplifications.\cite{DeSterck_etal_2019} 
We adopt the shorthand $\bm{\varphi}_{\ell}(\theta)$ for denoting the vector that is populated with the Fourier mode \eqref{eq:Fourier_modes_def} sampled at all time points $t \in \bm{G}_{\ell}$.
The modes $\bm{\varphi}_{\ell}(\theta)$ lie at the heart of LFA because they are formally eigenfunctions of any infinite-dimensional Toeplitz operator that acts on the grid $\bm{G}_{\ell}$.\cite{Trottenberg_etal_2001,Wienands_Joppich_2005}

On $\bm{G}_{\ell}$, we introduce the scaled Hermitian inner product of two grid functions $a_{\ell}, b_{\ell} \colon \bm{G}_{\ell} \to \mathbb{C}$ as
\begin{align} \label{eq:inner_prod}
\langle \bm{a}_{\ell}, \bm{b}_{\ell} \rangle 
:= 
\lim \limits_{n_t \to \infty} \frac{1}{n_t} \sum_{k = 0}^{n_t-1} \bar{a}_{\ell} (t_k) b_{\ell} (t_k),
\end{align}
where $\bar{a}(t_k)$ denotes the complex conjugate of $a (t_k)$. 
Note that the Fourier modes \eqref{eq:Fourier_modes_def} are orthonormal with respect to this inner product.

We partition the frequency space $\Theta_{0}$ into two disjoint sets according to
\begin{align} \label{eq:Theta_low_def}
\Theta^{\rm low} \coloneqq \left[ -\frac{\pi}{m}, \frac{\pi}{m} \right), 
\quad  
\Theta^{\rm high} \coloneqq \left[ \frac{\pi}{m}, 2\pi - \frac{\pi}{m}\right).
\end{align}
Observe that for any $\theta \in \Theta^{\rm low}$, we have $\varphi_0 \left(\theta + \tfrac{2 \pi \alpha}{m}, t\right) = \tfrac{1}{\sqrt{m}} \varphi_1 \left(m \theta, t\right)$ for all $t \in \bm{G}_{1}$ and ${\alpha \in \{0, \ldots, m-1\}}$.
These $m$ fine-grid functions, $\varphi_0 \left(\theta + \tfrac{2 \pi \alpha}{m}, t\right)$, are known as \textit{harmonics} of one another. 
The fact that the harmonics are indistinguishable from one another when sampled on the coarse-grid points is the motivation for the partitioning in \eqref{eq:Theta_low_def} (see De~Sterck~et~al.,\cite{DeSterck_etal_2019} Section 4.1.2).
A second useful fact is that constant-stencil interpolation and restriction operators map the function $\varphi_1(m\theta,t)$ for $\theta\in\Theta^{\rm low}$ onto the associated harmonic functions and vice-versa, leading to the definition of the $m$-dimensional spaces of harmonics.
\begin{definition}[$m \delta t$-harmonics] 
\label{def:mh-harmonics}
For a given $\theta \in \Theta^{\rm low}$, the associated $m$-dimensional space of harmonics is 
\begin{align} \label{eq:mh-harmonics}
{\cal H}^{\theta}_{\delta t}
\coloneqq
\underset{0 \leq \alpha < m}{\rm span}
\left\{
\bm{\varphi}_0 \left(\theta + \frac{2 \pi \alpha}{m} \right) 
\right\}.
\end{align}
\end{definition}

As we will show, ${\cal H}^{\theta}_{\delta t}$ is also an invariant subspace of both C- and F-relaxation, although each Fourier mode is not an eigenfunction.  Thus, for a given $\theta \in \Theta^{\rm low}$, the action of the various MGRIT components \eqref{eq:A0i_def}--\eqref{eq:SCFi_def} on Fourier modes \eqref{eq:Fourier_modes_def} is characterized by
\begin{subequations}
\begin{alignat}{2}
\label{eq:A0i_map}
&\textrm{Fine-grid matrix for mode $i$:} \quad\quad\quad
A_{0,i}
&&\colon 
{\cal H}^{\theta}_{\delta t} \to {\cal H}^{\theta}_{\delta t},
\\
\label{eq:A1i_map}
&\textrm{Coarse-grid matrix for mode $i$:} \quad\quad
A_{1,i}
&&\colon 
{\rm span} \{{\varphi}_1(m \theta) \} \to {\rm span} \{{\varphi}_1(m \theta) \},
\\ 
\label{eq:Pi_map}
&\textrm{Interpolation matrix for mode $i$:} \quad\quad
P_i 
&&\colon 
{\rm span} \{{\varphi}_1(m \theta) \} \to {\cal H}^{\theta}_{\delta t},\\
&\textrm{F-relaxation matrix for mode $i$:} \quad\quad
S^{\rm F}_i 
&&\colon 
{\cal H}^{\theta}_{\delta t} \to {\cal H}^{\theta}_{\delta t},\\
\label{eq:SCFi_map}
&\textrm{CF-relaxation matrix for mode $i$:} \quad
S^{\rm CF}_i 
&&\colon 
{\cal H}^{\theta}_{\delta t} \to {\cal H}^{\theta}_{\delta t}.
\end{alignat}
\end{subequations}
Note that the spans in \eqref{eq:A1i_map} and \eqref{eq:Pi_map} are over a one-dimensional set.
Since the MGRIT error propagator ${\cal E}_i$ is composed of the above operators (see \eqref{eq:Ei_def}), it is invariant on the space of $m \delta t$-harmonics:
\begin{align}
\textrm{MGRIT error propagation matrix for mode $i$:} \quad
{\cal E}_i \colon {\cal H}^{\theta}_{\delta t} 
\to 
{\cal H}^{\theta}_{\delta t} 
\quad \textrm{for all } 
\theta \in \Theta^{\rm low}.
\end{align}
Therefore, by grouping together harmonic Fourier modes, the infinite-dimensional error propagator ${\cal E}_i$ can be block diagonalized. Specifically, the transformed operator has one diagonal block $\wh{{\cal E}}_i(\theta) \in \mathbb{C}^{m \times m}$ associated with each $\theta \in \Theta^{\rm low}$, where, for a given $\theta \in \Theta^{\rm low}$, $\wh{{\cal E}}_i(\theta)$ is the representation of ${\cal E}_i$ on the harmonic space ${\cal H}^{\theta}_{\delta t}$.
That is,
\begin{align} \label{eq:Ei_similarity}
{\cal E}_i 
\;
\underset{\rm unitary\, similarity\, transform} \mapsto
\;
\underset{\theta \in \Theta^{\rm low}} 
\diag 
\big( 
\wh{{\cal E}}_i(\theta)
\big).
\end{align}
We call $\wh{{\cal E}}_i(\theta)$ the \textit{Fourier symbol} of ${\cal E}_i$ associated with the harmonic space ${\cal H}^{\theta}_{\delta t}$ (or just the Fourier symbol of ${\cal E}_i$ for short).

Let $V(\theta)$ be a matrix with $m$ columns given by the $m$ harmonic Fourier modes from \eqref{eq:mh-harmonics}:
\begin{align} \label{eq:V_basis_def}
V(\theta)
:=
\begin{bmatrix}
\bm{\varphi}_0(\theta)
\quad
\bm{\varphi}_0 \left(\theta + \frac{2\pi}{m}  \right)
\quad
\ldots
\quad
\bm{\varphi}_0 \left(\theta + \frac{2\pi(m-1)}{m}  \right)
\end{bmatrix}.
\end{align}
Then, the Fourier symbol arising in the similarity transform \eqref{eq:Ei_similarity} associated with ${\cal H}^{\theta}_{\delta t}$ satisfies
$
{\cal E}_i V(\theta) 
=
V(\theta) 
\wh{{\cal E}}_i(\theta)
$.
Since $V$ has orthonormal columns, the Fourier symbol itself can be written
$
\wh{{\cal E}}_i(\theta)
=
V^*(\theta) {\cal E}_i V(\theta) \in \mathbb{C}^{m \times m}
$.
Since the $\ell^2$-norm and the eigenvalues of a matrix are preserved under a unitary similarity transform, from \eqref{eq:Ei_similarity} we have
\begin{align} 
\label{eq:Ei_norm_rho_equivalence}
\Vert {\cal E}_i \Vert_2
=
\underset{\theta \in \Theta^{\rm low}} \sup
\big\Vert 
\wh{{\cal E}}_i(\theta)
\big\Vert_2,
\quad
\rho ( {\cal E}_i )
=
\underset{\theta \in \Theta^{\rm low}} \sup
\rho \big(
\wh{{\cal E}}_i(\theta)
\big).
\end{align}
Thus, the computation of the norm and spectral radius of the infinite-dimensional ${\cal E}_i$ has been reduced to the computation of these quantities on an infinite number of finite-dimensional matrices, $\wh{{\cal E}}_i(\theta) \in \mathbb{C}^{m \times m}$, $\theta \in \Theta^{\rm low}$.

From the definition of ${\cal E}_i$ in \eqref{eq:Ei_def}, the error propagator $\wh{{\cal E}}_i(\theta)$ takes the form 
\begin{align}
\label{eq:Ei_eigen_components}
\textrm{Fourier symbol of MGRIT error propagation matrix for mode $i$:} \quad
\wh{{\cal E}}_i(\theta)
&=
\wh{S}^{\rm F}_i (\theta)
\,
\wh{{\cal K}}_i(\theta)
\,
\big(
\wh{S}^{\rm CF}_i (\theta)
\big)^{\nu}
\,
\wh{S}^{\rm F}_i (\theta)
\in 
\mathbb{C}^{m \times m},
\end{align}
where $\wh{{\cal K}}_i(\theta)$ is the Fourier symbol of the coarse-grid correction,
\begin{align}
\label{eq:Ki_eigen_components}
\wh{{\cal K}}_i(\theta)
&=
I_m
-
\wh{P}_i (\theta)
\,
\big[ \wh{A}_{1,i} (m \theta) \big]^{-1}
\,
\wh{P}_i^\top (\theta)
\,
\wh{A}_{0,i} (\theta)
\in 
\mathbb{C}^{m \times m}.
\end{align}
The component matrices used in \eqref{eq:Ei_eigen_components} and \eqref{eq:Ki_eigen_components} are the Fourier symbols of the infinite-dimensional extensions of the multigrid components defined in \eqref{eq:A0i_def}--\eqref{eq:SCFi_def}.
From \eqref{eq:A0i_map}--\eqref{eq:SCFi_map}, these Fourier symbols may be expressed as
\begin{subequations}
\begin{alignat}{2}
\label{eq:A0i_eigen_def}
&\textrm{Fourier symbol of fine-grid matrix for mode $i$:} \quad\quad\quad
\wh{A}_{0,i}(\theta)
&&=
V^*(\theta) {A}_{0,i} V(\theta) \in \mathbb{C}^{m \times m},\\
\label{eq:A1i_eigen_def}
&\textrm{Fourier symbol of coarse-grid matrix for mode $i$:} \quad\quad
\wh{A}_{1,i}(m \theta)
&&=
\bm{\varphi}_1^*(m \theta) {A}_{1,i} \bm{\varphi}_1(m \theta) \in \mathbb{C},\\
\label{eq:Pi_eigen_def}
&\textrm{Fourier symbol of interpolation matrix for mode $i$:} \quad\quad
\wh{P}_i (\theta)
&&=
V^*(\theta) {P}_i \bm{\varphi}_1(m \theta) \in \mathbb{C}^{m},\\
\label{eq:SFi_eigen_def}
&\textrm{Fourier symbol of F-relaxation matrix for mode $i$:} \quad\quad
\wh{S}^{\rm F}_i (\theta)
&&= 
V^*(\theta) {S}^{\rm F}_i V(\theta) \in \mathbb{C}^{m \times m},\\
\label{eq:SCFi_eigen_def}
&\textrm{Fourier symbol of CF-relaxation matrix for mode $i$:} \quad
\wh{S}^{\rm CF}_i (\theta)
&&= 
V^*(\theta) {S}^{\rm CF}_i V(\theta) \in \mathbb{C}^{m \times m}.
\end{alignat}
\end{subequations}
%

\subsection{Derivations of Fourier symbols}
\label{sec:LFAeigenmatrices}

Deriving expressions for the Fourier symbols \eqref{eq:Ei_eigen_components}--\eqref{eq:SCFi_eigen_def} is a straightforward, but tedious calculation. Here, we state the main results, but defer their proofs to \cref{app:eigenmatrix-proofs}.
We note that the LFA theory in De~Sterck~et~al.\cite{DeSterck_etal_2019} also derived Fourier symbols for MGRIT; however, De~Sterck~et~al.\cite{DeSterck_etal_2019} used both a different basis than that in~\eqref{eq:V_basis_def} and an alternative formulation of the error-propagation operator, so the expressions derived here do not match those there.
Furthermore, the analysis in De~Sterck~et~al.\cite{DeSterck_etal_2019} focused on numerical evaluation of the LFA convergence estimates, while the focus here will be to algebraically manipulate the Fourier symbols in order to give closed-form convergence estimates.

\subsubsection{Fourier symbols of fine- and coarse-grid operators, and interpolation}
\label{sec:LFAeigenmatrices_interp}

Since $A_{0,i}$ in \eqref{eq:A0i_def} and $A_{1,i}$ in \eqref{eq:A1i_def} are infinite-dimensional Toeplitz operators, the Fourier modes $\bm{\varphi}_{\ell}(\theta)$ are their eigenfunctions and, so, their Fourier symbols are simply
\begin{align}
\label{eq:A0and1i_eigen}
\wh{A}_{i,0}(\theta) 
=
\underset{0 \leq \alpha < m} \diag \left( \wt{A}_{0,i} \left(\theta + \frac{2 \pi \alpha}{m} \right) \right)
\in \mathbb{C}^{m \times m},
\quad
\wh{A}_{1,i}(m\theta) 
=
\wt{A}_{1,i}(m \theta) \in \mathbb{C}.
\end{align}
Here, $\wt{A}_{\ell,i}(\theta)$ is the eigenvalue or Fourier symbol of $A_{i}$ associated with the eigenfunction $\bm{\varphi}_{\ell}(\theta)$. 
Due to the lower bidiagonal structure of $A_{i}$, these Fourier symbols are given by
\begin{align} \label{eq:Ai_symbols}
\wt{A}_{0,i}(\theta) = 1 - \lambda_i e^{-\im \theta} \in \mathbb{C},
\quad
\wt{A}_{1,i}(m \theta) = 1 - \mu_i e^{-\im m \theta} \in \mathbb{C}.
\end{align}

We next consider the Fourier symbol of interpolation \eqref{eq:Pi_eigen_def}, which can be derived by considering its transpose, which is restriction by injection.
Recall that for any ${\theta \in \Theta^{\rm low}}$, $\varphi_0 \left(\theta + \tfrac{2 \pi \alpha}{m}, t\right) = \tfrac{1}{\sqrt{m}} \varphi_1 \left(m \theta, t\right)$ for time points $t \in \bm{G}_{1}$, where ${\alpha \in \{0, \ldots, m-1\}}$.
Thus, injection of the values of $\varphi_0 \left(\theta + \tfrac{2 \pi \alpha}{m}, t\right)$ at time points $t \in \bm{G}_{1}$ maps any fine-grid harmonic to the associated coarse-grid Fourier mode $\varphi_1 \left(m \theta, t\right)$, with its amplitude scaled by $\tfrac{1}{\sqrt{m}} $.
Thus, the Fourier symbol of injection restriction is $\tfrac{1}{\sqrt{m}} \bm{1}^\top$, where $\bm{1} \in \mathbb{R}^m$ denotes a column vector of ones.
Note that when interpolation is the adjoint of restriction, its Fourier symbol is the adjoint of that of restriction up to some constant scaling (see, e.g., Remark 4.4.3 in Trottenberg~et~al.\cite{Trottenberg_etal_2001}), with injection representing the case in which the scaling is unity. 
Therefore the Fourier symbol of injection interpolation is simply
\begin{align} \label{eq:Pi_eigen}
\wh{P}_i (\theta) = \frac{1}{\sqrt{m}} \bm{1}.
\end{align}
%

\subsubsection{Relaxation Fourier symbols}
\label{sec:LFAeigenmatrices_relax}

We now compute the relaxation Fourier symbols \eqref{eq:SFi_eigen_def} and \eqref{eq:SCFi_eigen_def}. 
These Fourier symbols are more complicated than those in the previous section because they intermix harmonics in a non-trivial way.

\begin{lemma}[F-relaxation Fourier symbol]
\label{lem:eigmat-F-relax}
The Fourier symbol \eqref{eq:SFi_eigen_def} of F-relaxation may be written as the following rank-1 matrix
\begin{align} \label{eq:SFi_eigen}
\wh{S}^{\rm F}_i (\theta) = c(\theta) 
\big[ \wh{A}_{0,i} (\theta) \big]^{-1}
\bm{1}
\bm{1}^\top,
\end{align}
in which $\wh{A}_{0,i} (\theta)$ is given in \eqref{eq:A0and1i_eigen}, and $c(\theta)$ is the function
\begin{align} \label{eq:c_def}
c(\theta) := \frac{1}{m} \left[1 - \big(\lambda_i e^{-\im \theta} \big)^m \right].
\end{align}
\end{lemma}
\begin{proof}
See \Cref{sec:proof:eigmat-F-relax}.
\end{proof}

The fact that $\wh{S}^{\rm F}_i (\theta)$ is dense reflects that F-relaxation couples together the $m$ harmonics in ${\cal H}^{\theta}_{\delta t}$ (see \cref{def:mh-harmonics}).
This contrasts with simple relaxation methods typically used in the multigrid solution of Poisson problems, for example, for which the Fourier symbol of relaxation is diagonal (or block diagonal), representing that harmonics are not mixed (or partially mixed) by relaxation.\cite{Trottenberg_etal_2001}

Recall that F-relaxation updates F-point values to have zero residuals while leaving C-point values unchanged. 
Thus, after an initial F-relaxation successive F-relaxations have no effect since F-point residuals are already zero.
The following useful result shows that the Fourier symbol $\wh{S}^{\rm F}_i(\theta)$ for F-relaxation inherits this property.
\begin{corollary}[Idempotence of F-relaxation]
\label{cor:SF_idempotence}
The Fourier symbol $\wh{S}^{\rm F}_i (\theta)$ for F-relaxation, as given by \eqref{eq:SFi_eigen}, is idempotent:
$
\wh{S}^{\rm F}_i (\theta) \wh{S}^{\rm F}_i (\theta) = \wh{S}^{\rm F}_i (\theta).
$
Furthermore, $\wh{S}^{\rm F}_i (\theta)$ has a single eigenvalue of unity, and $m-1$ eigenvalues that are zero.
\end{corollary}
\begin{proof}
See \cref{sec:proof:idempotence-of-F-relax}.
\end{proof}

Having identified the Fourier symbol for F-relaxation in \cref{lem:eigmat-F-relax}, we now consider the Fourier symbol for CF-relaxation and that of pre-relaxation as a whole.
\begin{lemma}[Pre-relaxation Fourier symbol]
\label{lem:eigmat-prerelax}
The Fourier symbol \eqref{eq:SCFi_eigen_def} for CF-relaxation may be expressed as
\begin{align} \label{eq:SCFi_eigen}
\wh{S}^{\rm CF}_i (\theta) 
=
\wh{S}^{\rm F}_i (\theta) 
\Big[
I - \wh{A}_{0,i}(\theta)
\Big],
\end{align}
where $\wh{S}^{\rm F}_i (\theta)$ is given by \eqref{eq:SFi_eigen}, and $\wh{A}_{0,i} (\theta)$ is given in \eqref{eq:A0and1i_eigen}.
Furthermore, the Fourier symbol for the entire pre-relaxation operator in \eqref{eq:Ei_eigen_components} may be expressed as
\begin{align} \label{eq:pre-relax_eigen}
\big(
\wh{S}^{\rm CF}_i (\theta)
\big)^{\nu}
\,
\wh{S}^{\rm F}_i (\theta)
=
\big( \lambda_i e^{- \im \theta} \big)^{m \nu} \, \wh{S}^{\rm F}_i (\theta), 
\quad 
\nu \in \mathbb{N}_0.
\end{align}
\end{lemma}
\begin{proof}
See \Cref{sec:proof:eigmat-prerelax}.
\end{proof}
%

\subsubsection{Error propagator Fourier symbol}
\label{sec:LFAeigenmatrices_error}

Having computed convenient representations for the Fourier symbols that compose the Fourier symbol for the error propagator in \eqref{eq:Ei_eigen_components}, we now compute this Fourier symbol itself.
\begin{theorem}[Error propagator Fourier symbol]
\label{thm:eigmat-error-prop}
The error propagator Fourier symbol \eqref{eq:Ei_eigen_components} may be written as
\begin{align} \label{eq:Ei_eigen}
\wh{{\cal E}}_i(\theta)
=
f(\theta)
\wh{S}^{\rm F}_i (\theta),
\end{align}
where $\wh{S}^{\rm F}_i (\theta)$ is the Fourier symbol for F-relaxation given in \eqref{eq:SFi_eigen}, and $f(\theta)$ is
\begin{align}
\label{eq:f_def}
f(\theta)
:=
\big( \lambda_i e^{-\im \theta} \big)^{m \nu}
\frac{\lambda_i^m - \mu_i}{e^{\im m \theta} - \mu_i}.
\end{align}
\end{theorem}
\begin{proof}
See \cref{sec:proof:eigmat-error-prop}.
\end{proof}

\subsection{LFA convergence estimates}
\label{sec:LFA-estimates}

In this section, we present our main theoretical results on the LFA estimates for error propagation, which are based on the \textit{simple} representation of the error propagator Fourier symbol given in \Cref{thm:eigmat-error-prop}.
We begin with the norm of this matrix.
\begin{theorem}[Error propagator Fourier symbol norm]
\label{thm:Ei_eigen_norm}
The $\ell^2$-norm of the Fourier symbol for the error propagator given in \cref{thm:eigmat-error-prop} is 
\begin{align}
\label{eq:Ei_eigen_norm}
\big\Vert
\wh{{\cal E}}_i(\theta)
\big\Vert_2
=
|\lambda_i|^{m \nu}
\frac{|\lambda_i^m - \mu_i|}{|e^{\im m \theta} - \mu_i|}
\sqrt{ \frac{1 - |\lambda_i|^{2m}}{1 - |\lambda_i|^2} }.
\end{align}
\end{theorem}
\begin{proof}
Using \eqref{eq:Ei_eigen}, the squared norm of $\wh{{\cal E}}_i(\theta)$ can be expressed as
$
\big\Vert
\wh{{\cal E}}_i(\theta)
\big\Vert_2^2
=
\rho \big(
\wh{{\cal E}}_i^*(\theta)
\wh{{\cal E}}_i(\theta)
\big)
=
|f(\theta)|^2
\rho
\Big( 
\big( \wh{S}^{\rm F}_i (\theta) \big)^*
\wh{S}^{\rm F}_i (\theta)
\Big).
$
Substituting $\wh{S}^{\rm F}_i (\theta) = c(\theta) \big[ \wh{A}_{0,i}(\theta)\big]^{-1} \bm{1} \bm{1}^{\top}$ from \eqref{eq:SFi_eigen}, this becomes
\begin{align} 
\big\Vert
\wh{{\cal E}}_i(\theta)
\big\Vert_2^2
&=
|f(\theta)|^2
\rho\Big(
\Big[
\bm{1} \bm{1}^\top \bar{c}(\theta) \big[ \wh{A}^*_{0,i}(\theta) \big]^{-1}
\Big]
\Big[
c(\theta) \big[ \wh{A}_{0,i}(\theta) \big]^{-1}
\bm{1} \bm{1}^\top
\Big]
\Big),\\
\label{eq:Ei_eigen_norm_square_intermediate0}
&=
|f(\theta)|^2
\rho\Big(
\Big[
\bm{1}^\top \bar{c}(\theta) \big[ \wh{A}^*_{0,i}(\theta) \big]^{-1}
c(\theta) \big[ \wh{A}_{0,i}(\theta) \big]^{-1}
\bm{1}
\Big]
\bm{1} \bm{1}^\top
\Big),\\
\label{eq:Ei_eigen_norm_square_intermediate}
&=
m
|f(\theta)|^2
\,
\Big(
\bm{1}^\top \bar{c}(\theta) \big[ \wh{A}^*_{0,i}(\theta) \big]^{-1}
c(\theta) \big[ \wh{A}_{0,i}(\theta) \big]^{-1}
\bm{1}
\Big).
\end{align}
To arrive at \eqref{eq:Ei_eigen_norm_square_intermediate}, note that the term in the closed parenthesis multiplying $\bm{1} \bm{1}^\top$ in \eqref{eq:Ei_eigen_norm_square_intermediate0} is a non-negative scalar, and that $\rho(\bm{1} \bm{1}^\top) = \bm{1}^\top \bm{1} = m$.
Now we consider the inner product in \eqref{eq:Ei_eigen_norm_square_intermediate}.
Using the identity of \eqref{eq:geo_sum_identity} for the $p$th element, $p \in \{0, \ldots, m-1\}$, of $c(\theta) \big[ \wh{A}_{0,i}(\theta) \big]^{-1}
\bm{1}$, we have
\begin{align} \label{eq:geo_identity_real}
\Big[ c(\theta) \big[ \wh{A}_{0,i}(\theta) \big]^{-1}
\bm{1} \Big]_p
&=
\frac{1}{m}
\sum
\limits_{r = 0}^{m-1}
\big[
\lambda_i 
\exp\big( -\im \big( \theta + \tfrac{2\pi p}{m} \big) \big)
\big]^r.
\end{align}
Using \eqref{eq:geo_identity_real}, the inner product of concern can thus be written as
\begin{align}
\nonumber
\bm{1}^\top \bar{c}(\theta) \big[ \wh{A}^*_{0,i}(\theta) \big]^{-1}
c(\theta) \big[ \wh{A}_{0,i}(\theta) \big]^{-1}
\bm{1}
&=
\sum \limits_{p = 0}^{m-1}
\Bigg(
\Big[ \bm{1}^\top \bar{c}(\theta) \big[ \wh{A}^*_{0,i}(\theta) \big]^{-1}
\Big]_p
\Big[ c(\theta) \big[ \wh{A}_{0,i}(\theta) \big]^{-1}
\bm{1} \Big]_p
\Bigg),
\\
&=
\frac{1}{m^2}
\sum \limits_{p = 0}^{m-1}
\Bigg(
\sum \limits_{r = 0}^{m-1}
\big[
\lambda_i 
\exp\big( -\im \big( \theta + \tfrac{2\pi p}{m} \big) \big)
\big]^r
\sum \limits_{s = 0}^{m-1}
\big[
\bar{\lambda}_i 
\exp\big( \im \big( \theta + \tfrac{2\pi p}{m} \big) \big)
\big]^s
\Bigg), 
\\
&=
\frac{1}{m^2}
\sum \limits_{p = 0}^{m-1}
\sum \limits_{r = 0}^{m-1}
\sum \limits_{s = 0}^{m-1}
\lambda_i^r \bar{\lambda}_i^s
e^{\im  \theta (s - r )}
\exp \Big(  \tfrac{2\pi \im p(s - r)}{m} \Big),
\\
\label{eq:inner_prod_intermediate}
&= 
\frac{1}{m^2}
\sum \limits_{r = 0}^{m-1}
\sum \limits_{s = 0}^{m-1}
\lambda_i^r \bar{\lambda}_i^s
e^{\im  \theta (s-r)}
\bigg(
\sum \limits_{p = 0}^{m-1}
\exp\Big(  \tfrac{2\pi \im p(s - r)}{m} \Big)
\bigg).
\end{align}
Slightly rewriting the geometric sum in parentheses in \eqref{eq:inner_prod_intermediate} it becomes
\begin{align} \label{eq:trailing_geo_sum}
\sum \limits_{p = 0}^{m-1}
\Big[\exp\Big(  \tfrac{2\pi \im (s - r)}{m} \Big)\Big]^p
=
\begin{cases}
m, \quad \textrm{if } (s-r) \bmod m = 0,\\
0, \quad \textrm{else}.
\end{cases}
\end{align}
In \eqref{eq:inner_prod_intermediate}, $r,s \in \{0, \ldots, m-1\}$, and, thus, $s-r \in \{ 1-m, \ldots, m-1\}$. 
Therefore, the only time that $s-r$ is an integer multiple of $m$ is when $s-r = 0$, or $r = s$. 
As such, \eqref{eq:trailing_geo_sum} is equal to $m \delta _{r,s}$, where $\delta$ is the Kronecker delta function.
Substituting this result into \eqref{eq:inner_prod_intermediate} gives
\begin{align} \label{eq:inner_product_in_norm}
\bm{1}^\top \bar{c}(\theta) \big[ \wh{A}^*_{0,i}(\theta) \big]^{-1}
c(\theta) \big[ \wh{A}_{0,i}(\theta) \big]^{-1}
\bm{1}
=
\frac{1}{m^2}
\sum \limits_{r = 0}^{m-1}
\sum \limits_{s = 0}^{m-1}
\lambda_i^r \bar{\lambda}_i^s
e^{\im  \theta (s-r)} m \delta_{r,s}
=
\frac{1}{m}
\sum \limits_{r = 0}^{m-1}
\lambda_i^r \bar{\lambda}_i^r
=
\frac{1}{m} \frac{1 - |\lambda_i|^{2m}}{1 - |\lambda_i|^2},
\end{align}
with the final equality following from the prior expression being a geometric sum in $|\lambda_i|^2$.

The claimed result \eqref{eq:Ei_eigen_norm} follows by substituting \eqref{eq:inner_product_in_norm} into \eqref{eq:Ei_eigen_norm_square_intermediate} along with the value of $|f(\theta)|^2$ coming from \eqref{eq:f_def}, and then taking the square root of the result. 
\end{proof}

Recall from \eqref{eq:Ei_norm_rho_equivalence} that the norm of ${\cal E}_i$ is given by the maximum norm of its Fourier symbols $\wh{{\cal E}}_i(\theta)$ over $\theta \in \Theta^{\rm low}$.
\begin{theorem}[Error propagator norm]
\label{thm:Ei_norm}
The $\ell^2$-norm of the error propagator ${\cal E}_i$ defined in \eqref{eq:Ei_def} is
\begin{align}
\label{eq:Ei_norm}
\Vert
{{\cal E}}_i
\Vert_2
=
\underset{\theta \in \Theta^{\rm low}} \sup
\big\Vert
\wh{{\cal E}}_i(\theta)
\big\Vert_2
=
|\lambda_i|^{m \nu}
\frac{|\lambda_i^m - \mu_i|}{1 - |\mu_i|}
\sqrt{ \frac{1 - |\lambda_i|^{2m}}{1 - |\lambda_i|^2} }.
\end{align}
Furthermore, of all harmonic spaces ${\cal H}_{\delta t}^{\theta}$ with $\theta \in \Theta^{\rm low}$ from \cref{def:mh-harmonics}, the one with the least error reduction is that associated with frequency $\theta = \theta^{\dagger}_i$, where 
\begin{align} \label{eq:theta_dagger}
\theta^{\dagger}_i := \underset{\theta \in \Theta^{\rm low}}\argmax \,
\big\Vert
\wh{{\cal E}}_i(\theta)
\big\Vert_2
=
\frac{1}{m} \arg \mu_i,
\end{align}
in which $\arg \mu_i$ denotes the argument of the complex number $\mu_i$.
\end{theorem}
\begin{proof}
The first equality in \eqref{eq:Ei_norm} was already given as \eqref{eq:Ei_norm_rho_equivalence}.
Let us first consider the slowest converging harmonic space, and then return to the second equality in \eqref{eq:Ei_norm}.
From the expression for $ \big\Vert\wh{{\cal E}}_i(\theta) \big\Vert_2$ given in \eqref{eq:Ei_eigen_norm}, the only dependence on frequency $\theta$ is via the term $\tfrac{1}{|e^{\im m \theta} - \mu_i|}$.
Therefore, we have
\begin{align} 
\theta^{\dagger}_i := 
\underset{\theta \in \Theta^{\rm low}}\argmax \,
\big\Vert
\wh{{\cal E}}_i(\theta)
\big\Vert_2
=
\underset{\theta \in \Theta^{\rm low}}\argmax \,
\frac{1}{\big| e^{\im m \theta} - \mu_i \big|}
=
\underset{\theta \in \Theta^{\rm low}}\argmin \,
\big| e^{\im m \theta} - \mu_i \big|.
\end{align}
Furthermore, since $\Theta^{\rm low}$ is the continuous frequency space spanning $\big[-\tfrac{\pi}{m}, \tfrac{\pi}{m} \big)$ (see \eqref{eq:Theta_low_def}), introducing  the new variable $\vartheta = m \theta$ gives
\begin{align}
\underset{\theta \in \Theta^{\rm low}}\min \,
\big| e^{\im m \theta} - \mu_i \big|
=
\underset{\vartheta  \in [-\pi, \pi)}\min \,
\big| e^{\im \vartheta} - \mu_i \big|.
\end{align}
This is simply the shortest distance from the unit circle to the complex number $\mu_i$ that lies inside it (recall that $|\mu_i| < 1$ under \Cref{ass:stability}).
By a simple geometric argument, this distance is minimized by the point on the unit circle having the same argument as $\mu_i$; that is, the minimum over $\vartheta \in [-\pi, \pi)$ is achieved at $\vartheta = \arg \mu_i$.
Since $\theta = \vartheta / m$, the minimizing frequency over $\theta \in \Theta^{\rm low}$ is simply $\theta^{\dagger}_i = \frac{1}{m} \arg \mu_i$.

To evaluate this minimum distance, write $\mu_i$ in polar form as $\mu_i = |\mu_i| e^{\im \arg \mu_i}$ and substitute it into the above equation to yield
\begin{align}
\underset{\theta \in \Theta^{\rm low}}\min \,
\big| e^{\im m \theta} - \mu_i \big|
=
\big| e^{\im \arg \mu_i } - |\mu_i| e^{\im \arg \mu_i} \big|
=
| 1 - |\mu_i| | \big| e^{\im \arg \mu_i } \big|
=
1 - |\mu_i|,
\end{align}
with the last equality following since $|\mu_i| < 1$.
Finally, the result \eqref{eq:Ei_norm} follows by evaluating $\big\Vert
\wh{{\cal E}}_i(\theta^{\dagger}_i)
\big\Vert_2$ from \eqref{eq:Ei_eigen_norm} using the fact that $\big| e^{\im m \theta^{\dagger}_i} - \mu_i \big| = 1 - |\mu_i|$.
\end{proof}

Given the structure of the Fourier symbol $\wh{{\cal E}}_i(\theta)$ in \Cref{thm:eigmat-error-prop}, and the result from \Cref{thm:Ei_norm}, it is straightforward to compute the spectral radius of ${{\cal E}}_i$.
\begin{corollary}[Error propagator spectral radius]
\label{cor:Ei_spectral_radius}
The spectral radius of the error propagator Fourier symbol $\wh{{\cal E}}_i(\theta)$ given in \cref{thm:eigmat-error-prop} is 
\begin{align} \label{eq:Ei_eigen_spectral_radius}
\rho \big(
\wh{{\cal E}}_i(\theta)
\big)
=
|\lambda_i |^{m \nu}
\frac{|\lambda_i^m - \mu_i|}{|e^{\im m \theta} - \mu_i|}.
\end{align}
Furthermore, the spectral radius of the error propagator ${\cal E}_i$ given in \eqref{eq:Ei_def} is
\begin{align} \label{eq:Ei_spectral_radius}
\rho ( {\cal E}_i )
=
\underset{\theta \in \Theta^{\rm low}} \sup
\rho \big(
\wh{{\cal E}}_i(\theta)
\big)
=
\rho \big(
\wh{{\cal E}}_i(\theta_i^\dagger)
\big)
=
|\lambda_i |^{m \nu}
\frac{|\lambda_i^m - \mu_i|}{1 - |\mu_i|}.
\end{align}
\end{corollary}
\begin{proof}
From the expression for $\wh{{\cal E}}_i(\theta)$ given in \eqref{eq:Ei_eigen}, its spectral radius is simply 
\begin{align}
\rho \big(
\wh{{\cal E}}_i(\theta)
\big)
=
\rho \big( f(\theta) \wh{S}^{\rm F}_i (\theta) \big)
=
|f(\theta)| \rho \big( \wh{S}^{\rm F}_i (\theta) \big)
=
|f(\theta)|,
\end{align}
with the final equality following as an immediate consequence of \Cref{cor:SF_idempotence}.
Substituting $f(\theta)$ from its definition given in \eqref{eq:f_def} gives the claimed result of \eqref{eq:Ei_eigen_spectral_radius}.

The first equality in \eqref{eq:Ei_spectral_radius} was already given as \eqref{eq:Ei_norm_rho_equivalence}.
The second equality follows from $\theta_i^\dagger$ being the maximizer of $|e^{\im m \theta} - \mu_i|$, as shown in the proof of \Cref{thm:Ei_norm}.
\end{proof}

\begin{remark}[Rigorous Fourier analysis for time-periodic problems]
While we have derived an LFA theory for the initial-value problem, our theory can be applied exactly or rigorously to a time-periodic MGRIT solver that employs a time-periodic coarse-grid equation (see, e.g., the solvers in References~\citenum{Gander_etal_2013,Hessenthaler_etal_2022}).
The reason for this is because the operators appearing in the time-periodic analogue of the error propagator \eqref{eq:Ei_def} are (block) circulant.
Therefore, the periodic Fourier modes $\bm{\varphi}_{\ell}(\theta)$---with $\theta$ discretely sampled at $n_t$ equidistant frequencies---are eigenfunctions for any finite $n_t$, and not only formally in the limit as $n_t \to \infty$ (as for the initial-value problem).
See Section 3.4.4 of Trottenberg~et~al.\cite{Trottenberg_etal_2001} for a discussion on the link between rigorous and local Fourier analysis.
Also note that our LFA theory bears a resemblance to the Fourier theory from Gander~et~al.\cite{Gander_etal_2013} for a time-periodic Parareal algorithm employing a time-periodic coarse-grid correction.
\end{remark}

\section{Comparison to existing literature}
\label{sec:literature_comparison}

The LFA results \eqref{eq:Ei_norm} and \eqref{eq:Ei_spectral_radius} from \cref{sec:LFA} bear a close resemblance to MGRIT convergence results in 
References~\citenum{DeSterck_etal_2019,
Dobrev_etal_2017, 
Hessenthaler_etal_2020,
Southworth_2019} obtained by other means.
One salient difference is that References~
\citenum{Dobrev_etal_2017, 
Hessenthaler_etal_2020,
Southworth_2019} analyze error propagation only over the coarse grid, following the argument that the use of ideal interpolation (or injection followed by F-relaxation) results in error dominated by its coarse-grid representation (see also \cref{sec:error_prop}).  
Further, References~\citenum{Dobrev_etal_2017, 
Hessenthaler_etal_2020,
Southworth_2019} consider the case of a finite temporal grid (finite $n_t$) and an initial-value problem. 
Our results apply rigorously (i.e., exactly) to time-periodic problems for finite $n_t$.
Our results show also that LFA convergence estimates for the initial-value problem are similar to existing convergence bounds for the initial-value problem obtained by other means.
In addition, as described earlier, our results extend the LFA theory from De~Sterck~et~al.\cite{DeSterck_etal_2019} in which Fourier symbols were derived, but they were not used in such a way as to derive analytical convergence estimates like those in \cref{sec:LFA}.

This close resemblance to References~\citenum{DeSterck_etal_2019,
Dobrev_etal_2017, 
Hessenthaler_etal_2020,
Southworth_2019} described above is perhaps best seen by considering Southworth~et~al.\cite{Southworth_etal_2021}---a companion article to Reference~\citenum{Southworth_2019}---because it gives results for fine-grid error propagation.
Specifically, adopting our notation, and imposing \Cref{ass:Phi_Psi}, Corollary 3 of  Southworth~et~al.\cite{Southworth_etal_2021} states that with FCF-relaxation (i.e., $\nu = 1$) the $\ell^2$-norm of the MGRIT error propagator \eqref{eq:E_mgrit} for finite $n_t$ is
\begin{align} 
\label{eq:Southworth_etal_2021_FCF}
\Vert {\cal E} \Vert_2 
= 
\underset{1 \leq i \leq n_x} 
\max |\lambda_i|^{m} \frac{|\lambda_i^m - \mu_i|}{1 - |\mu_i| + {\cal O}(1/N_c)} \sqrt{\frac{1 - |\lambda_i|^{2m}}{1 - |\lambda_i|^2}}.
\end{align}
Note that the formulation of ${\cal E}$ used by Southworth~et~al.\cite{Southworth_etal_2021} is different than ours given by \eqref{eq:E_mgrit}; however, they both represent the error propagator of the MGRIT algorithm and are, therefore, equivalent (accounting for different conventions for the sizes of the fine and coarse grids).
In \eqref{eq:Southworth_etal_2021_FCF}, $N_c$ is the number of time points on the coarse grid, which is slightly different from the quantity of $n_t/m$ that we have used.
It is important to stress that \eqref{eq:Southworth_etal_2021_FCF} represents a genuine equality for finite $n_t$, up to the ${\cal O}(1/N_c)$ terms, unlike the LFA approximations considered in \cref{sec:LFA} which hold for the initial-value problem in the limit of infinite $n_t$.

Recalling our LFA approximation of $\Vert
{{\cal E}}_i
\Vert_2$ given by \eqref{eq:Ei_norm}, and its relation to $\Vert {\cal E} \Vert_2$ given in \eqref{eq:E_similarity_trans}, with FCF-relaxation our LFA theory gives the following \textit{approximation} for finite $n_t$ (for clarity, we write the LFA result following from \eqref{eq:Ei_norm} with $\nu = 1$ explicitly as an approximation here),
\begin{align} \label{eq:E_norm_Southworth_compare}
\Vert {\cal E} \Vert_2
=
\underset{1 \leq i \leq n_x} \max
\Vert
{{\cal E}}_i
\Vert_2
\approx
\underset{1 \leq i \leq n_x} \max
\,
\underset{\theta \in \Theta^{\rm low}} \sup
\big\Vert
\wh{{\cal E}}_i(\theta)
\big\Vert_2
=
\underset{1 \leq i \leq n_x} \max
|\lambda_i|^{m}
\frac{|\lambda_i^m - \mu_i|}{1 - |\mu_i|}
\sqrt{ \frac{1 - |\lambda_i|^{2m}}{1 - |\lambda_i|^2} }.
\end{align}
Comparing \eqref{eq:E_norm_Southworth_compare} and \eqref{eq:Southworth_etal_2021_FCF}, they differ only by the small perturbation of ${\cal O}(1/N_c)$ appearing in \eqref{eq:Southworth_etal_2021_FCF}.
Moreover, as $n_t \to \infty$, and thus $N_c = {\cal O}(n_t/m) \to \infty$, the two expressions are equivalent. 
In other words, since \eqref{eq:Southworth_etal_2021_FCF} is valid for any value of $n_t$, it is consistent with our LFA approximation \eqref{eq:E_norm_Southworth_compare} holding exactly for $\Vert {\cal E} \Vert_2$ as $n_t \to \infty$.
This consistency provides independent verification that our LFA theory of \cref{sec:LFA} is correct.

A key distinction of our theory compared to the those mentioned above is that it can naturally connect convergence issues for MGRIT to those of spatial multigrid methods for steady-state advection-dominated problems. 
Specifically, our theory can be used to describe convergence of space-time Fourier modes, while the above theories can only be used to describe convergence of spatial Fourier modes. It is precisely this retention of temporal Fourier information that makes our theory directly comparable with fully Fourier-based convergence theories in the steady-state spatial multigrid case. This comparison is studied in \cref{sec:char_comp}.

It is also interesting to note that our LFA approximation for the spectral radius of the error propagator \eqref{eq:Ei_spectral_radius} is the same as the bounds in Dobrev~et~al.\cite{Dobrev_etal_2017} Theorem 3.3 for the $\ell^2$-norm of the coarse-grid error propagator, provided one takes $n_t \to \infty$.
In practice, MGRIT possesses a well-known exactness property that, in exactly $k = \frac{1}{\nu+1}\frac{n_t}{m}$ iterations, MGRIT converges to the exact solution of the fine-grid problem computed by sequentially time-stepping the initial condition across the time domain.
As such, the spectral radius of the MGRIT error propagator is zero (for example, by Gelfand's formula) and does not provide useful information about pre-asymptotic convergence of the algorithm.
Nonetheless, in References~\citenum{Dobrev_etal_2017,DeSterck_etal_2023_SL,DeSterck_etal_2023_MOL}, \eqref{eq:Ei_spectral_radius} was shown to be an accurate prediction to the convergence factor measured in practice before the exactness property of MGRIT becomes dominant.  
Our LFA estimate of the spectral radius (which assumes periodicity-in-time and, thus, applies rigorously to cases outside of this exactness property) can, thus, be a reasonable predictor of the MGRIT convergence factor for these ``middle iterations'' of a time-parallel solve, which often determine the overall efficiency of the MGRIT methodology.

\section{Characteristic components}
\label{sec:char_comp}

We now use the LFA theory of \cref{sec:LFA} to shed light on the poor convergence of MGRIT for advection-dominated problems.
General theoretical arguments and connections to spatial multigrid methods are presented in \cref{sec:char_comp_theory}.
Building on this, \cref{sec:char_comp_SL} theoretically analyzes MGRIT convergence for a class of semi-Lagrangian discretizations.
Supporting numerical results are given in \cref{sec:char_comp_num}, and a discussion on the implications of these results as well as potential remedies for the slow convergence are given in \cref{sec:char_comp_discussion}.

This section makes use of a number of theoretical results that are rather lengthy to derive. In order not to distract from the key messages and findings of this section, many of these theoretical results and their proofs have been placed in \cref{app:eig-est,app:conv-fac-constant} rather than in the main text.

\subsection{General theoretical arguments and connection to spatial multigrid methods}
\label{sec:char_comp_theory}

We now consider the constant-coefficient, one-dimensional advection-diffusion problem
\begin{align} \label{eq:A_differential}
{\cal A} u := 
\frac{\partial u}{\partial t} + \alpha \frac{\partial u}{\partial x} - \beta \frac{\partial^2 u}{\partial x^2} = 0, 
\quad
(x,t) \in (-1, 1) \times (0, T],
\quad
u(x, 0) = u_0(x),
\end{align}
with $\alpha >0$, $\beta \geq 0$, and $u(x,t)$ subject to periodic boundary conditions in space.
Specifically, we analyze MGRIT convergence for discretizations of this problem by employing rigorous Fourier analysis in space and the LFA theory from \cref{sec:LFA}.
To this end, suppose that the space-time discretizations $A_0$ and $A_1$ given in \eqref{eq:A0_system} and \eqref{eq:A1_system}, respectively, correspond to discretizations of ${\cal A}$ on space-time meshes using $n_x$ points in the $x$-direction separated by a distance of $h$.

Next, consider the space-time discretizations $A_0$ and $A_1$ on the semi-infinite space-time meshes $\bm{M}_{0}$ and $\bm{M}_{1}$, respectively, which are defined by
\begin{align}
\bm{M}_{\ell} := \Big\{ (x,t) = (jh, k m^{\ell} \delta t) \colon j \in 
\Big\{ -\frac{n_x}{2}, \ldots, \frac{n_x}{2}-1
\Big\}, k \in \mathbb{N}_0 \Big\}, 
\quad \ell \in \{0, 1\}.
\end{align}
On $\bm{M}_{\ell}$, we also consider the space-time Fourier modes
\begin{align} \label{eq:Fourier_modes_space-time_def}
\varrho_{\ell} (\omega, \theta) 
:= 
\chi(\omega)  \varphi_{\ell}(\theta),
\quad (\omega, \theta) \in [-\pi, \pi) \times \Theta_{\ell},
\end{align}
in which $\chi$ is a spatial Fourier mode, and $\varphi_{\ell}$ is the temporal Fourier mode from \eqref{eq:Fourier_modes_def}:
\begin{align} \label{eq:Fouier_modes_space_def}
\chi(\omega) := \exp \left( \frac{\im \omega x}{h} \right),
\quad
\varphi_{\ell}(\theta) := \exp \left( \frac{\im \theta t}{m^{\ell} \delta t} \right),
\quad
\quad (x,t) \in \bm{M}_{\ell}.
\end{align}
The spatial frequency $\omega$ discretizes $[-\pi, \pi)$ with $n_x$ points separated by a distance $h$, and, as previously, $\theta$ varies continuously in $\Theta_{\ell}$ (defined in \eqref{eq:Theta_ell_def}). 
We refer to the smoothest Fourier modes on a given spatial mesh, that is, those with frequency $\omega = {\cal O}(h)$, as being \textit{asymptotically smooth}.
%

%
%
Let $\lambda(\omega)$ and $\mu(\omega)$ be the Fourier symbols associated with spatial frequency $\omega$ of the fine- and coarse-grid time-stepping operators $\Phi$ and $\Psi$, respectively.
Observe the slight change in notation from earlier sections: The eigenvalues $\lambda(\omega)$ and $\mu(\omega)$, and functions of them, are now described by the argument $\omega$ rather than with the subscript notation $\lambda_i$ and $\mu_i$.
Then, the Fourier symbols $\wt{A}_{\ell}(\omega,\theta)$ of $A_{\ell}$ are given by (see also \eqref{eq:Ai_symbols})
\begin{align} \label{eq:char_comp_symbols_space-time}
\wt{A}_{0}(\omega,\theta) = 1 - \lambda(\omega) e^{- \im \theta}, 
\quad
\wt{A}_{1}(\omega,m \theta) = 1 - \mu(\omega) e^{- \im m \theta}.
\end{align}

The infinite-dimensional MGRIT error propagator ${\cal E}$ in \eqref{eq:E_mgrit} can be block diagonalized, analogously to how it was first block diagonalized in space in \cref{sec:error_prop_time-only}, and then in time in \cref{sec:LFA}.
Specifically, under a similarity transform we have
\begin{align} \label{eq:E_diag_char}
{\cal E} 
\;
\underset{\rm similarity\, transform} \mapsto
\;
\underset{(\omega,\theta) \in [-\pi, \pi) \times \Theta^{\rm low}}\diag \Big( \wh{{\cal E}}(\omega, \theta) \Big), 
\end{align}
where, for a fixed frequency pair $(\omega, \theta)$, $\wh{{\cal E}}(\omega, \theta) \in \mathbb{C}^{m \times m}$ is the Fourier symbol associated with the $m$ space-time modes \eqref{eq:Fourier_modes_space-time_def} having frequencies $\big(\omega, \theta + \frac{2 \pi p}{m} \big)$, $p \in \{0, \ldots, m-1\}$ (see \Cref{def:mh-harmonics}).
The set of low frequencies $\Theta^{\rm low}$ is as in \eqref{eq:Theta_low_def}.

The LFA theory from \Cref{sec:LFA} can be used to compute the spectral radius of the diagonal blocks $\wh{{\cal E}}(\omega, \theta)$ as follows.
Note that an analogous expression for the norm can be found, but the spectral radius is sufficient for our purposes.
\begin{lemma}[Spectral radius of space-time error propagator]
\label{lem:rho_E_space-time} 
For a fixed pair of frequencies $(\omega, \theta)$, the spectral radius of the associated diagonal block in \eqref{eq:E_diag_char} is
\begin{align} \label{eq:rho_E_space-time}
\rho \big( 
\wh{{\cal E}}(\omega, \theta)
\big )
=
| \lambda(\omega) |^{m \nu}
\left|
\frac{\wt{A}_{1}(\omega, m \theta) - \wt{A}^{\rm ideal}_{1}(\omega, m \theta)}{\wt{A}_{1}(\omega, m \theta)}
\right|,
\quad
(\omega,\theta) \in [-\pi, \pi) \times \Theta^{\rm low},
\end{align}
with $\wt{A}^{\rm ideal}_{1}(\omega, m \theta) = 1 - \big[\lambda(\omega)\big]^m e^{- \im m \theta}$ the Fourier symbol of the coarse-grid space-time discretization that uses the ideal coarse-grid time-stepping operator $\Psi = \Phi^m$.
\end{lemma}
\begin{proof}
This follows directly from rearranging \eqref{eq:Ei_eigen_spectral_radius} of \Cref{cor:Ei_spectral_radius}.
\end{proof}

From \eqref{eq:rho_E_space-time}, there are two mechanisms governing MGRIT convergence.
First, the factor $| \lambda(\omega) |^{m \nu}$ arising from $\nu$ sweeps of CF-relaxation in the pre-relaxation (recall \cref{sec:MGRIT}). 
This mechanism will quickly damp error modes that rapidly decay under time-stepping, but will have little effect for error modes that decay slowly under time-stepping, when ${| \lambda(\omega) | \approx 1}$.
Since MGRIT uses time-stepping as its local relaxation scheme, these modes must, instead, be targeted by the global coarse-grid correction, which corresponds to the second mechanism in \eqref{eq:rho_E_space-time}.
The effectiveness of this coarse-grid correction for a given space-time mode is determined by the corresponding \textit{relative accuracy} of $A_1$ with respect to $A^{\rm ideal}_{1}$.
This realization allows us to draw a connection between the convergence of 
MGRIT for PDE \eqref{eq:A_differential} and classical spatial multigrid for steady-state advection-diffusion problems, as we now discuss.
Consider the steady-state advection-diffusion PDE ${\cal L} u := (\alpha_x, \alpha_y) \cdot \nabla u - \beta \Delta u = 0$ in two dimensions, with fine- and coarse-grid discretizations given by $L_h$ and $L_H$, respectively. Here, $(\alpha_x, \alpha_y)$ is a two-dimensional wave-speed, and $\beta \geq 0$ is the diffusivity.
Let $(\omega_x, \omega_y)$ be a two-component spatial Fourier frequency and consider the effect of coarse-grid correction on asymptotically smooth modes, i.e., $(\omega_x, \omega_y) = ({\cal O}(h), {\cal O}(h))$.
Under reasonable assumptions on the relaxation scheme and the intergrid transfer operators, it can be shown that the two-grid convergence factor of an asymptotically smooth mode is proportional to\cite{Brandt_Yavneh_1993,Yavneh_1998}
\begin{align} \label{eq:LH_relative}
\left| \frac{\wt{L}_H( \omega_x, \omega_y ) - \wt{L}_h(\omega_x, \omega_y)}{\wt{L}_H(\omega_x, \omega_y)} \right|, 
\quad
(\omega_x, \omega_y) = ({\cal O}(h), {\cal O}(h)).
\end{align}
Notice that \eqref{eq:LH_relative} closely resembles \eqref{eq:rho_E_space-time}.
For a given $(\omega_x, \omega_y)$, the fraction \eqref{eq:LH_relative} determines the \textit{relative accuracy} of $L_H$ with respect to $L_h$, while its numerator determines the \textit{absolute accuracy} of $L_H$ with respect to $L_h$.\cite{Yavneh_1998} 
Suppose $L_H$ is derived by rediscretizing $L_h$, then, since both $L_H$ and $L_h$ are consistent with ${\cal L}$, the numerator of \eqref{eq:LH_relative} should be small as $(\omega_x, \omega_y) \to (0,0)$.
When $\beta \gg {\cal O}(h)$, so that $L_h$ and $L_H$ are close to elliptic, the denominator in \eqref{eq:LH_relative} is typically some ${\cal O}(1)$ constant, independent of $(\omega_x, \omega_y)$. Thus, two-grid convergence is fast for all asymptotically smooth modes.
However, in the strongly non-elliptic case $\beta \to 0^+$ the denominator of \eqref{eq:LH_relative} is not an ${\cal O}(1)$ constant independent of $(\omega_x, \omega_y)$.
Instead, the denominator of \eqref{eq:LH_relative} approximately vanishes for modes $(\alpha_x, \alpha_y ) \cdot \big( \frac{\omega_x}{h}, \frac{\omega_y}{h} \big) \approx 0$, which are the so-called \textit{characteristic components}.\cite{Yavneh_1998} Overall two-grid convergence is dramatically slowed by the inadequate coarse-grid correction of these components. 
Characteristic components are Fourier modes that vary slowly---relative to the grid spacing---along the direction of characteristics (as defined by the advective component of the PDE), but that are free to vary in the direction normal to characteristics. 
Characteristic components that are also asymptotically smooth (as described above) are those which vary slowly also in the direction normal to characteristics.
In the $\beta \to 0^+$ limit, on asymptotically smooth characteristic components, $\wt{L}_H$ vanishes up to small terms related to the truncation error of $L_H$, since $L_H$ is a consistent discretization of ${\cal L}$, and $\wt{\cal L}$ itself vanishes on these components.
On non-characteristic components, $\wt{L}_H$ would typically be some ${\cal O}(1)$ constant.
In other words, the relative accuracy of $L_H$ on asymptotically smooth characteristic components is much less than that for all other asymptotically smooth components.
The poor convergence of characteristic components in spatial multigrid was first described in Section 5.1 of Brandt,\cite{Brandt_1981} is covered in great detail by Yavneh,\cite{Yavneh_1998} and has been discussed in many other contexts.\cite{Brandt_Yavneh_1993,
Oosterlee_Washio_2000,
Trottenberg_etal_2001,
Wan_Chan_2003,
Bank_etal_2006,
Yavneh_Weinzierl_2012}
Based on the resemblance of \eqref{eq:rho_E_space-time} to \eqref{eq:LH_relative} and the above discussion, it is reasonable to expect that \textit{if using a rediscretized coarse-grid operator, MGRIT convergence for discretizations of PDE \eqref{eq:A_differential} will deteriorate in the advection-dominated limit, at least in part, due to a poor coarse-grid correction of asymptotically smooth characteristic components}.
Here we add the qualifier ``at least in part'' because there may also exist non-asymptotically smooth modes that are not efficiently damped by the solver (as numerical results in \cref{sec:char_comp_num} will reveal).
In the spatial, steady-state setting, when low-order discretizations are used, of all modes, only asymptotically smooth characteristic components are slow to converge, while the picture is more complicated for higher-order discretizations (see Figure 4 of Oosterlee \& Washio\cite{Oosterlee_Washio_2000}).

To make the above discussion on MGRIT convergence deterioration more precise: $A_1$ is a consistent space-time discretization of the PDE \eqref{eq:A_differential} up to some constant scaling factor, say, $\zeta$, and, thus, its Fourier symbol for modes $\omega = {\cal O}(h)$ and $\theta = {\cal O}(\delta t)$ must be consistent with that of the continuous differential operator ${\cal A}$ up to a scaling of $\zeta$.
Therefore, we have the Taylor expansion,
\begin{align}
\wt{A}_1(\omega, m \theta) \Big|_{\omega = {\cal O}(h)} 
= 
\zeta \wt{{\cal A}}(\omega, m \theta) + \textrm{h.o.t}
=
\im \zeta 
\left(
\frac{\omega \alpha}{h}
+
\frac{m \theta}{m \delta t} 
\right)
+
\zeta \beta \frac{\omega^2}{h^2}
+ \textrm{h.o.t}.
\end{align}
Here, the higher-order terms arise from the truncation error of $A_1$.
Thus, the advective space-time component of $\wt{{\cal A}}(\omega, m \theta)$ vanishes on space-time characteristic components; that is, in terms of  modes \eqref{eq:Fourier_modes_space-time_def}, are those satisfying 
\begin{align}
\big(1, \alpha \big) \cdot \left( \frac{\theta}{\delta t}, \frac{\omega}{h} \right)
=
\frac{\theta}{\delta t} + \alpha \frac{\omega}{h} \approx 0
\quad
\Longrightarrow
\quad
\theta \approx - \frac{\omega \alpha \delta t}{h}. 
\end{align}
For the discretization $\wt{A}_1(\omega, m \theta)$, we therefore have
\begin{align} \label{eq:A1tilde_on_char_comp}
\left.
\wt{A}_1\left(\omega, m \theta \approx - \frac{\omega \alpha m \delta t}{h} \right) 
\right|_{\omega = {\cal O}(h)} 
\approx
\zeta \beta \frac{\omega^2}{h^2}
+ \textrm{h.o.t}.
\end{align}
Since $\omega = {\cal O}(h)$, the right-hand side of \eqref{eq:A1tilde_on_char_comp} will be bounded away from zero so long as $\beta$ itself is. However, as $\beta \to 0^+$, the right-hand side of \eqref{eq:A1tilde_on_char_comp} vanishes up to terms related to the truncation error of $A_1$, and the relative error factor in \eqref{eq:rho_E_space-time} may be large. 

We emphasize that the result of \cref{lem:rho_E_space-time} is a key distinguishing factor our LFA-based convergence theory from existing non-LFA MGRIT convergence theories.\cite{Dobrev_etal_2017, 
Hessenthaler_etal_2020,
Southworth_2019} Specifically, since \eqref{eq:rho_E_space-time} retains both spatial and temporal Fourier information (results in References \citenum{Dobrev_etal_2017, 
Hessenthaler_etal_2020,
Southworth_2019} retain only spatial Fourier information; see also Section~\ref{sec:literature_comparison}), we have been able to compare it directly to fully Fourier-based convergence estimates developed in the spatial, steady-state setting, thereby showing that MGRIT convergence issues manifest for the same reasons as in the spatial setting.

\subsection{Semi-Lagrangian discretizations of linear advection}
\label{sec:char_comp_SL}

In this section, we investigate the general arguments made in the previous section as they apply specifically to semi-Lagrangian discretizations of the constant-wave-speed advection problem obtained by setting $\beta = 0$ in PDE \eqref{eq:A_differential}.
In De~Sterck~et~al.,\cite{DeSterck_etal_2023_MOL} we consider a similar analysis for method-of-lines discretizations of advection problems that use finite-difference discretizations in space and Runge-Kutta temporal discretizations. 

A general understanding of semi-Lagrangian discretizations for linear advection problems is useful for interpreting the results in this section, although it is not necessary. 
In essence, semi-Lagrangian methods approximate the solution at spatial mesh points at time $t_n + \delta t$ by integrating backwards along characteristics of the PDE to time $t_n$, and then approximating the PDE solution at the feet of these characteristics using polynomial interpolation (the feet of the characteristics do not, in general, intersect spatial mesh points, which is where the current approximation is known).
Excellent introductions to these discretizations can be found in References~ \citenum{Durran_2010,Falcone_Ferretti_2014}, or, alternatively, see our previous works in References~\citenum{DeSterck_etal_2023_SL,KrzysikThesis2021} for the specific formulations we consider here.
We write the time-stepping operator for a semi-Lagrangian discretization as $\Phi = {\cal S}^{(\delta t)}_p \in \mathbb{R}^{n_x \times n_x}$, indicating that the discretization has a (global) order of accuracy $p \in \mathbb{N}$, and that it uses a time-step size of $\delta t$.
So long as the spatial boundary conditions are periodic, note that ${\cal S}^{(\delta t)}_p \in \mathbb{R}^{n_x \times n_x}$ can be written as a circulant matrix regardless of the size of $\delta t > 0$. 
Since circulant matrices are unitarily diagonalized by the discrete Fourier transform, if $\Phi = {\cal S}^{(\delta t)}_p$ and $\Psi = {\cal S}^{(m \delta t)}_p$ these operators will be simultaneously unitarily diagonalizable, thus satisfying Assumption \ref{ass:Phi_Psi}.
Here, we consider only odd $p$, since in De~Sterck~et~al.\cite{DeSterck_etal_2023_SL} we were able to develop efficient MGRIT solvers for odd $p$ and not even $p$.
The key distinction between odd and even $p$ is whether the dominant term in the truncation error acts dissipatively or dispersively, respectively.\cite{DeSterck_etal_2023_SL}
It is important to note that the calculations and results given here for odd $p$ do not trivially extend to the case of even $p$, which have additional structure in a certain sense; we provide further details about this later.

We write the CFL number of the discretization ${\cal S}^{(\delta t)}_p$ as $c := \frac{\alpha \delta t}{h} > 0$. 
A key quantity in our analysis here is the fractional part of the CFL number, which we denote by 
\begin{align} \label{eq:epsilon_def}
\varepsilon^{(\delta t)} := \frac{\alpha \delta t}{h}  - \left\lfloor \frac{\alpha \delta t}{h} \right\rfloor \in [0, 1).
\end{align}
We note that semi-Lagrangian discretizations for this linear advection problem are unconditionally stable with respect to $c$.
However, the discretizations admit a form of translational symmetry in $c$ that means it will only be necessary for us to study MGRIT convergence for $c \in (0, 1]$.\footnote{If $p$ is odd, and ${\cal S}_p^{(\delta t)}$ is a discretization for some CFL number $c = \hat{c} \in (0, 1]$, then the discretization for the CFL number $c = \hat{c}+k$, $k \in \mathbb{N}$, can be written as $L^{(k)} {\cal S}_p^{(\delta t)} $, where $L^{(k)} \in \mathbb{R}^{n_x \times n_x}$ is a circulant matrix with ones on its $k$th subdiagonal. 
Note that $\big[ L^{(k)} {\cal S}_p^{(\delta t)} \big]^m  = L^{(m k)} \big[  {\cal S}_p^{(\delta t)} \big]^m$.
Analogously, if ${\cal S}_p^{(m \delta t)}$ is the rediscretized operator (i.e., the discretization for CFL number $m \hat{c}$), then the discretization for the CFL number $m (\hat{c} + k) = m \hat{c}+m k$ can be written as $L^{(m k)} {\cal S}_p^{(m \delta t)}$.
Multiplication by $L^{(m k)}$ causes a rotation of eigenvalues, but not a change in their magnitude.
Applying this fact to, for example, $\rho ({\cal E}_i)$ in \cref{cor:Ei_spectral_radius}, it is easy to see that MGRIT convergence is independent of $k$.} 
For this reason, we also sometimes abuse notation by referring to $\varepsilon^{(\delta t)}$ as the fine-grid CFL number.
From Krzysik\cite{KrzysikThesis2021} p. 103, note that the fractional component of the coarse-grid CFL number satisfies $\varepsilon^{(m \delta t)} := \frac{\alpha m \delta t}{h}  - \left\lfloor \frac{\alpha m \delta t}{h} \right\rfloor = m \varepsilon^{( \delta t)} - \left\lfloor m \varepsilon^{( \delta t)} \right\rfloor \in [0, 1).$
This brings us to the following key assumption needed in the forthcoming analysis.
\begin{assumption} \label{ass:no-mesh-intersect}
The fine-grid CFL number $c := \frac{\alpha \delta t}{h}$ is not an integer, and, furthermore, it is such that the associated coarse-grid CFL number of $m c$ is not an integer.
In other words, the domain of $\varepsilon^{(\delta t)}$ in \eqref{eq:epsilon_def} is restricted as
\begin{align} \label{eq:Upsilon_def}
\varepsilon^{(\delta t)} 
\in 
\Upsilon_m :=
(0, 1) 
\setminus 
\{ \varepsilon \colon
m \varepsilon - \lfloor m \varepsilon \rfloor = 0 \}.
\end{align}
\end{assumption}
The reason for this assumption is that if the CFL number on a given level is an integer, then characteristics on that level will intersect with the mesh, and all eigenvalues of the associated semi-Lagrangian discretization will have unit magnitude, which is not covered by our LFA theory as per \cref{ass:stability}.
We stress that this assumption is a technical necessity, and that it does not carry with it any practical implication in the sense that the MGRIT convergence factor is a continuous function of the fractional component of coarse-grid CFL number, $\varepsilon^{(m \delta t)}$. That is to say, the MGRIT convergence factor when $\varepsilon^{(m \delta t)} = 0$ will not qualitatively differ from that when $\varepsilon^{(m \delta t)} = \varpi$ for some $0 < \varpi \ll 1$.

With notation and assumptions now specified, we present our two results for this section.
First, \cref{thm:cgc_frac_SL} confirms that asymptotically smooth characteristic components do receive a poor coarse-grid correction compared to all other asymptotically smooth modes.
Second, \cref{THM:RHO-LWR-BOUND-SL} shows that this poor coarse-grid correction prevents the possibility of robust two-grid convergence with respect to problem parameters.
\begin{theorem}[Order zero coarse-grid correction] \label{thm:cgc_frac_SL}
Suppose $\Phi = {\cal S}_{p}^{(\delta t)}$, and $\Psi = {\cal S}_{p}^{(m \delta t)}$, with $p$ odd.
Suppose that \cref{ass:no-mesh-intersect} holds.
Then, asymptotically smooth Fourier modes receive a coarse-grid correction that is order $p+1$ small in $\omega$ if they are not characteristic components, while asymptotically smooth characteristic components receive an order zero coarse-grid correction in $\omega$,
\begin{align} \label{eq:cgc_frac_SL}
\left.
\frac{\big|
\widetilde{A}_{1}(\omega, m \theta) - \widetilde{A}^{\rm ideal}_{1}(\omega, m \theta)
\big|}
{\big| \widetilde{A}_{1}(\omega, m \theta) \big|}
\right|_{\omega = {\cal O}(h)}
=
\begin{cases}
{\cal O}(\omega^{p+1}), 
\quad & \displaystyle{\theta \not\approx - \frac{\omega \alpha \delta t}{h}},
\\[2ex]
{\cal O}(1),
\quad & \displaystyle{\theta \approx - \frac{\omega \alpha \delta t}{h}}.
\end{cases}
\end{align}
\end{theorem}

\begin{proof}
This proof works by estimating the numerator and denominator of \eqref{eq:cgc_frac_SL} for asymptotically smooth Fourier modes by using the semi-Lagrangian eigenvalue estimates developed in \cref{app:eig-est-SL}.

We first consider the square of the numerator in \eqref{eq:cgc_frac_SL}.
From \eqref{eq:char_comp_symbols_space-time}, recall $\widetilde{A}_{1}(\omega, m \theta) = 1 - \mu(\omega) e^{-\im m \theta}$ and, from \cref{lem:rho_E_space-time}, $\widetilde{A}_{1}^{\rm ideal}(\omega, m \theta) = 1 - \big[ \lambda(\omega) \big]^m e^{-\im m \theta}$.
Since $\Psi = {\cal S}_{p}^{(m \delta t)}$, we take $\mu(\omega)=s^{(m \delta t)}_p(\omega)$, the eigenvalue of the coarse-grid semi-Lagrangian discretization, which we will estimate using \eqref{eq:SL_eig_est}. 
Since $\Phi = {\cal S}_{p}^{(\delta t)}$, we take $\big[ \lambda(\omega) \big]^m = \big[ s_p^{(\delta t)}(\omega) \big]^m$, for which we directly use the estimate of \eqref{eq:SL_ideal_eig_est}. 
The square of the numerator of \eqref{eq:cgc_frac_SL} then becomes
\begin{align} 
\big|
\widetilde{A}_{1}(\omega, m \theta) - \widetilde{A}^{\rm ideal}_{1}(\omega, m \theta)
\big|^2
&=
\Big|
s^{(m \delta t)}_p(\omega) 
-
\big[ s_p^{(\delta t)}(\omega) \big]^m
\Big|^2
\\
&=
\left|
\exp \left( - \frac{\im \omega \alpha m \delta t}{h} \right)
\right|^2
\left|
\Big[
f_{p+1} \big( \varepsilon^{(m \delta t)} \big) 
- m f_{p+1} \big( \varepsilon^{(\delta t)} \big)
\Big] d_{p+1}(\omega)
+ 
{\cal O}(\omega^{p+2})
\right|^2.
\end{align}
%
%
Now substitute $d_{p+1}(\omega) = (-1)^{\frac{p+1}{2}} \omega^{p+1} + {\cal O}(\omega^{p+2})$, as per \eqref{eq:eig_est-D}, and then simplify to give 
\begin{align} \label{eq:rho_frac_SL_num}
\big|
\widetilde{A}_{1}(\omega, m \theta) & - \widetilde{A}^{\rm ideal}_{1}(\omega, m \theta)
\big|^2
=
\Big( \omega^{p+1} 
\Big[  
f_{p+1} \big( \varepsilon^{(m \delta t)} \big) - m f_{p+1} \big( \varepsilon^{(\delta t)} 
\big)
\Big] \Big)^2
\big(1 + {\cal O}(\omega) \big).
\end{align}
Thus, from \eqref{eq:rho_frac_SL_num}, the numerator of \eqref{eq:cgc_frac_SL} satisfies
\begin{align} \label{eq:rho_frac_SL_num_aux}
\big|
\widetilde{A}_{1}(\omega, m \theta) - \widetilde{A}^{\rm ideal}_{1}(\omega, m \theta)
\big| = \omega^{p+1} \widecheck{C}_1 + {\cal O}(\omega^{p+2}),
\end{align}
for asymptotically smooth Fourier modes, with $\widecheck{C}_1$ some positive constant.

We now consider the square of the denominator in \eqref{eq:cgc_frac_SL}.
To do so, we again take $\mu(\omega) = s^{(m \delta t)}_p(\omega)$, which we will estimate by appealing to \eqref{eq:SL_eig_est}, to give 
\begin{align} 
\big|
\widetilde{A}_{1}(\omega, m \theta)
\big|^2
&=
\Big|
1 - s^{(m \delta t)}_p(\omega) e^{- \im m \theta}
\Big|^2,
\\
\label{eq:order-zero-cf-thm-xi_def}
&=
\bigg|
1 - \exp 
\bigg( 
- \im \underbrace{m \bigg[ \theta + \frac{\omega \alpha \delta t}{h} \bigg]}_{=: \eta} 
\bigg)
\Big(
1 - 
\underbrace{
f_{p+1} \big( \varepsilon^{(m \delta t)} \big) d_{p+1} (\omega) + {\cal O}(\omega^{p+2})
}_{=: \xi}
\Big)
\bigg|^2,
\\
&=
\big|
1 - (1 - \xi) \cos \eta
+ 
\im
(1 - \xi) \sin \eta
\big|^2
=
2(1 - \cos \eta) (1 - \xi) + \xi^2.
\end{align}
Now apply $d_{p+1}(\omega) = (-1)^{\frac{p+1}{2}} \omega^{p+1} + {\cal O}(\omega^{p+2})$, as per \eqref{eq:eig_est-D}, to write $1 - \xi = 1 + {\cal O}(\omega^{p+1})$, and $\xi^2 = \Big(
\omega^{p+1}
f_{p+1} \big( \varepsilon^{(m \delta t)} \big) 
\Big)^2
\big( 1 + {\cal O}(\omega) \big)$.
This gives the the following estimate 
\begin{align} \label{eq:rho_frac_SL_den}
\big|
\widetilde{A}_{1}(\omega, m \theta)
\big|^2
=
2 \left[
1 - \cos \left( m \left[ \theta + \frac{ \omega \alpha \delta t}{h}  \right] \right) 
\right]
\Big(
1 + {\cal O}(\omega^{p+1})
\Big)
+
\Big(
\omega^{p+1}
f_{p+1} \big( \varepsilon^{(m \delta t)} \big) 
\Big)^2
\big( 1 + {\cal O}(\omega) \big).
\end{align}
Thus, for asymptotically smooth modes, the denominator of \eqref{eq:cgc_frac_SL} satisfies \\
$\big|
\widetilde{A}_{1}\left(\omega, m \theta \not\approx - \frac{\omega \alpha m \delta t}{h} \right)
\big|
=
\widecheck{C}_2 + {\cal O}(\omega)$, and 
$\big|
\widetilde{A}_{1}\left(\omega, m \theta \approx - \frac{\omega \alpha m \delta t}{h} \right)
\big|
\geq \widecheck{C}_3 \omega^{p+1}$ for some positive constants $\widecheck{C}_2,\widecheck{C}_3$.
Using this in combination with \eqref{eq:rho_frac_SL_num_aux} gives the claimed result of \eqref{eq:cgc_frac_SL}.
\end{proof}

\begin{remark}[Additional structure for dispersive discretizations]
We briefly mentioned earlier that our results for even orders $p$ do not trivially follow from those for odd orders $p$. 
If $p$ is even, then the first term in the truncation error of the discretization is imaginary, rather than real as when $p$ is odd; specifically, the $d_{p+1}(\omega)$ symbol in \eqref{eq:order-zero-cf-thm-xi_def} is imaginary.
The significance of this is that the symbol of the continuous differential operator is imaginary also, and, so, there exist characteristic components for which the symbol of the continuous differential operator cancels the first truncation error term. 
That is, if we denote these modes by $\theta = \theta_*(\omega)$, then we have $\big|
\widetilde{A}_{1}(\omega, m \theta_*(\omega))
\big| = {\cal O}(\omega^{p+2})$. 
Following through the details of the above proof, we see that the coarse-grid correction for these modes is ${\cal O}(\omega^{-1})$, meaning that they are blown up by the solver.
It is unclear, however, to what extent, if any, these modes are present in simulations of the initial-value problem for finite $n_t$, or whether they only occur for the time-periodic problem in the limit of $n_t \to \infty$. 
For example, initial experiments suggest that such modes may only arise for initial-value problems that are much more resolved in time than in space.
In any event, this insight is interesting, and it possibly has a connection with why the approaches in De~Sterck~et~al.\cite{DeSterck_etal_2023_SL} did not prove successful for dispersive discretizations.
We leave more detailed study of this topic for future work.
\end{remark}

One might anticipate that the poor coarse-grid correction of characteristic components described in \cref{thm:cgc_frac_SL} impairs overall two-grid convergence, destroying the possibility of fast convergence. 
For example, in the spatial, steady-state case, the poor coarse-grid correction of asymptotically smooth characteristic components leads to a two-grid convergence factor of one half when using a basic discretization and standard multigrid components.\cite{Brandt_1981,Brandt_Yavneh_1993,Yavneh_1998}
Indeed, \cref{THM:RHO-LWR-BOUND-SL} below shows that, if rediscretizing the semi-Lagrangian method on the coarse grid, MGRIT convergence is not robust with respect to CFL number or coarsening factor, with the two-grid convergence factor 
$\max \limits_{(\omega, \theta) \in [-\pi, \pi) \times \Theta^{\rm low}}
\rho\big( 
\widehat{{\cal E}}(\omega, \theta)
\big )$ 
being significantly larger than unity for certain combinations of these parameters. 
This is consistent with our previous numerical results in References~\citenum{DeSterck_etal_2023_SL,KrzysikThesis2021}, where MGRIT often diverged for such problems. Recall that in ${\cal O}(n_t/m)$ iterations MGRIT converges to the exact solution by sequentially propagating the initial condition across the time domain---something not captured by our LFA estimate of the spectral radius (see the end of \cref{sec:literature_comparison}). 
The aforementioned ``divergence''---which is captured by the LFA estimate---is that occurring on these ``middle iterations'' before the local relaxation scheme has any significant global impact on convergence.
Finally, note that although the context is different, the below theorem is also consistent with the numerical results reported in the final paragraph of Section 6.3.2 of Schmitt~et~al.,\cite{Schmitt_etal_2018} wherein Parareal convergence degraded substantially in the hyperbolic limit of a nonlinear advection-diffusion problem when using a semi-Lagrangian discretization on the coarse grid and $m = {\cal O}(10^4)$.
\begin{theorem}[Convergence factor lower bound] \label{THM:RHO-LWR-BOUND-SL}
Suppose $\Phi = {\cal S}_{p}^{(\delta t)}$, and $\Psi = {\cal S}_{p}^{(m \delta t)}$, with $p$ odd.
Suppose that \cref{ass:no-mesh-intersect} holds.
Then, over all space-time Fourier modes, the MGRIT spectral radius \eqref{eq:rho_E_space-time} satisfies the following lower bound independent of the number of CF-relaxations $\nu$,
\begin{align} \label{eq:rho-SL-lwr-bnd}
\max \limits_{(\omega, \theta) \in [-\pi, \pi) \times \Theta^{\rm low}}
\rho\big( 
\widehat{{\cal E}}(\omega, \theta)
\big )
&\geq
\widecheck{\rho}_p \big( \varepsilon^{(\delta t)} \big)
\big( 1 + {\cal O}(h) \big),
\end{align}
where,
\begin{align} \label{eq:rho-check}
\widecheck{\rho}_p \big( \varepsilon^{(\delta t)} \big)
:=
m \frac{f_{p+1} \big( \varepsilon^{(\delta t)} \big)}
{f_{p+1} \big( \varepsilon^{(m \delta t)} \big) }
-
1,
\quad
\varepsilon^{(\delta t)}
\in
\Upsilon_m,
\end{align}
with $f_{p+1}$ the degree $p+1$ polynomial defined in \eqref{eq:fpoly_def}.

Furthermore, $\widecheck{\rho}_p \big( \varepsilon^{(\delta t)} \big)$ is larger than unity for an interval of fine-grid CFL numbers $\varepsilon^{(\delta t)}$ that quickly covers $\Upsilon_m$ in \eqref{eq:Upsilon_def} as $m$ increases.
More precisely, for $m \geq 2$,
\begin{align} \label{eq:rho-check-interval}
\widecheck{\rho}_p \big( \varepsilon^{(\delta t)} \big)
>
1
\quad
\textrm{when}
\quad
\varepsilon^{(\delta t)} \in \Upsilon_m \cap \left( \frac{2}{3 m}, 1 - \frac{2}{3 m} \right).
\end{align}
\end{theorem}

\begin{proof}
Since asymptotically smooth characteristic components represent only a subset of all Fourier modes, we immediately get the following lower bound on the worst-case spectral radius in \eqref{eq:rho_E_space-time}
\begin{align} \label{eq:rho-SL-lwr-bnd-aux}
\max \limits_{(\omega, \theta) \in [-\pi, \pi) \times \Theta^{\rm low}}
\rho \big( 
\widehat{{\cal E}}(\omega, \theta)
\big )
&\geq
\left.
\rho \left(\widehat{{\cal E}} \left(\omega, \theta = - \frac{\omega \alpha \delta t}{h} \right)\right)\right|_{\omega = {\cal O}(h)}.
\end{align}
Now we evaluate the (square of the) right-hand side of this equation.
First, consider the factor $\big| \lambda(\omega) \big|^{m \nu}$ in the spectral radius \eqref{eq:rho_E_space-time} that arises from $\nu \in \mathbb{N}_0$ sweeps of CF-relaxation. 
Invoking the eigenvalue estimate for the ideal coarse-grid operator from \eqref{eq:SL_ideal_eig_est}, we have for asymptotically smooth Fourier modes 
\begin{align}
\big| \lambda(\omega) \big|^{m \nu}
=
\Big| \big[ s_p^{(\delta t)}(\omega) \big]^m \Big|^{\nu}
=
\bigg|
\exp\bigg(- \frac{\im \omega \alpha m \delta t}{h} \bigg)
\bigg|^{\nu}
\big| 1 + {\cal O}(\omega^{p+1}) \big|^{\nu}
=
1 + {\cal O}(h^{p+1}).
\end{align}
Notice that, as described in the previous section, this relaxation effectively does nothing for the convergence of asymptotically smooth Fourier modes.

Thus, by invoking the numerator \eqref{eq:rho_frac_SL_num}, and the denominator \eqref{eq:rho_frac_SL_den}, the square of the right-hand side of \eqref{eq:rho-SL-lwr-bnd-aux} is
\begin{align}
\begin{split}
\left.
\rho^2 \left(\widehat{{\cal E}} \left(\omega, \theta = - \frac{\omega \alpha \delta t}{h} \right)\right)\right|_{\omega = {\cal O}(h)}
&=
\big[ 1 + {\cal O}(h^{p+1}) \big]^2
\frac{\Big( \omega^{p+1} 
\Big[  
f_{p+1} \big( \varepsilon^{(m \delta t)} \big) - m f_{p+1} \big( \varepsilon^{(\delta t)} 
\big)
\Big] \Big)^2
\big(1 + {\cal O}(h) \big)}{\Big(
\omega^{p+1}
f_{p+1} \big( \varepsilon^{(m \delta t)} \big) 
\Big)^2
\big( 1 + {\cal O}(h) \big)},
\end{split}
\\
\label{eq:rho-SL-lwr-bnd-aux2}
&\quad=
\left(
m \frac{f_{p+1} \big( \varepsilon^{(\delta t)} \big)}
{f_{p+1} \big( \varepsilon^{(m \delta t)} \big) }
-
1
\right)^2
\frac{1 + {\cal O}(h)}{1 + {\cal O}(h)}
=
\Big( \widecheck{\rho}_p \big( \varepsilon^{(\delta t)} \big) \Big)^2 
\, 
\frac{1 + {\cal O}(h)}{1 + {\cal O}(h)}.
\end{align}
Next, use the geometric expansion $\frac{1 + {\cal O}(h)}{1 + {\cal O}(h)} = 1 + {\cal O}(h)$, then take the square root of both sides of \eqref{eq:rho-SL-lwr-bnd-aux2} and apply that $\sqrt{1 + {\cal O}(h)} = 1 + {\cal O}(h)$.
Plugging the result into \eqref{eq:rho-SL-lwr-bnd-aux} gives the claimed result of \eqref{eq:rho-SL-lwr-bnd}.
When taking the square root of \eqref{eq:rho-SL-lwr-bnd-aux2}, note that $\widecheck{\rho}_p \big( \varepsilon^{(m \delta t)} \big)  > 0$ for all $\varepsilon^{(\delta t)} \in \Upsilon_m$, which follows by combining $\big| m f_{p+1} \big( \varepsilon^{(\delta t)} \big) \big| > \big| {f_{p+1} \big( \varepsilon^{(m \delta t)} \big) } \big|$ (see Krzysik\cite{KrzysikThesis2021} Lemma 4.4),
and 
$\textrm{sign} 
\big( {f_{p+1} \big( \varepsilon^{(m \delta t)} \big) } \big) = \textrm{sign} \big ( m f_{p+1} \big( \varepsilon^{(\delta t)} \big) \big)$ (see Krzysik\cite{KrzysikThesis2021} Lemma B.1).
The remainder of the proof regarding when $ \widecheck{\rho}_p > 1$ is technical and can be found in \cref{app:conv-fac-constant}.
\end{proof}

Note that there may also exist $\varepsilon^{(\delta t)} \in \Upsilon_m$ outside of the interval considered in \eqref{eq:rho-check-interval} for which $\widecheck{\rho}_p \big( \varepsilon^{(\delta t)} \big) > 1$; see \cref{fig:rho_vs_CFL-SL-lwr-bnd}.
However, this interval is sufficient for demonstrating our primary point which is that the interval of $\varepsilon^{(\delta t)}$ for which $\widecheck{\rho}_p \big( \varepsilon^{(\delta t)} \big) > 1$ quickly covers $\Upsilon_m$ as $m$ grows.
%

\subsection{Numerical results}
\label{sec:char_comp_num}

%
\begin{figure}[t!]
\centerline{
\includegraphics[scale=0.425]{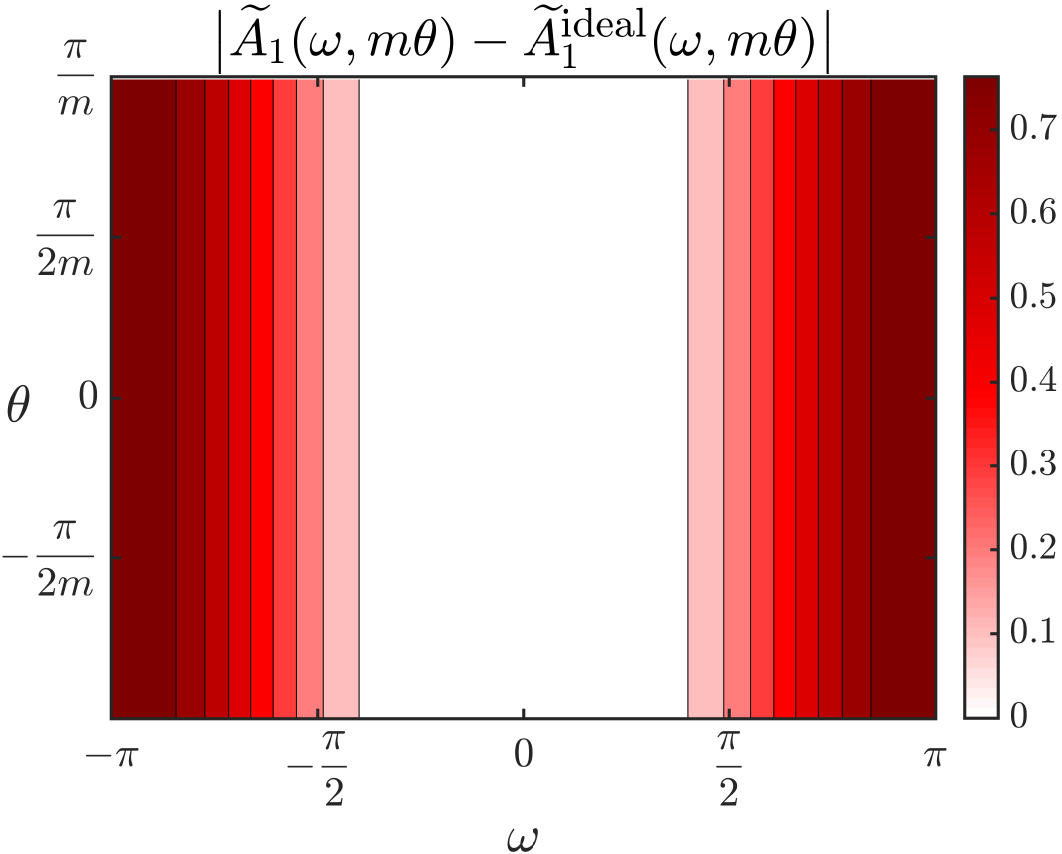}
\quad\quad
\includegraphics[scale=0.425]{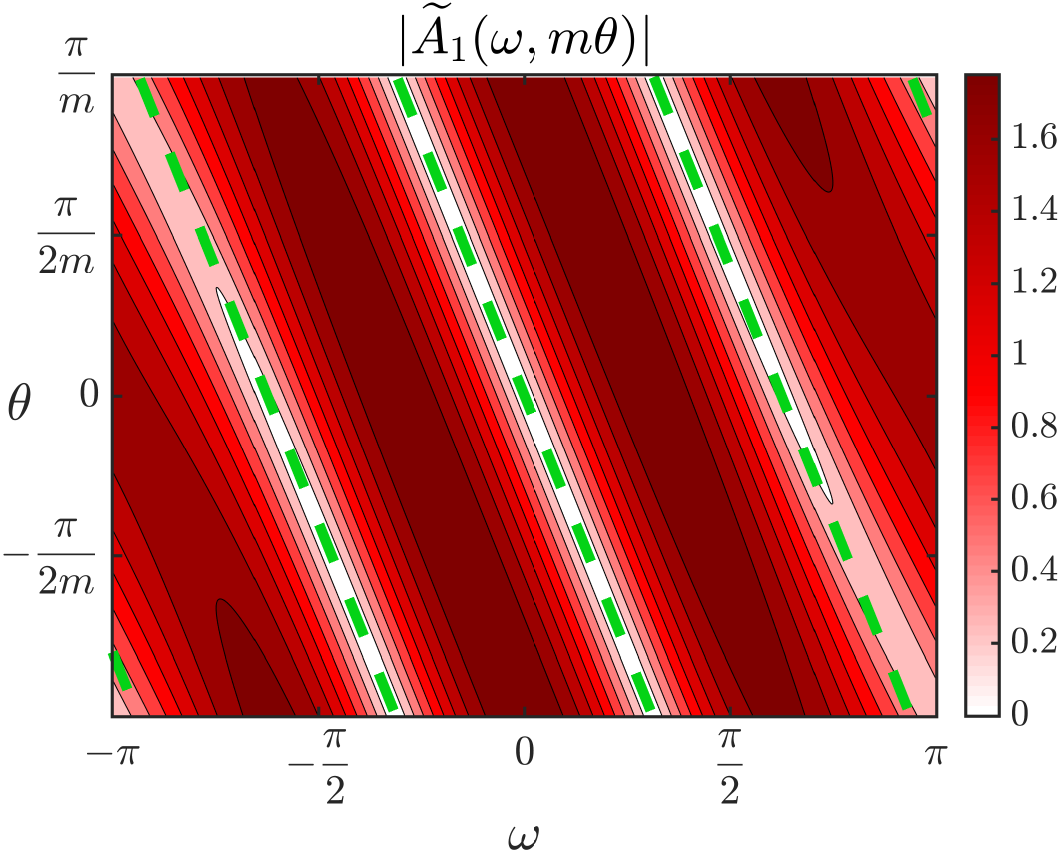}
}
\vspace{2ex}
\hspace{-1ex}
\centerline{
\includegraphics[scale=0.425]{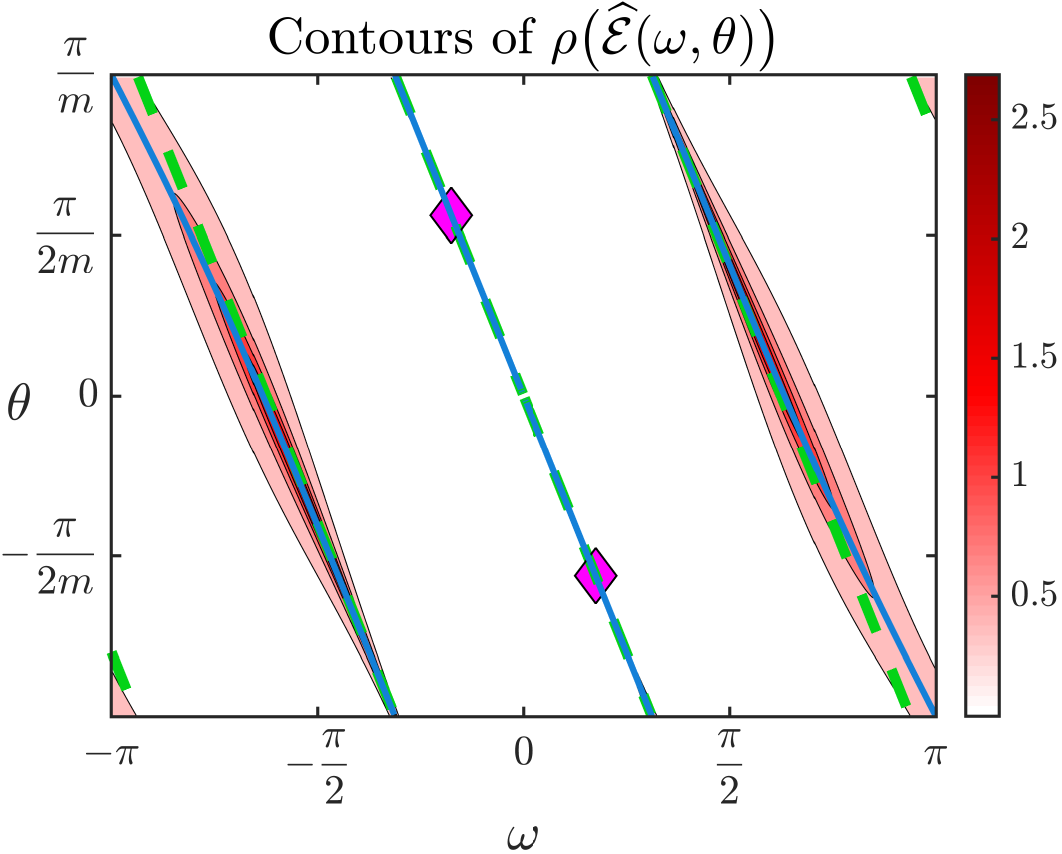}
\hspace{6ex}
\includegraphics[scale=0.425]{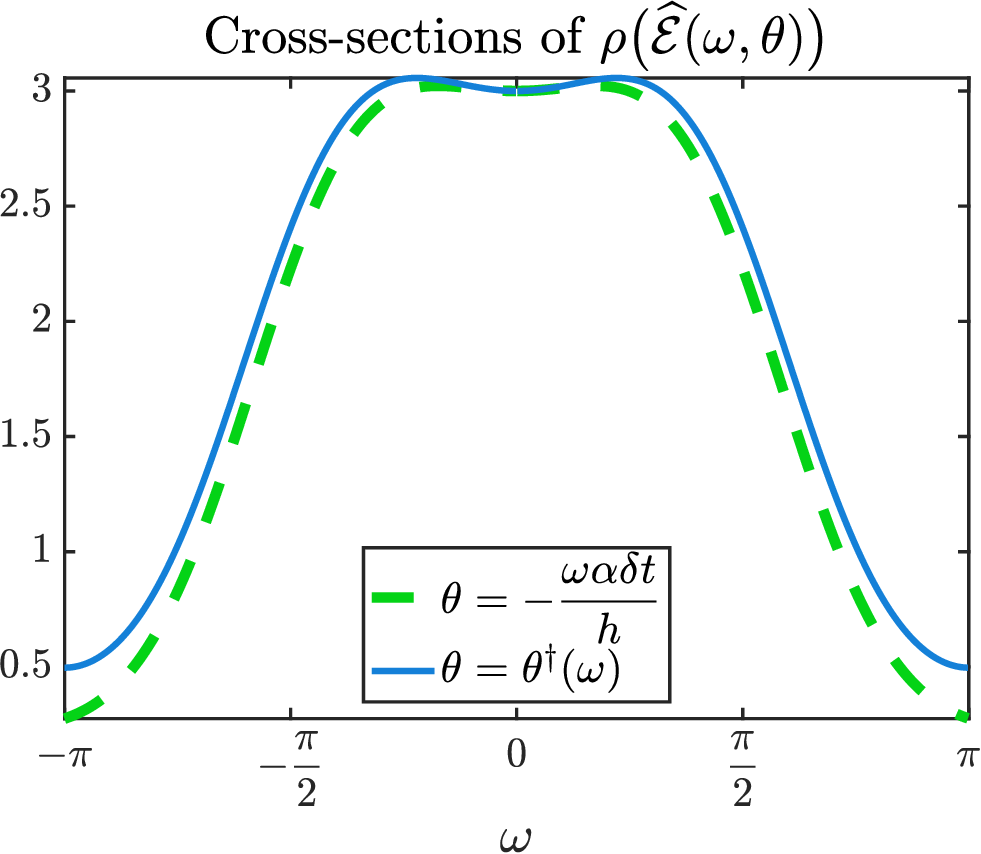}
}
\caption{
Plots in Fourier frequency space of quantities relating to the LFA spectral radius in \cref{lem:rho_E_space-time} for a $p = 3$ semi-Lagrangian discretization of  linear advection using rediscretization on the coarse grid.
The coarsening factor is $m = 4$, FCF-relaxation is used, and the CFL number is $c = 0.8$.
Top left: Difference between symbols of coarse-grid operator and ideal coarse-grid operator.
Top right: Symbol of coarse-grid operator.
Bottom left: Contour of spectral radius evaluated with $1024$ points in $\omega$ and $1024/m$ points in $\theta$, with its maxima over the whole space marked with magenta diamonds.
Bottom right: Cross-sections of the spectral radius along the green and blue lines pictured in the bottom left panel.
Dashed green lines mark characteristic components with frequency $\theta(\omega) = - \frac{\omega \alpha \delta t}{h}$, 
while blue lines mark the slowest converging modes $\theta(\omega) = \theta^{\dagger}(\omega) = \frac{1}{m} \arg \mu(\omega)$ (see \cref{thm:Ei_norm}).
\label{fig:SL_space-time}
}
\end{figure}

In this section, we present numerical results to support the general arguments made in \cref{sec:char_comp_theory} on MGRIT convergence and characteristic components, and those specific to semi-Lagrangian schemes made in \cref{sec:char_comp_SL}.
In \cref{fig:SL_space-time} a visual picture of what is described in \cref{sec:char_comp_theory,sec:char_comp_SL} is given. 
Specifically, for the order $p=3$ semi-Lagrangian discretization, we arbitrarily pick a CFL number of $c = 0.8$ and coarsening factor of $m = 4$, and plot in Fourier space the key quantities from \cref{lem:rho_E_space-time} governing MGRIT convergence for this problem.
To make the contour plots, we sample $\omega \in [-\pi, \pi)$ with 1024 points, and $\theta \in \Theta^{\rm low}$ with 1024$/m$ points.

First, consider the absolute error of the coarse-grid discretization, which is shown in the top left panel of \cref{fig:SL_space-time}.
This quantity is small for $\omega \approx 0$, arising from the fact that both $A_1$ and $A_1^{\rm ideal}$ are consistent with the continuous differential operator ${\cal A}$ from \eqref{eq:A_differential} as $\omega \to 0$ and, thus, are consistent with one another in this regime.
Next, consider the magnitude of the coarse-grid symbol, which is shown in the top right panel of \cref{fig:SL_space-time}.
In this plot, we overlay with dashed green lines characteristic components with frequency $\theta = - \frac{\omega \alpha \delta t}{h}$.
As expected, the symbol approximately vanishes for asymptotically smooth characteristic components (with $\omega \approx 0$), but is bounded away from zero for all other asymptotically smooth components.
Again, this occurs since $A_1$ is consistent with ${\cal A}$ as $\omega \to 0$, and ${\cal A}$ vanishes on characteristic components.
Notice that even for some non-asymptotically smooth characteristic components, that is, $\theta \approx - \frac{\omega \alpha \delta t}{h}$ with $\omega \gg {\cal O}(h)$, this symbol appears to approximately vanish. 

Next, consider the spectral radius itself, which is shown in the bottom row of \cref{fig:SL_space-time}, beginning with the contours shown in the left panel.
On this plot, we again overlay characteristic components, and we also overlay with blue lines the set of slowest converging modes $\theta = \theta^{\dagger} (\omega)$ as given in \cref{thm:Ei_norm}.
Not surprisingly, the spectral radius is large where $|\wt{A}_1|$ approximately vanishes.
As reflected by the green lines overlapping with the blue lines, it appears that for temporally smooth components, with $\omega \approx 0$, the convergence is poorest for the characteristic components.
In fact, \textit{all modes with $\omega \approx 0$ are rapidly damped except those that are characteristic components}; it is not particularly obvious from this plot, however, what the value of the function is along the green and blue lines.
Therefore, we plot the values of $\rho\big( \wh{{\cal E}}(\omega, \theta ) \big)$ along these green and blue lines in the bottom right panel of \cref{fig:SL_space-time}.
From the bottom right panel, it becomes apparent that smooth characteristic components having $\theta = - \frac{\omega \alpha \delta t}{h}$ do indeed diverge, with $\rho \approx 3 > 1$ for these modes.
Finally, notice that, while convergence is certainly poor for asymptotically smooth characteristic components, convergence is poor also for many characteristic components which are not asymptotically smooth.
This is an example of why we say that poor MGRIT convergence stems \textit{at least in part} from the poor coarse-grid correction of asymptotically smooth characteristic components.
%

\begin{figure}[t!]
\centerline{
\includegraphics[scale=0.425]{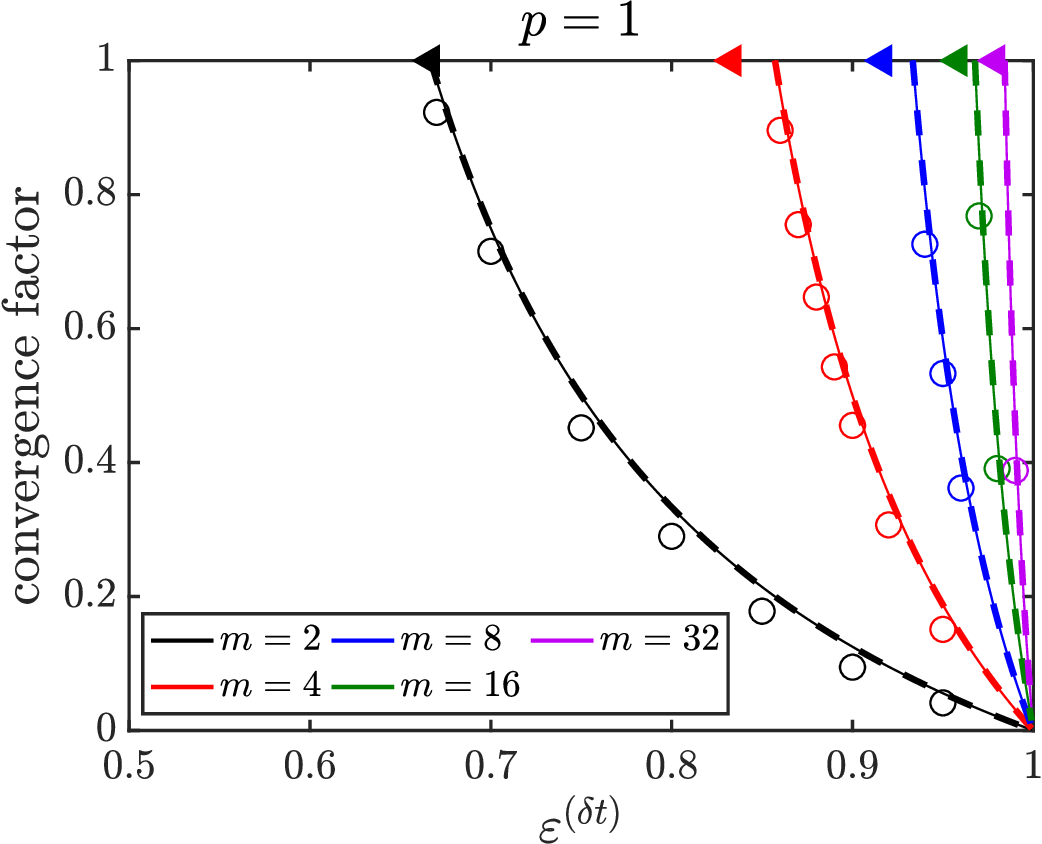}
\quad\quad\quad
\includegraphics[scale=0.425]{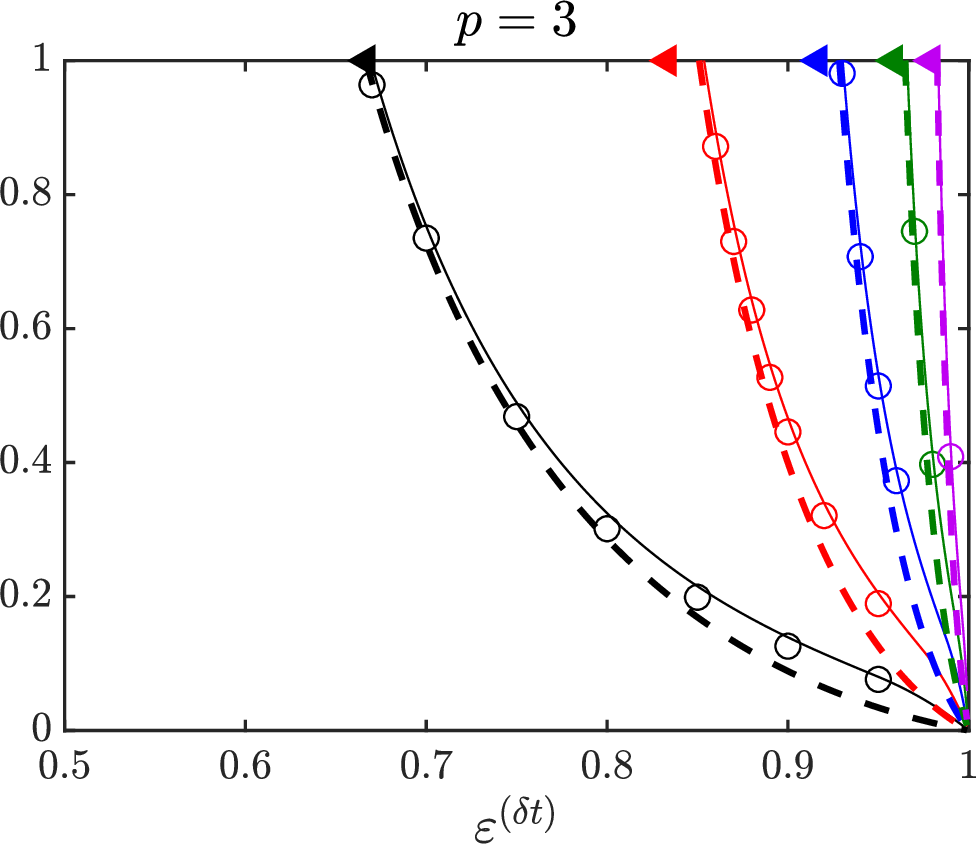}
}
\caption{
Numerical confirmation of \cref{THM:RHO-LWR-BOUND-SL}:
Convergence factors for the semi-Lagrangian discretizations $\Phi = {\cal S}_p^{(\delta t)}$, with $p = 1$ (left), and $p = 3$ (right), as a function of fractional fine-grid CFL number $\varepsilon^{(\delta t)}$ when rediscretizing on the coarse grid, $\Psi = {\cal S}_p^{(m \delta t)}$. 
MGRIT uses FCF-relaxation ($\nu = 1$), and a coarsening factor of $m$.
Thin solid lines are the LFA convergence factor $\max \limits_{(\omega, \theta) \in [-\pi, \pi) \times \Theta^{\rm low}}
\rho\big( 
\widehat{{\cal E}}(\omega, \theta)
\big )$
obtained by discretizing $\omega$ with $128$ points.
Thick dashed lines are the function $\widecheck{\rho}_p \big( \varepsilon^{(\delta t)} \big)$ from \eqref{eq:rho-check} that acts as a lower bound on the LFA convergence factor (see \eqref{eq:rho-SL-lwr-bnd}).
Filled triangle markers are $\varepsilon^{(\delta t)}  = 1 - \frac{2}{3 m}$, the right-hand end point of the interval in \eqref{eq:rho-check-interval} on which $\widecheck{\rho}_p \big( \varepsilon^{(\delta t)} \big)$ is shown to exceed unity.
Open circle markers are experimentally measured effective convergence factors of MGRIT on a finite interval $t \in (0, T]$ taken from Figure 2 of De~Sterck~et~al.\cite{DeSterck_etal_2023_SL}
\label{fig:rho_vs_CFL-SL-lwr-bnd}
}
\end{figure}

We now provide numerical evidence that the LFA spectral radius of MGRIT on infinite time domains given in \cref{cor:Ei_spectral_radius} is an accurate predictor of the effective MGRIT convergence factor on finite time domains, and we also provide numerical verification of its lower bound given in \cref{THM:RHO-LWR-BOUND-SL}.
\Cref{fig:rho_vs_CFL-SL-lwr-bnd} shows effective and LFA MGRIT convergence factors as a function of the CFL number for semi-Lagrangian schemes of orders $p = 1,3$ (note that the quantities shown in \cref{fig:rho_vs_CFL-SL-lwr-bnd} are symmetric over $\varepsilon^{(\delta t)} \in (0, 1)$, so, to better highlight the details we show them only for $\varepsilon^{(\delta t)} \in \big(\frac{1}{2}, 1\big)$).
The thin solid lines in the plots are the worst-case LFA convergence factor $\max \limits_{(\omega, \theta) \in [-\pi, \pi) \times \Theta^{\rm low}}
\rho\big( 
\widehat{{\cal E}}(\omega, \theta)
\big )$, which we evaluate by discretizing $\omega$ with $2^7$ points.
Circle markers overlaid on the plots are numerical data from Figure 2 of De~Sterck~et~al.,\cite{DeSterck_etal_2023_SL} and correspond to convergence factors measured on final MGRIT iterations before stopping criteria are reached and, thus, represent effective MGRIT convergence factors for a finite interval $t \in (0, T]$. 
These MGRIT tests used a uniformly random initial iterate, and a space-time domain discretized with $n_x \times n_t = 2^7 \times 2^{15}$ points.
First, notice that the LFA MGRIT convergence factor provides a very good approximation to the experimentally measured effective convergence factors.
Second, notice that convergence is not robust for these problems, with the solver diverging for most combinations of $m$ and $c$, as indicated by the convergence factor being larger than unity.

Also shown in \cref{fig:rho_vs_CFL-SL-lwr-bnd} as the thick dashed lines is the lower bound on the LFA estimate of the spectral radius (thin solid lines) given in \cref{THM:RHO-LWR-BOUND-SL}. 
This lower bound appears very tight, essentially sitting on top of the LFA estimate in the $p = 1$ case. 
This suggests for $p = 1$ that it is indeed the inadequate coarse-grid correction of asymptotically smooth characteristic components described in \cref{thm:cgc_frac_SL} that determines the overall two-grid convergence factor, at least for problems with convergence factors visible in the plot (i.e., those less than one).
In contrast, the lower bound is slightly less tight in the $p = 3$ case, suggesting that it is not always these modes that determine overall two-grid performance.
This is consistent with what is shown in the example in the bottom right panel of \cref{fig:SL_space-time}.
Nevertheless, the lower bound from \cref{THM:RHO-LWR-BOUND-SL} does  accurately capture the interval for which the overall two-grid convergence factor exceeds unity, and it is clearly capable of demonstrating in the process that MGRIT convergence is not robust for this problem with respect to either $m$ or $c$.
%

\subsection{Improved coarse-grid operators}
\label{sec:char_comp_discussion}
Slow convergence of characteristic components has been widely studied in the spatial multigrid case, and, so, it is useful to consider the solutions proposed there, and to understand to what extent they can be applied in the MGRIT context.
While several potential fixes have been proposed, for which detailed summaries can be found in Section 5 of Yavneh\cite{Yavneh_1998} and Section 7 of  Trottenberg~et~al.,\cite{Trottenberg_etal_2001} it is important to note that none appear to yield multigrid convergence that is as efficient and robust as that for elliptic problems in general. 

One proposed remedy is to use so-called downstream relaxation,\cite{Brandt_Yavneh_1992,
Yavneh_1998,
Yavneh_etal_1998,
Trottenberg_etal_2001} 
which essentially amounts to carrying out relaxation in an order that propagates errors and/or residuals downstream, along characteristics. 
But unfortunately, in the multigrid-in-time context considered here, downstream relaxation is nothing more than sequential time-stepping, the exact procedure that we are trying to avoid!
Other ideas considered include residual overweighting,\cite{Brandt_Yavneh_1993} and accelerating multigrid convergence using a Krylov method.\cite{Oosterlee_Washio_2000}

Considering \eqref{eq:LH_relative}, arguably the most effective way to improve convergence for smooth characteristic components is to increase the absolute accuracy of the coarse-grid operator with respect to the fine-grid operator.
To do this, the coarse-grid operator needs to be designed so that its truncation error better approximates that of the fine-grid operator, at least for smooth characteristic components.
This key idea was first introduced in Yavneh,\cite{Yavneh_1998} and was also later considered by Bank~et~al.\cite{Bank_etal_2006}
In the MGRIT context, that is, \cref{lem:rho_E_space-time}, this amounts to modifying the coarse-grid time-stepping operator so that its truncation error better approximates that of the ideal coarse-grid operator.
An idea related to this was explored in Danieli \& MacLachlan\cite{Danieli_MacLachlan_2023} for explicit temporal discretizations of nonlinear hyperbolic PDEs, where $\Psi$ was directly constructed to match terms in a Taylor expansion of $\Phi^m$, though stability issues limit that approach to very small fine-grid CFL numbers.
A similar idea was also recently used by Vargas et al.\cite{Vargas_etal_2023} to develop effective MGRIT coarse-grid operators for chaotic problems, including the Lorentz system and the Kuramoto--Sivashinsky equation.

This idea of approximately matching truncation errors formed the basis of our recent works in References \citenum{DeSterck_etal_2023_SL,DeSterck_etal_2023_MOL,DeSterck_etal_2023_nonlin}. Perhaps the most relevant of these works to the present context is De~Sterck~et~al.,\cite{DeSterck_etal_2023_SL} where we designed very effective MGRIT coarse-grid operators for semi-Lagrangian discretizations of variable-wave-speed linear advection problems in multiple space dimensions.
The coarse-grid operators in De~Sterck~et~al.\cite{DeSterck_etal_2023_SL} consist of first applying a rediscretized coarse-grid operator, followed by a truncation error correction, which has to be applied in an implicit sense (i.e., it requires a linear solve) to ensure stability.
In fact, the following theorem confirms that the coarse-grid operator from equation (3.20) of De~Sterck~et~al.\cite{DeSterck_etal_2023_SL} does provide a better coarse-grid correction for asymptotically smooth characteristic components than rediscretization does, as considered in \cref{thm:cgc_frac_SL}.
Furthermore, the fact that this coarse-grid operator yields fast and robust convergence for all modes can be seen from results in Figure 3 of De~Sterck~et~al.,\cite{DeSterck_etal_2023_SL} which plots the LFA spectral radius of MGRIT, analogous to those plots in \cref{fig:rho_vs_CFL-SL-lwr-bnd} for rediscretization.
\begin{theorem}[Improved coarse-grid correction]
Suppose $\Phi = {\cal S}_{p}^{(\delta t)}$, with $p$ odd.
Further, suppose that $\Psi 
= 
\Big( I - 
\big[ 
f_{p+1} \big( \varepsilon^{(m \delta t)} \big) 
-
m f_{p+1} \big( \varepsilon^{(\delta t)} \big)
\big] {\cal D}_{p+1} \Big)^{-1} {\cal S}_{p}^{(m \delta t)}$, 
where $f_{p+1}$ is as in \eqref{eq:fpoly_def}, 
$\varepsilon^{(\delta t)}$ is as in \eqref{eq:epsilon_def},
and ${\cal D}_{p+1}$ is as explained in \cref{app:eig-est-SL}.
Suppose the fine-grid CFL number is not an integer.
Then, asymptotically smooth Fourier modes receive a coarse-grid correction that is order $p+2$ small in $\omega$ if they are not characteristic components, while asymptotically smooth characteristic components receive an order one coarse-grid correction in $\omega$, 
\begin{align} \label{eq:rho_frac_SL-mod}
\left.
\frac{\big|
\widetilde{A}_{1}(\omega, m \theta) - \widetilde{A}^{\rm ideal}_{1}(\omega, m \theta)
\big|}
{\big| \widetilde{A}_{1}(\omega, m \theta) \big|}
\right|_{\omega = {\cal O}(h)}
=
\begin{cases}
{\cal O}(\omega^{p+2}), 
\quad & \displaystyle{\theta \not\approx - \frac{\omega \alpha \delta t}{h}},
\\[2ex]
{\cal O}(\omega),
\quad & \displaystyle{\theta \approx - \frac{\omega \alpha \delta t}{h}}.
\end{cases}
\end{align}
\end{theorem}

\begin{proof}
This proof proceeds analogously to that of \cref{thm:cgc_frac_SL}.
We first consider the square of the denominator in \eqref{eq:rho_frac_SL-mod}.
For eigenvalue estimates of $\Psi$ we use \eqref{eq:eig-est-SL-mod} of \cref{lem:eig-est-SL-mod} to obtain
\begin{align} \label{eq:rho_frac_SL-mod_den}
\begin{split}
\big|
\widetilde{A}_{1}(\omega, m \theta)
\big|^2
&=
2 \left[
1 - \cos \left( m \left[ \theta + \frac{ \omega \alpha \delta t}{h}  \right] \right) 
\right]
\Big(
1 + {\cal O}(\omega^{p+1})
\Big)
\\
&\quad\quad\quad+
\Big(
\omega^{p+1}
m f_{p+1} \big( \varepsilon^{(\delta t)} \big) 
\Big)^2
\big( 1 + {\cal O}(\omega) \big).
\end{split}
\end{align}
Just as in \cref{thm:cgc_frac_SL}, we have for asymptotically smooth modes that \\
$\big|
\widetilde{A}_{1}\left(\omega, m \theta \not\approx - \frac{\omega \alpha m \delta t}{h} \right)
\big|
=
\widecheck{C}_4 + {\cal O}(\omega)$, and 
$\big|
\widetilde{A}_{1}\left(\omega, m \theta \approx - \frac{\omega \alpha m \delta t}{h} \right)
\big|
\geq \widecheck{C}_5 \omega^{p+1}$ for some positive constants $\widecheck{C}_4, \widecheck{C}_5$.

Now consider the square of the numerator in \eqref{eq:rho_frac_SL-mod}.
For eigenvalue estimates of the ideal coarse-grid operator we use \eqref{eq:SL_ideal_eig_est}, and for $\Psi$ we again use \eqref{eq:eig-est-SL-mod}.
This gives
\begin{align} 
\begin{split}
\big|
\widetilde{A}_{1}(\omega, m \theta) - \widetilde{A}^{\rm ideal}_{1}(\omega, m \theta)
\big|^2
&=
\left|
\exp \left( - \frac{\im \omega \alpha m \delta t}{h} \right)
\right|^2 
\bigg|
\Big(1
-
m f_{p+1} \big( \varepsilon^{(\delta t)} \big) 
d_{p+1}(\omega)
+ {\cal O}(\omega^{p+2})
\Big)
\\ 
&\quad-
\Big(
1
- m f_{p+1} \big( \varepsilon^{(\delta t)} \big) d_{p+1}(\omega)
+ 
{\cal O}(\omega^{p+2})
\Big)
\bigg|^2.
\end{split}
\end{align}
For all asymptotically smooth Fourier modes we have $\big|
\widetilde{A}_{1}(\omega, m \theta) - \widetilde{A}^{\rm ideal}_{1}(\omega, m \theta)
\big| = {\cal O}(\omega^{p+2})$;
combining this with the above result for $\big|
\widetilde{A}_{1}(\omega, m \theta)
\big|$ results in \eqref{eq:rho_frac_SL-mod}.
\end{proof}

\section{Conclusions}
\label{sec:conclusions}

We have developed an LFA convergence theory for the two-level iterative parallel-in-time method MGRIT, which also applies to the Parareal method.
The theory is presented in closed form, with analytical expressions for the LFA results for the norm and spectral radius of MGRIT applied to infinite time domains, approximating the effective convergence factor of MGRIT on finite time domains.
Final convergence results from our LFA theory closely resemble several existing results from the literature that were derived via alternate means; however, our LFA framework is uniquely placed to shed light on the poor convergence of MGRIT for advection-dominated problems when using the standard approach of rediscretizing on the coarse grid. 
In essence, this is because our theory can be used to describe convergence of space-time Fourier modes, allowing for straightforward comparison with existing Fourier-based analyses in the spatial, steady-state multigrid setting used to describe convergence issues there.\cite{Brandt_1981,Yavneh_1998}

We find that when using a rediscretized coarse-grid operator for MGRIT, convergence issues arise that are due, at least in part, to an inadequate coarse-grid correction of a subset of asymptotically smooth Fourier modes known as characteristic components.
It is well known that an inadequate coarse-grid correction of these same modes also gives rise to poor convergence for spatial multigrid methods when applied to steady-state advection-dominated problems.
For a class of semi-Lagrangian discretizations of linear advection problems, we prove that this inadequate coarse-grid correction precludes robust MGRIT convergence when using rediscretization on the coarse grid.
In De~Sterck~et~al.,\cite{DeSterck_etal_2023_MOL} we again use this framework to investigate MGRIT convergence for classical method-of-lines type discretizations of advection problems.

Leveraging this connection to the understanding of slow convergence for spatial multigrid solvers is key to developing improved parallel-in-time solvers for advection-dominated problems, by generalizing to the MGRIT setting techniques used to improve convergence.
This approach has already proven successful in our work for semi-Lagrangian discretizations\cite{DeSterck_etal_2023_SL} and, also for method-of-lines discretizations for both linear advection equations,\cite{DeSterck_etal_2023_MOL} and for nonlinear hyperbolic PDEs with shocks.\cite{DeSterck_etal_2023_nonlin}

One direction for future work is to extend the closed-form LFA theory developed here to the case of three-level MGRIT, to better understand the impacts of inexact coarse-grid solves on characteristic components.
Another possible direction would be to use the LFA framework in similar ways as has been done for spatial multigrid, by optimizing algorithmic parameters  within relaxation and coarse-grid correction.\cite{Brown_etal_2021}

\section*{Acknowledgments}

We are very grateful to Irad Yavneh pointing us to his work in Reference~\citenum{Yavneh_1998}, and for initially explaining to us the possibility of a connection between MGRIT and classical spatial multigrid methods for advection-dominated problems from the perspective of characteristic components.
Constructive comments from referees are also greatly acknowledged.

\bibliography{mgrit-lfa-char-components}

\appendix

\section{Error propagation of the constant mode}
\label{app:sm:lem:const_mode}

\begin{lemma} \label{sm:lem:const_mode}
Suppose there exists an index $i_* \in \{1, \ldots, n_x\}$ such that $\lambda_{i_*}^m = \mu_{i_*}$. 
Then, the associated error propagator \eqref{eq:Ei_def} is
\begin{align} \label{sm:eq:Ei_const_mode_zero}
{\cal E}_{i_*} = 0.
\end{align}
\end{lemma}
\begin{proof}
Observe that the error propagator ${\cal E}_{i_*}$ in \eqref{eq:Ei_def} corresponds to a scalar initial-value problem with fine-grid time-stepping operator $\lambda_{i_*}$, and coarse-grid time-stepping operator $\mu_{i_*}$ (see, e.g., the expressions for $A_{0,i}$ and $A_{1,i}$ in \eqref{eq:A0i_def} and \eqref{eq:A1i_def}).
Because $\mu_{i_*} = \big( \lambda_{i_*} \big)^m$ this scalar problem for eigenmode $i_*$ uses an ideal coarse-grid operator.
The result \eqref{sm:eq:Ei_const_mode_zero} follows immediately by recalling that MGRIT converges to the exact solution in a single iteration when using the ideal coarse-grid operator.
\end{proof}
%

\section{Interpolation matrices and error propagators of relaxation}
\label{sm:sec:interp_and_relax}

The purpose of this section is to derive convenient representations for the interpolation and relaxation components of the MGRIT error propagator ${\cal E}$ in \eqref{eq:E_mgrit}.

To begin, it is useful to define a CF-interval as a C-point and the $m-1$ F-points that follow it. We denote the $k$th such CF-block of a space-time vector ${\bm{u} = \big( \bm{u}_0, \bm{u}_1, \ldots, \bm{u}_{n_t-1} \big)^\top \in \mathbb{R}^{n_t n_x }}$ by
\begin{align} \label{sm:eq:block_notation}
\wh{\bm{u}}_k = \big( \bm{u}_{km}, \bm{u}_{km+1}, \ldots \bm{u}_{km+m-1} \big)^\top \in \mathbb{R}^{m n_x},
\quad 
k \in \{0, \ldots, n_t/m-1\}.
\end{align}
Furthermore, the $j$th vector in the $k$th CF-block is denoted as $\wh{\bm{u}}_{k,j} = \bm{u}_{km + j}$,\\ $j \in \{0, \ldots, m-1\}$.

We now consider the interpolation operator.
Recall that the interpolation operator $P$ in \eqref{eq:E_mgrit} is based on injection.
That is, $P$ maps the $k$th C-point variable $\bm{u}_{km} \in \mathbb{R}^{n_x}$ as 
\begin{align}
\bm{u}_{km} \mapsto
\begin{bmatrix}
\bm{u}_{km} \\
\bm{0} \\
\vdots \\
\bm{0}
\end{bmatrix}
=
\big( \bm{e}_1 \otimes I_{n_x} \big) \bm{u}_{km}
\in 
\mathbb{R}^{mn_x},
\end{align}
in which $\bm{e}_1 \in \mathbb{R}^m$ is the canonical (column-oriented) basis vector in the first direction, and is unrelated to the algebraic error $\bm{e}^{(p)}$.
Therefore, the global space-time interpolation operator that acts on all $n_t/m$ C-point variables simultaneously is the block diagonal matrix
\begin{align} \label{sm:eq:interp_def}
P = I_{n_t/m} \otimes \big( \bm{e}_1 \otimes I_{n_x} \big).
\end{align}
Note that the transpose of injection, which acts as the restriction operator in \eqref{eq:E_mgrit}, is simply 
\begin{align}
P^\top
=  
I_{n_t/m} \otimes \big( \bm{e}_1^\top \otimes I_{n_x} \big).
\end{align}

Next we consider the more complicated cases of the iteration operators for F- and C-relaxation in \eqref{eq:E_mgrit}. 
To this end, suppose we have some approximation to the true solution of the system $A_0 \bm{u} = \bm{b}$ denoted by $\bm{w}^{(0)} \approx \bm{u}$.
Then, as described in \cref{sec:MGRIT}, F-relaxation generates a new approximation $\bm{w}^{(0)} \mapsto \bm{w}^{(1)}$ such that C-point values of $\bm{w}^{(1)}$ are unchanged from those of $\bm{w}^{(0)}$, and F-point values of $\bm{w}^{(1)}$ have zero residual.
In other words, F-relaxation represents an exact solve for the F-point variables of the system $A_0 \bm{w}^{(1)} = \bm{b}$, where C-point variables in $\bm{w}^{(1)}$ are equal to those in $\bm{w}^{(0)}$.
Therefore, on the $k$th CF-block, $k \in \{0, 1, \ldots, n_t/m-1\}$, F-relaxation can be expressed as the update
\begin{equation}
\begin{aligned} \label{sm:eq:F-relax_update}
\wh{\bm{w}}_{k,0}^{(1)} 
&= 
\wh{\bm{w}}_{k,0}^{(0)}, 
\\
\wh{\bm{w}}_{k,j}^{(1)} 
&=
\Phi^j \wh{\bm{w}}_{k,0}^{(0)} + \wh{\bm{b}}_{k,j}, 
\quad j \in \{1, \ldots, m-1\}.
\end{aligned}
\end{equation}
Replacing the approximations $\bm{w}^{(q)}$ via the error equations $\bm{w}^{(q)} = \bm{u} - \bm{e}^{(q)}$ leads to
\begin{equation} \label{sm:eq:F-relax_derive_temp}
\begin{aligned} 
\wh{\bm{u}}_{k,0} - \wh{\bm{e}}_{k,0}^{(1)} 
&= 
\wh{\bm{u}}_{k,0} - \wh{\bm{e}}_{k,0}^{(0)}, 
\\
\wh{\bm{u}}_{k,j} - \wh{\bm{e}}_{k,j}^{(1)} 
&=
\Phi^j \left( \wh{\bm{u}}_{k,0} - \wh{\bm{e}}_{k,0}^{(0)} \right) + \wh{\bm{b}}_{k,j}
=
\left( \Phi^j \wh{\bm{u}}_{k,0}+ \wh{\bm{b}}_{k,j} \right) -  \Phi^j  \wh{\bm{e}}_{k,0}^{(0)},
\end{aligned}
\end{equation}
for $\quad j \in \{1, \ldots, m-1\}$, with the last equality following from the linearity of $\Phi$. Note that the exact solution $\bm{u}$ is a fixed-point of the update \eqref{sm:eq:F-relax_update}, $\wh{\bm{u}}_{k,j} = \Phi^j \wh{\bm{u}}_{k,0}+ \wh{\bm{b}}_{k,j},\, j \in \{1, \ldots, m-1\}$. Therefore, update \eqref{sm:eq:F-relax_derive_temp} can be recast in terms of the error as
\begin{equation}
\begin{aligned}
\label{sm:eq:error_F-relax_local_intermediate}
\wh{\bm{e}}_{k,0}^{(1)} 
&= 
\wh{\bm{e}}_{k,0}^{(0)}, 
\\
\wh{\bm{e}}_{k,j}^{(1)} 
&=
\Phi^j \wh{\bm{e}}_{k,0}^{(0)}, 
\quad j \in \{1, \ldots, m-1\}.
\end{aligned}
\end{equation}
%
%
Expressed in CF-block form, update \eqref{sm:eq:error_F-relax_local_intermediate} is
\begin{align} 
\label{sm:eq:error_F-relax_local}
\wh{\bm{e}}_k^{(1)} 
= 
\begin{bmatrix}
I & 0 &\cdots &0 \\
\Phi & 0 &\cdots &0 \\
\vdots & \vdots &\cdots &\vdots \\
\Phi^{m-1} & 0 &\cdots &0
\end{bmatrix}
\wh{\bm{e}}_k^{(0)}
=
\left[ \bm{e}_1^\top \otimes v(\Phi) \right] \wh{\bm{e}}_k^{(0)},
\end{align}
with the function $v$ defined in \eqref{eq:Ln_v_defs}.
Thus, based on \eqref{sm:eq:error_F-relax_local}, the error propagator for F-relaxation that acts on all CF-blocks $k = 0,\ldots,n_t/m-1$ simultaneously is the block diagonal matrix
\begin{align}
\label{sm:eq:SF}
S^{\rm F} 
&= 
I_{n_t/m} \otimes \left[ \bm{e}_1^\top \otimes v(\Phi) \right].
\end{align}

Now consider C-relaxation. Recall from \cref{sec:MGRIT} that C-relaxation leaves F-point values unchanged and updates C-points such that they have zero residuals. 
In other words, C-relaxation represents an exact solve for the C-point variables of the system $A_0 \bm{w}^{(1)} = \bm{b}$, where F-point variables in $\bm{w}^{(1)}$ are equal to those in $\bm{w}^{(0)}$.
Therefore, the C-point update on variables in the $k$th CF-interval, $k \in \{{1, \ldots, n_t/m-1\}}$, may be written
\begin{equation}
\begin{aligned}
\label{sm:eq:error_C-relax_local_intermediate}
\wh{\bm{w}}_{k,0}^{(1)} 
&= 
\Phi \wh{\bm{w}}_{k-1,m-1}^{(0)} + \wh{\bm{b}}_{k,0}, 
\\
\wh{\bm{w}}_{k,j}^{(1)} 
&=
\wh{\bm{w}}_{k,j}^{(0)}, 
\quad j \in \{1, \ldots, m-1\}.
\end{aligned}
\end{equation}
Using the same logic as for F-relaxation above, \eqref{sm:eq:error_C-relax_local_intermediate} can be rewritten as the following update on the error,
\begin{align}
\label{sm:eq:error_C-relax_local}
\wh{\bm{e}}_k^{(1)} 
= 
\begin{bmatrix}
0 & 0 &\cdots & \Phi \\
0 & 0 &\cdots &0 \\
\vdots & \vdots &\cdots & \vdots \\
0 & 0 &\cdots &0
\end{bmatrix}
\wh{\bm{e}}_{k-1}^{(0)}
+
\begin{bmatrix}
0 \\
& I_n \\
& & \ddots \\
& & & I_n
\end{bmatrix}
\wh{\bm{e}}_{k}^{(0)}.
\end{align}

Recall that in our formulation of ${\cal E}$ given by \eqref{eq:E_mgrit}, a C-relaxation  is always followed by an F-relaxation to create a CF-relaxation.
Combining \eqref{sm:eq:error_F-relax_local} and \eqref{sm:eq:error_C-relax_local}, it is easy  to show that the error update for a CF-relaxation is simply
\begin{align} 
\label{sm:eq:error_CF-relax_local}
\wh{\bm{e}}_k^{(1)} 
=
\begin{bmatrix}
0  &\cdots & 0 & \Phi \\
0  &\cdots & 0 & \Phi^2 \\
\vdots &\cdots &\vdots  &\vdots \\
0  &\cdots &0 & \Phi^{m} 
\end{bmatrix} 
\wh{\bm{e}}_{k-1}^{(0)}
=
\left[ \bm{e}_m^\top \otimes v(\Phi) \Phi \right] \wh{\bm{e}}_{k-1}^{(0)},
\end{align}
Therefore, the error propagator for CF-relaxation that acts on all CF-blocks \\${k = 0, \ldots, n_t/m-1}$ simultaneously is the block lower bidiagonal matrix
\begin{align}
\label{sm:eq:SCF}
S^{\rm CF} 
&= 
L_{n_t/m} \otimes \left[ \bm{e}_m^\top \otimes v(\Phi) \Phi \right].
\end{align}

\section{Derivations of Fourier symbols}
\label{app:eigenmatrix-proofs}

\subsection{Proof of \Cref{lem:eigmat-F-relax} (Fourier symbol of F-relaxation)}
\label{sec:proof:eigmat-F-relax}

\begin{proof}
From the representation of $\wh{S}^{\rm F}_i (\theta)$ given by \eqref{eq:SFi_eigen_def}, its $(p,q)$th element is
\begin{align} \label{eq:SFi_eigen_pq_def}
\Big[ \wh{S}^{\rm F}_i (\theta) \Big]_{p,q}
=
\Big\langle
\bm{\varphi}_0\big( \theta + \tfrac{2 \pi p}{m} \big), S^{\rm F}_i \bm{\varphi}_0\big( \theta + \tfrac{2 \pi q}{m} \big)
\Big\rangle,
\quad
p,q \in \{0, \ldots, m-1\}.
\end{align}
We begin by considering the vector in the right-hand side of this inner product. 
Using the definition of $S^{\rm F}_i$ given in \eqref{eq:SFi_def}, the $j$th element of the vector $S^{\rm F}_i \bm{\varphi}_0\big( \theta + \tfrac{2 \pi q}{m} \big)$ is
\begin{align}
\Big[S^{\rm F}_i \bm{\varphi}_0\big( \theta + \tfrac{2 \pi q}{m} \big) \Big]_j
&=
\lambda_i^{j \bmod m}
\exp
\Big[ \tfrac{\im}{\delta t} \big( \theta + \tfrac{2 \pi q}{m} \big) \big( j -j \bmod m \big) \delta t \Big],
\\
\label{eq:zeta_i_def}
&=
{\underbrace{\Big[ \lambda_i \exp \Big( -\im \big( \theta + \tfrac{2 \pi q}{m} \big) \Big) \Big]}_{=: \zeta_i(\theta,q)}}^{j \bmod m}
\Big[ \bm{\varphi}_0\big( \theta + \tfrac{2 \pi q}{m} \big) \Big]_j.
\end{align}
Now, considering the inner product \eqref{eq:SFi_eigen_pq_def} and the definition of $\zeta_i(\theta,q)$ in~\eqref{eq:zeta_i_def}, some algebra gives
\begingroup
\allowdisplaybreaks
\begin{align}
\Big[ \wh{S}^{\rm F}_i (\theta) \Big]_{p,q}
&=
\lim \limits_{n_t \to \infty}
\frac{1}{n_t}
\sum
\limits_{k = 0}^{n_t-1}
\big[\zeta_i(\theta, q) \big]^{k \bmod m}
\exp
\Big[\tfrac{-\im}{\delta t} \big( \theta + \tfrac{2 \pi p}{m} \big) k \delta t \Big]
\exp
\Big[ \tfrac{\im }{\delta t} \big( \theta + \tfrac{2 \pi q}{m} \big) k \delta t \Big],
\\
&=
\lim \limits_{n_t \to \infty}
\frac{1}{n_t}
\sum
\limits_{k = 0}^{n_t-1}
\big[\zeta_i(\theta, q) \big]^{k \bmod m}
\exp
\Big( \tfrac{2 \pi \im k}{m} (q-p) \Big), 
\\
&=
\lim \limits_{n_t \to \infty}
\frac{1}{n_t}
\sum \limits_{r = 0}^{m-1}
\Bigg(
\big[\zeta_i(\theta, q) \big]^r
\Bigg[
\sum
\limits_{k = 0}^{\frac{n_t}{m}-1}
\exp
\Big( \tfrac{2 \pi \im (km+r)}{m} (q-p) \Big)
\Bigg]
\Bigg), 
\\
&=
\lim \limits_{n_t \to \infty}
\frac{1}{n_t}
\sum \limits_{r = 0}^{m-1}
\Bigg(
\big[\zeta_i(\theta, q) \big]^r
\Bigg[
\exp \Big( \tfrac{2 \pi \im r}{m} (q-p) \Big)
\sum
\limits_{k = 0}^{\frac{n_t}{m}-1}
\exp
\Big( 2 \pi \im k (q-p) \Big)
\Bigg]
\Bigg), 
\\
&=
\lim \limits_{n_t \to \infty}
\frac{1}{n_t}
\sum \limits_{r = 0}^{m-1}
\bigg(
\big[\zeta_i(\theta, q) \big]^r
\bigg[
\frac{n_t}{m}
\exp
\Big( \tfrac{2 \pi \im r}{m} (q-p) \Big)
\bigg]
\bigg),
\\
\label{eq:SF_eigen_calc_int}
&=
\frac{1}{m}
\sum \limits_{r = 0}^{m-1}
\Big[\zeta_i(\theta, q)
\exp
\Big( \tfrac{2 \pi \im}{m} (q-p) \Big)
\Big]^r,
\end{align}
\endgroup
Substituting $\zeta_i(\theta, q)$ from \eqref{eq:zeta_i_def} into \eqref{eq:SF_eigen_calc_int}  and simplifying the geometric sum gives
\begin{align}
\label{eq:SFi_eigen_calc_end}
\Big[ \wh{S}^{\rm F}_i (\theta) \Big]_{p,q}
&=
\frac{1}{m}
\sum \limits_{r = 0}^{m-1}
\Big[\lambda_i
\exp
\Big( -\im \Big( \theta + \tfrac{2 \pi p}{m} \Big) \Big)
\Big]^r
=
\frac{1}{m}
\frac{1 - \big( \lambda_i e^{-\im \theta} \big)^m}{1 - \exp
\big( -\im \big( \theta + \tfrac{2 \pi p}{m} \big) \big)}.
\end{align}
Observe from \eqref{eq:SFi_eigen_calc_end} that $\big[ \wh{S}^{\rm F}_i (\theta) \big]_{p,q}$ does not depend on the column index $q$, but only the row index $p$.
This means that $\wh{S}^{\rm F}_i (\theta)$ can be expressed as an outer product of the form $\wh{S}^{\rm F}_i (\theta) = \tfrac{1}{m} \big[ 1 - \big( \lambda_i e^{-\im \theta} \big)^m \big] \bm{a} \bm{1}^\top$, for some vector $\bm{a}$ whose $p$th element is $1/\big[1 - \exp
\big( -\im \big( \theta + \tfrac{2 \pi p}{m} \big) \big)\big]= 1/\wt{A}_{0,i} \big(\theta + \tfrac{2 \pi p}{m} \big)$, where $\wt{A}_{0,i}(\theta)$ is the Fourier symbol of $A_{0,i}$ given in \eqref{eq:Ai_symbols}.
Therefore, $\bm{a} = \big[ \wh{A}_{0,i}(\theta) \big]^{-1} \bm{1}$, in which $\wh{A}_{0,i}(\theta)$ is the diagonal matrix in \eqref{eq:A0and1i_eigen} holding the Fourier symbols of the harmonics.
This gives the result \eqref{eq:SFi_eigen} for $\wh{S}^{\rm F}_i (\theta)$.
\end{proof}

\subsection{Proof of \cref{cor:SF_idempotence} (Idempotence of F-relaxation)}
\label{sec:proof:idempotence-of-F-relax}

\begin{proof}
Consider first the eigenvalue claims. Since $\wh{S}^{\rm F}_i (\theta) \in \mathbb{C}^{m \times m}$ has a rank of one (see \eqref{eq:SFi_eigen}), it has $m-1$ zero eigenvalues, and if it is idempotent then its one remaining eigenvalue must be one.

Using the expression for $\wh{S}^{\rm F}_i (\theta)$ given by \eqref{eq:SFi_eigen}, its square is
\begin{align}
\wh{S}^{\rm F}_i (\theta) \, \wh{S}^{\rm F}_i (\theta)
&=
\Big(
c(\theta)
\big[ \wh{A}_{0,i}(\theta) \big]^{-1}
\bm{1}
\bm{1}^\top
\Big)
\Big(
c(\theta)
\big[ \wh{A}_{0,i}(\theta) \big]^{-1}
\bm{1}
\bm{1}^\top
\Big), 
\\
&=
c(\theta)
\big[ \wh{A}_{0,i}(\theta) \big]^{-1}
\bm{1}
\Big(
\bm{1}^\top
c(\theta)
\big[ \wh{A}_{0,i}(\theta) \big]^{-1}
\bm{1}
\Big)
\bm{1}^\top, 
\\
\label{eq:SFi_idempotence_proof_temp}
&=
\Big(
\bm{1}^\top
c(\theta)
\big[ \wh{A}_{0,i}(\theta) \big]^{-1}
\bm{1}
\Big)
\wh{S}^{\rm F}_i (\theta).
\end{align}
We now show that the inner product in \eqref{eq:SFi_idempotence_proof_temp} is one, and, thus, that $\wh{S}^{\rm F}_i (\theta)$ is idempotent.
From Equation~\eqref{eq:SFi_eigen_calc_end} in the proof of~\cref{lem:eigmat-F-relax}, the column vector $c(\theta) \big[ \wh{A}_{0,i}(\theta) \big]^{-1} \bm{1}$ can be written as
\begin{align}
\label{eq:geo_sum_identity}
\Big[ c(\theta) \big[ \wh{A}_{0,i}(\theta) \big]^{-1}
\bm{1} \Big]_p
=
\frac{1}{m}
\sum
\limits_{r = 0}^{m-1}
\Big[
\lambda_i 
\exp\big( -\im \big( \theta + \tfrac{2\pi p}{m} \big) \big)
\Big]^r.
\end{align}
Thus, with \eqref{eq:geo_sum_identity}, the inner product in \eqref{eq:SFi_idempotence_proof_temp} can be written as
\begin{align}
\bm{1}^\top
c(\theta)
\big[ \wh{A}_{0,i}(\theta) \big]^{-1}
\bm{1}
&=
\frac{1}{m}
\sum 
\limits_{p = 0}^{m-1}
\bigg(
\sum 
\limits_{r = 0}^{m-1}
\Big[
\lambda_i 
\exp\big( -\im \big( \theta + \tfrac{2\pi p}{m} \big) \big)
\Big]^r
\bigg),
\\
&=
\frac{1}{m}
\sum 
\limits_{r = 0}^{m-1}
\Big[
\lambda_i
e^{-\im \theta} 
\Big]^r
\sum 
\limits_{p = 0}^{m-1}
\Big[ 
\exp\big( \tfrac{-2\pi \im r}{m} \big) 
\Big]^{p}, 
\\
\label{eq:kron_delta_use}
&=
\frac{1}{m}
\sum 
\limits_{r = 0}^{m-1}
\Big[
\lambda_i
e^{-\im \theta} 
\Big]^r
m \delta_{r,0},
\\
\label{eq:SCFi_latter_function}
&=1.
\end{align}
In \eqref{eq:kron_delta_use}, $\delta_{r,0}$ denotes the Kronecker delta function, and it has arisen from simplifying the geometric sum over $p$ in the previous equation.
\end{proof}
%

\subsection{Proof of \cref{lem:eigmat-prerelax} (Fourier symbol of pre-relaxation)}
\label{sec:proof:eigmat-prerelax}

\begin{proof}
We begin by computing the Fourier symbol for CF-relaxation. From \eqref{eq:SCFi_eigen_def}, the $(p,q)$th element of $\wh{S}^{\rm CF}_i (\theta)$ is equal to
\begin{align} \label{eq:SCFi_eigen_pq_def}
\Big[ \wh{S}^{\rm CF}_i (\theta) \Big]_{p,q}
=
\Big\langle
\bm{\varphi}_0\big( \theta + \tfrac{2 \pi p}{m} \big), S^{\rm CF}_i \bm{\varphi}_0\big( \theta + \tfrac{2 \pi q}{m} \big)
\Big\rangle,
\quad
p,q \in \{0, \ldots, m-1\}.
\end{align}
Using the expression for $S^{\rm CF}_i$ given in \eqref{eq:SCFi_def}, the $j$th element of the vector $S^{\rm CF}_i \bm{\varphi}_0\big( \theta + \tfrac{2 \pi q}{m} \big)$ can be written as
\begin{align}
\Big[S^{\rm CF}_i \bm{\varphi}_0\big( \theta + \tfrac{2 \pi q}{m} \big) \Big]_j
&=
\lambda_i \lambda_i^{j \bmod m}
\exp
\Big[ \tfrac{\im }{\delta t} \big( \theta + \tfrac{2 \pi q}{m} \big) \big( j - 1 - j \bmod m \big) \delta t \Big],
\\
&=
\zeta_i(\theta, q)
\big[\zeta_i(\theta, q) \big]^{j \bmod m}
\Big[ \bm{\varphi}_0\big( \theta + \tfrac{2 \pi q}{m} \big) \Big]_j,
\\
\label{eq:SCFi-SFi-pq_temp}
&=
\zeta_i(\theta, q)
\Big[S^{\rm F}_i \bm{\varphi}_0\big( \theta + \tfrac{2 \pi q}{m} \big) \Big]_j,
\end{align}
where the function $\zeta_i(\theta,q)$ is as in \eqref{eq:zeta_i_def}, and $\big[\zeta_i(\theta, q) \big]^{j \bmod m} \big[ \bm{\varphi}_0\big( \theta + \tfrac{2 \pi q}{m} \big) \big]_j$ has been replaced with $\Big[S^{\rm F}_i \bm{\varphi}_0\big( \theta + \tfrac{2 \pi q}{m} \big) \Big]_j$ by using \eqref{eq:zeta_i_def}.

Therefore, from \eqref{eq:SCFi_eigen_pq_def} and \eqref{eq:SCFi-SFi-pq_temp}, the following simple relationship holds between the $(p,q)$th element of the Fourier symbols of CF- and F-relaxation:
\begin{align} \label{eq:SCFi-SFi-pq}
\Big[ \wh{S}^{\rm CF}_i (\theta) \Big]_{p,q}
=
\zeta_i(\theta, q)
\Big[ \wh{S}^{\rm F}_i (\theta) \Big]_{p,q}.
\end{align}
From its definition (see \eqref{eq:zeta_i_def}), $\zeta_i(\theta, q) = 1 - \wt{A}_{0,i}\big( \theta + \tfrac{2 \pi q}{m} \big)$, where $\wt{A}_{0,i}(\theta)$ defined in \eqref{eq:Ai_symbols} is the Fourier symbol of ${A}_{0,i}$.
Thus, from \eqref{eq:SCFi-SFi-pq} and \eqref{eq:SFi_eigen}, the Fourier symbol of CF-relaxation may be expressed as
\begin{align} \label{eq:SCFi_eigen_repeated}
\wh{S}^{\rm CF}_i (\theta) 
=
\wh{S}^{\rm F}_i (\theta) 
\Big[
I - \wh{A}_{0,i}(\theta)
\Big]
=
c(\theta)
\big[ \wh{A}_{0,i}(\theta) \big]^{-1}
\bm{1}
\bm{1}^\top
\Big[
I - \wh{A}_{0,i}(\theta)
\Big],
\end{align}
in which $\wh{A}_{0,i}(\theta)$ is the diagonal matrix holding the Fourier symbols of the harmonics (see \eqref{eq:A0and1i_eigen}).
This proves \eqref{eq:SCFi_eigen}, the first claim of the lemma.

Now we consider taking powers of the Fourier symbol. Exploiting the rank-1 structure of $\wh{S}^{\rm CF}_i (\theta)$ in \eqref{eq:SCFi_eigen_repeated}, any power $\nu \in \mathbb{N}$ of the matrix can be computed as
\begingroup
\allowdisplaybreaks
\begin{align}
\big[ \wh{S}^{\rm CF}_i (\theta) \big]^{\nu}
&=
\bigg(
c(\theta)
\big[ \wh{A}_{0,i}(\theta) \big]^{-1}
\bm{1}
\bm{1}^\top
\Big[
I - \wh{A}_{0,i}(\theta)
\Big]
\bigg)^{\nu}, 
\\
&=
c(\theta)
\big[ \wh{A}_{0,i}(\theta) \big]^{-1}
\bm{1}
\bigg(
\bm{1}^\top
\Big[
I - \wh{A}_{0,i}(\theta)
\Big]
c(\theta)
\big[ \wh{A}_{0,i}(\theta) \big]^{-1}
\bm{1}
\bigg)^{\nu-1}
\bm{1}^\top
\Big[
I - \wh{A}_{0,i}(\theta)
\Big],
\\
\label{eq:SCFi_nu_powers_formula}
&=
\Big(
\bm{1}^\top
c(\theta)
\big[ \wh{A}_{0,i}(\theta) \big]^{-1}
\bm{1}
-
c(\theta)
\bm{1}^\top
\bm{1}
\Big)^{\nu-1}
\wh{S}^{\rm CF}_i (\theta),
\quad
\nu \in \mathbb{N}.
\end{align}
\endgroup
The remaining problem is thus one of evaluating the two inner products in \eqref{eq:SCFi_nu_powers_formula}: 
$
\bm{1}^\top
c(\theta)
\big[ \wh{A}_{0,i}(\theta) \big]^{-1}
\bm{1}$,
and
$
c(\theta)
\bm{1}^\top
\bm{1}$.
The first inner product was already shown to be one in the proof of \cref{cor:SF_idempotence} (specifically, see \eqref{eq:SCFi_latter_function}), and the second is simply $c(\theta)
\bm{1}^\top
\bm{1} = 
m c(\theta)$.
Thus, the function raised to the power $\nu-1$ in \eqref{eq:SCFi_nu_powers_formula} is simply
\begin{align} \label{eq:powers_function}
\bm{1}^\top
c(\theta)
\big[ \wh{A}_{0,i}(\theta) \big]^{-1}
\bm{1}
-
c(\theta)
\bm{1}^\top
\bm{1}
=
1 - m c(\theta) =\big( \lambda_i e^{-\im \theta} \big)^m.
\end{align}
Substituting this into \eqref{eq:SCFi_nu_powers_formula} leads immediately to
\begin{align} \label{eq:SCFi_powers}
\big[ \wh{S}^{\rm CF}_i (\theta) \big]^{\nu}
=
\big(\lambda_i e^{-\im \theta} \big)^{(\nu-1)m} \wh{S}^{\rm CF}_i (\theta),
\quad 
\nu \in \mathbb{N}.
\end{align}

Finally, to complete the proof, consider the product of the CF- and F-relaxation Fourier symbols.
Using the same rank-1 exploit as above and using \eqref{eq:powers_function} gives
\begin{align}
\wh{S}^{\rm CF}_i (\theta) \wh{S}^{\rm F}_i (\theta)
&=
\Big(
c(\theta)
\big[ \wh{A}_{0,i}(\theta) \big]^{-1}
\bm{1}
\bm{1}^\top
\Big[
I - \wh{A}_{0,i}(\theta)
\Big]
\Big)
\Big(
c(\theta)
\big[ \wh{A}_{0,i}(\theta) \big]^{-1}
\bm{1}
\bm{1}^\top \Big),
\\
&=
c(\theta)
\big[ \wh{A}_{0,i}(\theta) \big]^{-1}
\bm{1}
\Big(
\bm{1}^\top
\Big[
I - \wh{A}_{0,i}(\theta)
\Big]
c(\theta)
\big[ \wh{A}_{0,i}(\theta) \big]^{-1}
\bm{1}
\Big)
\bm{1}^\top,
\\
&=
\big(\lambda_i e^{-\im \theta} \big)^{m} \, \wh{S}^{\rm F}_i (\theta).
\end{align}
Combining this result with \eqref{eq:SCFi_powers} leads immediately to the claimed result of \eqref{eq:pre-relax_eigen} that $\big[ \wh{S}^{\rm CF}_i (\theta) \big]^{\nu} \wh{S}^{\rm F}_i (\theta) = \big( \lambda_i e^{- \im \theta} \big)^{m \nu} \, \wh{S}^{\rm F}_i (\theta)$ for $\nu \in \mathbb{N}_0$.
\end{proof}
%

\subsection{Proof of \cref{thm:eigmat-error-prop} (Error propagator Fourier symbol)}
\label{sec:proof:eigmat-error-prop}

\begin{proof}
We begin by forming $\wh{{\cal K}}_i(\theta)$ from \eqref{eq:Ki_eigen_components}, which is the coarse-grid correction component of $\wh{{\cal E}}_i(\theta)$. 
Using the interpolation Fourier symbol \eqref{eq:Pi_eigen}, and exploiting that the Fourier symbol of the coarse-grid operator is a scalar (see \eqref{eq:A0and1i_eigen}), we have
\begin{align}
\wh{{\cal K}}_i(\theta)
=
I_m
-
\frac{1}{m}
\,
\big[\wh{A}_{1,i} (m \theta) \big]^{-1}
\,
\bm{1} \bm{1}^\top
\,
\wh{A}_{0,i} (\theta).
\end{align}
Substituting this into the expression for $\wh{{\cal E}}_i(\theta)$ given in \eqref{eq:Ei_eigen_components}, and using the result from \eqref{eq:pre-relax_eigen} that $\big[ \wh{S}^{\rm CF}_i (\theta) \big]^{\nu} \wh{S}^{\rm F}_i (\theta) = \big( \lambda_i e^{-\im \theta} \big)^{m \nu} \wh{S}^{\rm F}_i (\theta)$ gives 
\begin{align}
\wh{{\cal E}}_i(\theta)
&=
\wh{S}^{\rm F}_i (\theta)
\Big(
I_m
-
\frac{1}{m}
\,
\big[\wh{A}_{1,i} (m \theta) \big]^{-1}
\,
\bm{1} \bm{1}^\top
\,
\wh{A}_{0,i} (\theta)
\Big)
\big[ \wh{S}^{\rm CF}_i (\theta) \big]^{\nu} \wh{S}^{\rm F}_i (\theta),
\\
\label{eq:Ei_eigen_intermediate}
&=
\big( \lambda_i e^{-\im \theta} \big)^{m \nu}
\Big(
\Big[ \wh{S}^{\rm F}_i (\theta) \Big]^2
-
\frac{1}{m}
\,
\big[\wh{A}_{1,i} (m \theta) \big]^{-1}
\,
\Big[ 
\wh{S}^{\rm F}_i (\theta)
\bm{1} \bm{1}^\top
\,
\wh{A}_{0,i} (\theta)
\wh{S}^{\rm F}_i (\theta)
\Big]
\Big).
\end{align}

From \eqref{eq:SFi_eigen}, recall that $\wh{S}^{\rm F}_i (\theta) = c(\theta) 
\big[\wh{A}_{0,i} (\theta) \big]^{-1}
\,
\bm{1} \bm{1}^{\top}$, with $c$ a scalar. 
Using this, and that $\bm{1}^\top \bm{1} = m$, the last term in closed parentheses in \eqref{eq:Ei_eigen_intermediate} can be rewritten as
\begin{align}
\wh{S}^{\rm F}_i (\theta)
\Big(
\bm{1} \bm{1}^\top
\,
\wh{A}_{0,i} (\theta)
\wh{S}^{\rm F}_i (\theta)
\Big)
&=
\wh{S}^{\rm F}_i (\theta)
\Big(
\bm{1} \bm{1}^{\top}
\,
\wh{A}_{0,i} (\theta)
c(\theta) 
\big[ \wh{A}_{0,i} (\theta) \big]^{-1}
\,
\bm{1} \bm{1}^{\top}
\Big), 
\\
&=
\wh{S}^{\rm F}_i (\theta)
\Big(
c(\theta) \bm{1} \bm{1}^\top \bm{1} \bm{1}^\top
\Big),
\\
&=
c(\theta)
c(\theta)
\big[ \wh{A}_{0,i} (\theta) \big]^{-1}
\bm{1} \bm{1}^\top \bm{1} \bm{1}^\top \bm{1} \bm{1}^\top,
\\
&=
m^2 
c(\theta) 
\wh{S}^{\rm F}_i (\theta).
\end{align}
Substituting this result into \eqref{eq:Ei_eigen_intermediate} gives the error propagator Fourier symbol as
\begin{align}
\wh{{\cal E}}_i(\theta)
&=
\big( \lambda_i e^{-\im \theta} \big)^{m \nu}
\Big(
\Big[ \wh{S}^{\rm F}_i (\theta) \Big]^2
-
m c(\theta)
\,
\big[ \wh{A}_{1,i} (m \theta) \big]^{-1}
\,
\wh{S}^{\rm F}_i (\theta)
\Big),
\\
\label{eq:Ei_eigen_intermediate2}
&=
\big( \lambda_i e^{-\im \theta} \big)^{m \nu}
\Big(
1
-
m c(\theta)
\,
\big[ \wh{A}_{1,i} (m \theta) \big]^{-1}
\Big)
\wh{S}^{\rm F}_i (\theta),
\end{align}
where the second equality follows by the idempotence of $\wh{S}^{\rm F}_i (\theta)$ (see \Cref{cor:SF_idempotence}), and then pulling out the common factor of $\wh{S}^{\rm F}_i (\theta)$.

Finally, using the definition of $c(\theta)$ given in \eqref{eq:c_def}, and the expression for $\wh{A}_{1,i} (m \theta)$ given by \eqref{eq:A0and1i_eigen}, we get
\begin{align}
1
-
m c(\theta)
\,
\big[ \wh{A}_{1,i} (m \theta) \big]^{-1}
=
\frac{\wh{A}_{1,i} (m \theta) - m c(\theta)}{\wh{A}_{1,i} (m \theta)}
=
\frac{\lambda_i^m e^{- \im m \theta} - \mu_i e^{- \im m \theta}}{1 - \mu_i e^{- \im m \theta}}
=
\frac{\lambda_i^m - \mu_i }{e^{\im m \theta} - \mu_i }.
\end{align}
Substituting this into \eqref{eq:Ei_eigen_intermediate2} yields the claimed form of $\wh{{\cal E}}_i(\theta)$ given by \eqref{eq:Ei_eigen}.
\end{proof}
%

\section{Eigenvalue estimates}
\label{app:eig-est}

\subsection{Semi-Lagrangian discretizations}
\label{app:eig-est-SL}

In this appendix, we present eigenvalue estimates for the time-stepping operators of semi-Lagrangian discretizations by invoking truncation error estimates from our recent work in Reference~\citenum{DeSterck_etal_2023_SL}.
To this end, consider the matrix ${\cal D}_{p+1}$ from De~Sterck~et~al.\cite{DeSterck_etal_2023_SL}

\begin{definition} \label{def:Dp}
Define the matrix ${\cal D}_{p+1} \in \mathbb{R}^{n_x \times n_x}$ such that $h^{-(p+1)}{\cal D}_{p+1}$ represents a finite-difference rule for approximating the $p+1$st derivative of periodic grid functions.
Let $\bm{v} = \big(v(x_1), \ldots, v(x_{n_x}) \big)^\top \in \mathbb{R}^{n_x}$ denote a vector of a periodic function $v(x)$ evaluated on the spatial mesh. 
Then, if the finite-difference rule is of order $s \in \mathbb{N}$ and $v$ is at least $p+1+s$ times continuously differentiable
\begin{align} \label{eq:Dp+1_def}
\left(
\frac{{\cal D}_{p+1}}{h^{p+1}} \bm{v} 
\right)_i = \left. \frac{\d^{p+1} v}{\d x^{p+1}} \right|_{x_i} + {\cal O}(h^s), 
\quad i \in \{1, \ldots, n_x \}.
\end{align}
Since the mesh points are equispaced, the matrix ${\cal D}_{p+1}$ is circulant.
\end{definition}

Now, consider the following truncation error estimate, which is a simplified version of Lemma 3.1 from De~Sterck~et~al.\cite{DeSterck_etal_2023_SL}
\begin{lemma}[Truncation error estimate of ${\cal S}_{p}^{(\delta t)}$]
\label{lem:trunc-est}
Define $\bm{u}(t) \in \mathbb{R}^{n_x}$ as the vector composed of the exact solution of the advection equation sampled in space at the mesh points $\bm{x}$ and at time $t$.
Suppose that $p$ is odd.
Then, the local truncation error of the semi-Lagrangian discretization ${\cal S}_{p}^{(\delta t)}$ can be expressed as
\begin{align} \label{eq:SL_trunc}
\bm{u}(t_{n+1}) - {\cal S}_{p}^{(\delta t)} \bm{u}(t_{n})
=
h^{p+1} f_{p+1} \big(\varepsilon^{(\delta t)} \big) 
\frac{{\cal D}_{p+1}}{h^{p+1}}
\bm{u}(t_{n+1})
+ {\cal O}(h^{p+2}),
\end{align}
in which $\varepsilon^{(\delta t)}$ is the mesh-normalized distance from a departure point to its east-neighboring mesh point.
The function $f_{p+1}$ in \eqref{eq:SL_trunc} is the following degree $p+1$ polynomial 
\begin{align} \label{eq:fpoly_def}
f_{p+1}(z) = \frac{1}{(p+1)!} \prod \limits_{q = - \frac{p+1}{2}}^{\frac{p-1}{2}} (q+z).
\end{align}
\end{lemma}

The result of \cref{lem:trunc-est} is now used to develop an eigenvalue estimate for the semi-Lagrangian discretization.
\begin{theorem}[Eigenvalue estimates of ${\cal S}_{p}^{(\delta t)}$ and $\Psi_{\rm ideal}$] \label{thm:eig-est-SL}
Let $s_p^{(\delta t)}(\omega)$ denote the eigenvalue of the semi-Lagrangian time-stepping operator ${\cal S}_{p}^{(\delta t)}$ associated with the spatial Fourier mode $\bm{\chi}(\omega)$ from \eqref{eq:Fouier_modes_space_def}.
Analogously, let $\big[ s_p^{(\delta t)}(\omega) \big]^m$ denote the eigenvalue of the associated ideal coarse-grid operator, $\Psi_{\rm ideal} = \prod_{k = 0}^{m-1} {\cal S}_{p}^{(\delta t)}$.
Suppose that $p$ is odd.
Then, we have the following estimates for eigenvalues associated with asymptotically smooth Fourier modes (i.e., $\omega = {\cal O}(h)$)
\begin{align} \label{eq:SL_eig_est}
s_p^{(\delta t)} (\omega)
&= 
\exp \left( - \frac{\im \omega \alpha \delta t}{h} \right)
\Big[ 1 
- 
f_{p+1}\big( \varepsilon^{(\delta t)} \big) d_{p+1} (\omega) 
+ {\cal O}(\omega^{p+2})
 \Big],
\\
\label{eq:SL_ideal_eig_est}
\big[ s_p^{(\delta t)}(\omega) \big]^m
&= 
\exp \left( -  \frac{\im \omega \alpha m \delta t}{h} \right)
\Big[ 1 
- 
m
f_{p+1}\big( \varepsilon^{(\delta t)} \big) d_{p+1} (\omega) 
+ {\cal O}(\omega^{p+2})
 \Big],
\end{align}
with $d_{p+1}(\omega)$ the eigenvalue of ${\cal D}_{p+1}$ (see \cref{lem:eig-est-D} for further details).
\end{theorem}
\begin{proof}
Observe the following function is an exact solution of the advection equation $\frac{\partial u}{\partial t} + \alpha \frac{\partial u}{\partial x}
= 0$,
\begin{align} 
\label{eq:u_exact_exponential}
u(x, t) 
= 
\exp \left( \frac{\im \omega}{h} \big(x -  \alpha t \big) \right)
=
\exp \left( - \frac{\im \omega \alpha t }{h} \right) \chi(\omega).
\end{align}
Thus, the vector from \cref{lem:trunc-est} that is an exact PDE solution sampled at the spatial mesh points can be written as $\bm{u}(t) = \exp \left( - \frac{\im \omega  \alpha  t}{h} \right) 
\bm{\chi}(\omega) $. 
Substituting this vector into the truncation error estimate 
 \eqref{eq:SL_trunc} for ${\cal S}_{p}^{(\delta t)}$ gives the following relation for asymptotically smooth Fourier modes:
\begin{align} \label{eq:s_eig_proof_aux}
\begin{split}
&\exp \left( - \frac{\im \omega  \alpha  (t_n + \delta t)}{h} \right) 
\bm{\chi}(\omega) 
- 
\exp \left( - \frac{\im \omega \alpha t_n }{h} \right)  
{\cal S}_{p}^{(\delta t)} \bm{\chi}(\omega)
=
\exp \left( - \frac{\im \omega \alpha (t_n + \delta t)}{h} \right)
\left[
f_{p+1} \big(\varepsilon^{(\delta t)} \big) 
{\cal D}_{p+1}
+ {\cal O}(\omega^{p+2}) \right] \bm{\chi}(\omega).
\end{split}
\end{align}
Here, we have re-written the ${\cal O}(h^{p+2})$ term from \eqref{eq:SL_trunc} as ${\cal O}(\omega^{p+2}) \bm{u}(t_{n+1})$, since ${\cal O}(h^{p+2}) = {\cal O}(\omega^{p+2}) = {\cal O}(\omega^{p+2}) \bm{u}(t_{n+1})$.
Rearranging this equation for ${\cal S}_{p}^{(\delta t)} \bm{\chi}(\omega)$ gives
\begin{align} \label{eq:s_eig_proof_aux2}
{\cal S}_{p}^{(\delta t)} \bm{\chi}(\omega)
&= 
\exp \left( - \frac{\im \omega \alpha \delta t}{h} \right)
\Big[ 1 
-
f_{p+1}\big( \varepsilon^{(\delta t)} \big) {\cal D}_{p+1}
+ {\cal O}(\omega^{p+2})
 \Big]
 \bm{\chi}(\omega).
 \end{align}
Applying $ {\cal D}_{p+1} \bm{\chi}(\omega)  = d_{p+1}(\omega) \bm{\chi}(\omega)$ in \eqref{eq:s_eig_proof_aux2} gives the claim \eqref{eq:SL_eig_est}.

The eigenvalue estimate \eqref{eq:SL_ideal_eig_est} for the associated ideal coarse-grid operator follows by applying ${\cal S}_{p}^{(\delta t)}$ to both sides of \eqref{eq:s_eig_proof_aux2} a further $m-1$ times,  each time invoking on the right-hand side the eigenvalue estimate for ${\cal S}_{p}^{(\delta t)}$, then collecting terms of size ${\cal O}(\omega^{p+2})$ together.
\end{proof}

The following result estimates the eigenvalues of the matrix ${\cal D}_{p+1}$ given in \cref{def:Dp}.
\begin{lemma}[Eigenvalue estimate of ${\cal D}_{p+1}$] \label{lem:eig-est-D}
Let $d_{p+1} ( \omega )$ denote the eigenvalue of ${\cal D}_{p+1}$ from \cref{def:Dp} associated with the spatial Fourier mode $\bm{\chi}(\omega)$.
Suppose that $p$ is odd.
Then, we have the following estimate for eigenvalues associated with asymptotically smooth Fourier modes
\begin{align} \label{eq:eig_est-D} 
d_{p+1}( \omega ) 
= 
(\im \omega)^{p+1} + {\cal O}(\omega^{p+1+s})
=
(-1)^{\frac{p+1}{2}} \omega^{p+1} \big[ 1 + {\cal O}(\omega^s) \big],
\end{align}
with $s$ the order of approximation of ${\cal D}_{p+1}$.
\end{lemma}
\begin{proof}
This follows by taking $\bm{v}$ in \eqref{eq:Dp+1_def} as the spatial Fourier mode $\bm{\chi}(\omega)$ from \eqref{eq:Fouier_modes_space_def}, then using that $\omega = {\cal O}(h)$ for asymptotically smooth modes.
\end{proof}

\subsection{Modified semi-Lagrangian discretization}
\label{app:eig-est-mod-SL}


\begin{lemma} \label{lem:eig-est-SL-mod}
Suppose that $p$ is odd.
Consider the modified semi-Lagrangian coarse-grid operator from (3.20) of De~Sterck~et~al.,\cite{DeSterck_etal_2023_SL} $\Psi 
= 
\Big( I - 
\big[ 
f_{p+1} \big( \varepsilon^{(m \delta t)} \big) 
-
m f_{p+1} \big( \varepsilon^{(\delta t)} \big)
\big] {\cal D}_{p+1} \Big)^{-1} {\cal S}_{p}^{(m \delta t)}$.
Let $\psi(\omega)$ denote the eigenvalue of $\Psi$ associated with spatial Fourier mode $\bm{\chi}(\omega)$ from \eqref{eq:Fouier_modes_space_def}.
Then, we have the following estimate for eigenvalues of $\Psi$ associated with asymptotically smooth Fourier modes (i.e., $\omega = {\cal O}(h)$)
\begin{align} \label{eq:eig-est-SL-mod}
\psi(\omega) = \exp \left( - \frac{\im \omega \alpha m \delta t}{h} \right)
\Big(
1 - m f_{p+1} \big( \varepsilon^{(\delta t)} \big) d_{p+1} (\omega)
+
{\cal O}(\omega^{p+2})
\Big),
\end{align}
with $d_{p+1}(\omega)$ the eigenvalue of ${\cal D}_{p+1}$ (see \cref{lem:eig-est-D} for further details).
\end{lemma}

\begin{proof}
Applying $\Psi$ to an asymptotically smooth Fourier mode and then invoking the geometric expansion $\frac{1}{1 - \varepsilon} = 1 + \varepsilon+ {\cal O}(\varepsilon^2)$ gives
\begin{align}
\Psi \bm{\chi}(\omega)
&=
\frac{s_p^{(m \delta t)}}{
1 - \big[ 
f_{p+1} \big( \varepsilon^{(m \delta t)} \big) 
-
m f_{p+1} \big( \varepsilon^{(\delta t)} \big)
\big] d_{p+1} (\omega)} 
\bm{\chi}(\omega),
\\
\label{eq:eig-est-SL-mod-aux}
&=
\Big( 1 + \big[ 
f_{p+1} \big( \varepsilon^{(m \delta t)} \big) 
-
m f_{p+1} \big( \varepsilon^{(\delta t)} \big)
\big] d_{p+1} (\omega)
+
{\cal O}\big( [d_{p+1} (\omega) ]^2 \big)
\Big) 
s_p^{(m \delta t)}
\bm{\chi}(\omega).
\end{align}
The claimed result \eqref{eq:eig-est-SL-mod} follows by substituting into \eqref{eq:eig-est-SL-mod-aux} the estimate for $s_p^{(m \delta t)}$ that arises from \eqref{eq:SL_eig_est}, applying the fact that $d_{p+1} (\omega) = {\cal O}(\omega^{p+1})$ from \eqref{eq:eig_est-D}, and then collecting together terms of size ${\cal O}(\omega^{p+2})$.
\end{proof}

\section{Remainder of the proof of \cref{THM:RHO-LWR-BOUND-SL} (Convergence factor lower bound)}
\label{app:conv-fac-constant}

Here, we continue with the proof of \cref{THM:RHO-LWR-BOUND-SL}, showing that $
\widecheck{\rho}_p \big( \varepsilon^{(\delta t)} \big)
>
1
$ 
when
$
\varepsilon^{(\delta t)} \in \Upsilon_m \cap \left( \frac{2}{3 m}, 1 - \frac{2}{3 m} \right)
$.

\begin{proof}
For notational simplicity, write $\varepsilon \equiv \varepsilon^{(\delta t)}$ throughout the proof.
This proof relies on details about $f_{p+1} (\varepsilon)$ and $f_{p+1} (m \varepsilon - \lfloor m \varepsilon \rfloor)$ given in Lems. 4.5 \& B.1 of Krzysik,\cite{KrzysikThesis2021} as well as \cref{lem:fDelta} below.
We prove the claim only for the case of $\frac{p+1}{2}$ even corresponding to $f_{p+1}$ being non-negative. The proof extends to $\frac{p+1}{2}$ odd by following through the sign changes on $f_{p+1}$.

See \cref{fig:fp_plots} for examples of the functions $f_{p+1} (\varepsilon)$ and $f_{p+1} (m \varepsilon - \lfloor m \varepsilon \rfloor)$.
Specific properties of $f_{p+1} (\varepsilon)$ that we use are that it is symmetric over $\varepsilon \in [0, 1]$, and that it is monotonically increasing on $\varepsilon \in [0, \frac{1}{2}]$ to its only critical point, which is a global maximum at $\varepsilon = \frac{1}{2}$.
Additionally, note that $f_{p+1} (m \varepsilon - \lfloor m \varepsilon \rfloor)$ is periodic over $\varepsilon \in [0,1]$ due to its argument being $\frac{1}{m}$-periodic in $\varepsilon$. 
This function has global minima at $\varepsilon = \frac{k}{m}, \, k = 0, \ldots, m$ and global maxima at $\varepsilon = \frac{k}{m} + \frac{1}{2m}, \, k = 0, \ldots, m-1$, and between these global minima and maxima the function changes monotonically.
%

\begin{figure}[b!]
\centerline{
\includegraphics[scale=0.425]{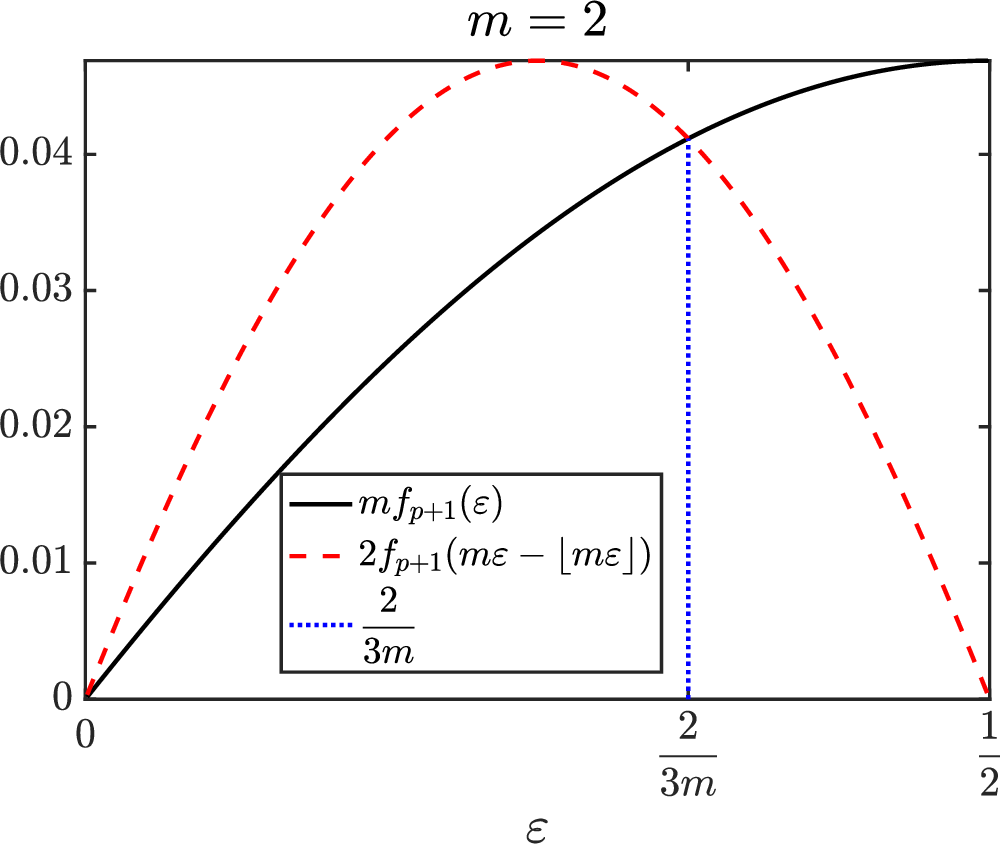}
\quad\quad\quad
\includegraphics[scale=0.425]{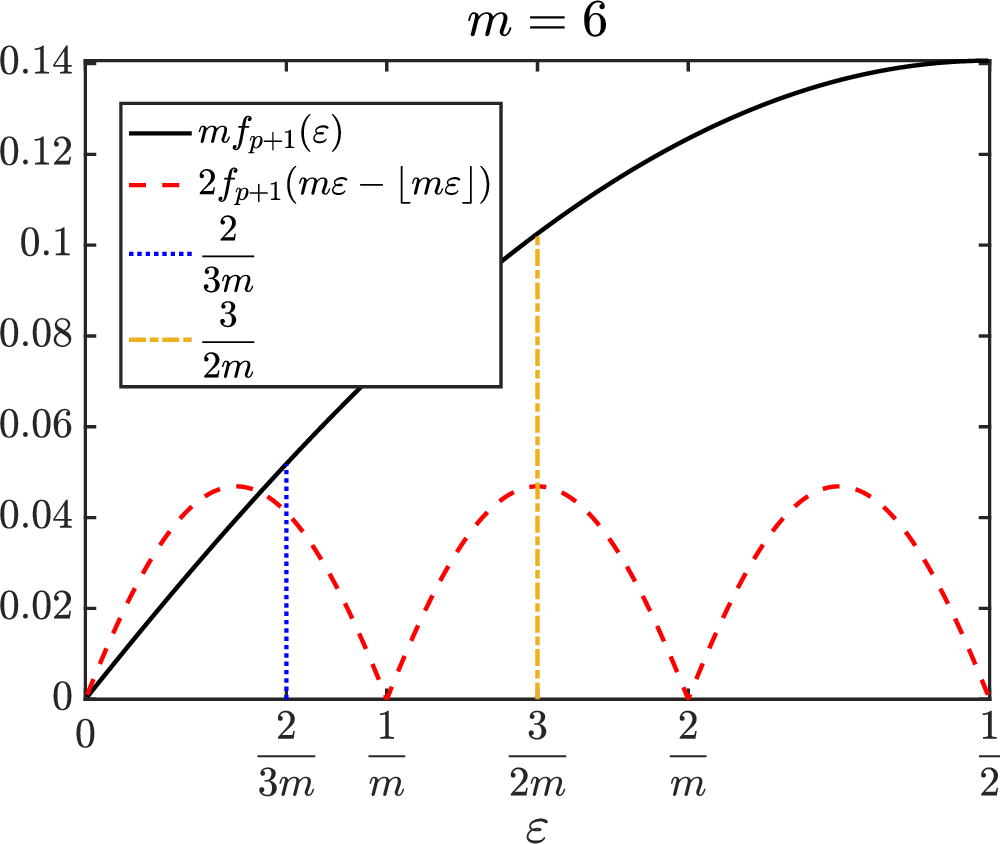}
}
\caption{
Examples of the functions analyzed in \cref{app:conv-fac-constant} for $p = 3$, with $m = 2$ (left), and $m = 6$ (right).
\label{fig:fp_plots}
}
\end{figure}

Define the function
\begin{align}
\widecheck{f}(\varepsilon) := m f_{p+1} (\varepsilon) - 2 f_{p+1} (m \varepsilon - \lfloor m \varepsilon \rfloor),
\end{align}
motivated by the fact that $\widecheck{\rho}(\varepsilon) > 1 \Longleftrightarrow \widecheck{f}(\varepsilon) > 0$ when $\varepsilon \in \Upsilon_m$.
Since $\widecheck{f}(\varepsilon)$ is symmetric on $\varepsilon \in [0, 1]$ we need only consider $\varepsilon \in [0, \frac{1}{2}]$.
The proof works by showing that $m f_{p+1} (\varepsilon) > 2 f_{p+1} (m \varepsilon - \lfloor m \varepsilon \rfloor)$ for all $\varepsilon \in {\cal I} := (\frac{2}{3m}, \frac{1}{2}]$. For $m > 2$ we do this by considering separately the two sub-intervals ${\cal I}_{1} := (\frac{2}{3m}, \frac{3}{2m}]$ and ${\cal I}_2 := [\frac{3}{2m}, \frac{1}{2}]$.
The special case of $m = 2$ is described at the end.

Suppose that $m > 2$ and consider $\varepsilon \in {\cal I}_2$.
From \cref{lem:fDelta} we have $m f_{p+1} (\varepsilon) > 2 f_{p+1} (m \varepsilon - \lfloor m \varepsilon \rfloor)$ at the left-hand boundary of ${\cal I}_2$, and recall that that $2 f_{p+1} (m \varepsilon - \lfloor m \varepsilon \rfloor)$ has a global maximum at this left-hand boundary of ${\cal I}_2$.
Then, since $m f_{p+1} (\varepsilon)$ is monotonically increasing over $\varepsilon \in {\cal I}_2$ it follows that $m f_{p+1} (\varepsilon) > 2 f_{p+1} (m \varepsilon - \lfloor m \varepsilon \rfloor)$ for all $\varepsilon \in {\cal I}_2$.

Suppose that $m > 2$ and consider $\varepsilon \in {\cal I}_1$. 
Recall that $2 f_{p+1} (m \varepsilon - \lfloor m \varepsilon \rfloor)$ is monotonically decreasing on a sub-interval of ${\cal I}_1$ given by $\varepsilon \in [\frac{2}{3m}, \frac{1}{m}]$, while it is monotonically increasing on the remainder of ${\cal I}_1$, $\varepsilon \in [\frac{1}{m}, \frac{3}{2m}]$.
Then, since $m f_{p+1} (\varepsilon) > 2 f_{p+1} (m \varepsilon - \lfloor m \varepsilon \rfloor)$ at both end points of ${\cal I}_1$ (see \cref{lem:fDelta}), it follows that $m f_{p+1} (\varepsilon) > 2 f_{p+1} (m \varepsilon - \lfloor m \varepsilon \rfloor)$ for all $\varepsilon \in {\cal I}_1$ because $m f_{p+1} (\varepsilon)$ is monotonically increasing on ${\cal I}_1$ and it is concave up.

Finally, consider the $m = 2$ case. From \cref{lem:fDelta}, we have that $m f_{p+1} (\varepsilon) = 2 f_{p+1} (m \varepsilon - \lfloor m \varepsilon \rfloor)$ at $\varepsilon = \frac{2}{3m}$. 
Then, recall that $m f_{p+1} (\varepsilon)$ is monotonically increasing on $\varepsilon \in [\frac{2}{3m}, \frac{1}{2}]$ and that $2 f_{p+1} (m \varepsilon - \lfloor m \varepsilon \rfloor)$ is monotonically decreasing on $\varepsilon \in [\frac{2}{3m}, \frac{1}{2}]$.
Therefore, it follows that $m f_{p+1} (\varepsilon) > 2 f_{p+1} (m \varepsilon - \lfloor m \varepsilon \rfloor)$ for all $\varepsilon \in (\frac{2}{3m}, \frac{1}{2}]$.
\end{proof}

\begin{lemma} \label{lem:fDelta}
Let $p \in \mathbb{N}$ such that $\frac{p+1}{2}$ is odd, let and $m \in \mathbb{N} \setminus \{1\}$.
Let $f_{p+1}(z)$ be the polynomial in \eqref{eq:fpoly_def}.
Then,
\begin{align} \label{eq:fDelta-inequality}
m f_{p+1} (\varepsilon) \geq 2 f_{p+1} (m \varepsilon - \lfloor m \varepsilon \rfloor)
\end{align}
holds at $\varepsilon = \frac{3}{2m}$ for $m \geq 3$ without equality (Claim I), and holds at $\varepsilon = \frac{2}{3m}$ with equality for $m = 2$, and without equality for $m \geq 3$ (Claim II).
\end{lemma}

\begin{proof}
It is useful to consider an alternative expression for $f_{p+1}(z)$ than \eqref{eq:fpoly_def}. From (B.8) of Krzysik,\cite{KrzysikThesis2021} for $\frac{p+1}{2}$ even, we have
\begin{align} \label{eq:fDelta_squares}
f_{p+1}(z) = \frac{z \left( \frac{p+1}{2} - z \right)}{(p+1)!} \prod \limits_{q = 1}^{\frac{p-1}{2}} \big( q^2 - z^2)
\end{align}

Consider Claim I.
If $\varepsilon = \frac{3}{2m}$, then $m \varepsilon - \lfloor m \varepsilon \rfloor = \frac{1}{2}$, and \eqref{eq:fDelta-inequality} reduces to $m f_{p+1} \big(\frac{3}{2m} \big) \geq 2 f_{p+1} \big( \frac{1}{2} \big)$.
Plugging into \eqref{eq:fDelta_squares} gives
\begin{align} 
\label{eq:fDelta_aux1}
m f_{p+1} \left(\frac{3}{2 m} \right) 
&= 
\frac{\frac{3}{2} \left( \frac{p+1}{2} - \frac{3}{2 m} \right)}{(p+1)!} \prod \limits_{q = 1}^{\frac{p-1}{2}} \left[ q^2 - \left( \frac{3}{2 m} \right)^2 \right],
\\
\label{eq:fDelta_aux2}
2 f_{p+1} \left(\frac{1}{2} \right) 
&= 
\frac{1 \left( \frac{p+1}{2} - \frac{1}{2} \right)}{(p+1)!} \prod \limits_{q = 1}^{\frac{p-1}{2}} \left[ q^2 - \left( \frac{1}{2} \right)^2 \right].
\end{align}
Now compare \eqref{eq:fDelta_aux1} and \eqref{eq:fDelta_aux2} factor by factor.
First, $\frac{3}{2} > 1$. 
Next, note that for $m \geq 3$, we have $\frac{3}{2 m} \leq \frac{1}{2}$.
Thus, 
$
\frac{p+1}{2} - \frac{3}{2 m} \geq  \frac{p+1}{2} - \frac{1}{2} > 0
$,
and, finally, \\
$q^2 - \left( \frac{3}{2 m} \right)^2 
\geq 
q^2 - \left( \frac{1}{2} \right)^2
>
0
$.
Clearly \eqref{eq:fDelta_aux1} is greater than \eqref{eq:fDelta_aux2} when $m \geq 3$.

Now consider Claim II. If $\varepsilon = \frac{2}{3m}$, then $m \varepsilon - \lfloor m \varepsilon \rfloor = \frac{2}{3}$, and \eqref{eq:fDelta-inequality} reduces to 
$
m f_{p+1} \big(\frac{2}{3m} \big) \geq 2 f_{p+1} \big( \frac{2}{3} \big)
$.
Consider first the $m = 2$ case, for which \eqref{eq:fDelta-inequality} reduces to 
$
2 f_{p+1} \big( \frac{1}{3}) \geq 2 f_{p+1} \big( \frac{2}{3} \big)
$.
From Lemma B.1 of Krzysik,\cite{KrzysikThesis2021} $f_{p+1}(z)$ is symmetric about $z = \frac{1}{2}$, therefore the statement
$
2 f_{p+1} \big( \frac{1}{3}) \geq 2 f_{p+1} \big( \frac{2}{3} \big)
$
must hold with equality.
Now consider $m \geq 3$, for which we again make use of \eqref{eq:fDelta_squares}.
Plugging $\varepsilon = \frac{2}{3m}$ into \eqref{eq:fDelta_squares} we have
\begin{align} 
\label{eq:fDelta_aux3}
m f_{p+1} \left( \frac{2}{3m} \right) 
&= 
\frac{
\frac{2}{3}  
\left( \frac{p+1}{2} - \frac{2}{3m} \right)
}{(p+1)!}
\prod \limits_{q = 1}^{\frac{p-1}{2}} \left[ q^2 - \left( \frac{2}{3m} \right)^2 \right].
\end{align}
Notice that all of the individual factors in \eqref{eq:fDelta_aux3} are positive for $m \geq 2$, and, furthermore, that they are all strictly increasing for $m \geq 2$.
Since we have $m f_{p+1} \left( \frac{2}{3m} \right) = 2 f_{p+1} \big( \frac{2}{3} \big)$ when $m = 2$, it must hold that $m f_{p+1} \left( \frac{2}{3m} \right) > 2 f_{p+1} \big( \frac{2}{3} \big)$ for $m > 2$ due to $m f_{p+1} \left( \frac{2}{3m} \right)$ strictly increasing for $m \geq 2$.
\end{proof}

\end{document}